\documentclass[a4paper,11pt,reqno]{amsart}
\usepackage[latin1]{inputenc}
\usepackage[english]{babel}
\usepackage{amsmath, amsthm, amssymb, amsopn, amsfonts, amstext, stmaryrd, enumerate, mathtools, mathrsfs, hyperref, enumitem, url}
\usepackage[margin=0.75in]{geometry}
\usepackage[final]{showkeys}
\numberwithin{equation}{section}
\usepackage{dsfont}
\usepackage{amsbsy}
\usepackage{cite}

\allowdisplaybreaks

\usepackage[usenames,dvipsnames]{color}

\hypersetup{
	colorlinks=true,
	linkcolor=blue,
	citecolor=blue,
	urlcolor=green!50!black,
	linktoc=all
}

\usepackage{colortbl}
\usepackage{tikz}

\newcommand{\C}{\mathcal{C}}

\newcommand{\N}{\mathbb{N}}

\newcommand{\R}{\mathbb{R}}

\renewcommand{\S}{\mathbb{S}}

\newcommand{\Z}{\mathbb{Z}}

\newcommand{\loc}{{\rm loc}}

\newcommand{\sgn}{{\mbox{\normalfont sgn}}}
\newcommand{\dist}{{\mbox{\normalfont dist}}}
\newcommand{\PV}{\mbox{\normalfont P.V.}}
\newcommand{\Haus}{\mathcal{H}}

\newcommand{\Lfrac}{{(-\Delta)^{\frac{1}{2}}}}

\renewcommand{\epsilon} {\varepsilon}

\DeclareMathOperator{\supp}{supp}

\DeclareMathOperator{\CAP}{Cap}
\DeclareMathOperator{\spn}{span}

\def\XXint#1#2#3{{\setbox0=\hbox{$#1{#2#3}{\int}$ }
\vcenter{\hbox{$#2#3$ }}\kern-.6\wd0}}

\theoremstyle{plain}
\newtheorem{definition}{Definition}[section]
\newtheorem{theorem}[definition]{Theorem}
\newtheorem{proposition}[definition]{Proposition}
\newtheorem{lemma}[definition]{Lemma}

\theoremstyle{definition}
\newtheorem{remark}[definition]{Remark}

\renewcommand{\le}{\leqslant}
\renewcommand{\leq}{\leqslant}
\renewcommand{\ge}{\geqslant}
\renewcommand{\geq}{\geqslant}

\begin{document}

\title[Blowing-up solutions for a nonlocal Liouville type equation]{\textbf{Blowing-up solutions for a nonlocal\\Liouville type equation}}

\author[Matteo Cozzi and Antonio J. Fern\'andez]{Matteo Cozzi and Antonio J. Fern\'andez}

\address{
\vspace{0.25cm}
\newline
\textbf{{\small Matteo Cozzi}}
\vspace{0.1cm}
\newline \indent Dipartimento di Matematica ``Federigo Enriques'', Universit\`a degli Studi di Milano, Milan, Italy}
\email{matteo.cozzi@unimi.it}

\address{
\vspace{-0.25cm}
\newline
\textbf{{\small Antonio J. Fern\'andez}} 
\vspace{0.1cm}
\newline \indent Instituto de Ciencias Matem\'aticas, Consejo Superior de Investigaciones Cient\'ificas, Madrid, Spain}
\email{antonio.fernandez@icmat.es}


\maketitle

\begin{abstract}
We consider the nonlocal Liouville type equation
$$
(-\Delta)^{\frac{1}{2}} u = \varepsilon \kappa(x) e^u, \quad u > 0, \quad \textup{in } I, \qquad u = 0, \quad \textup{in } \R \setminus I,
$$
where $I$ is a union of $d \geq 2$ disjoint bounded intervals, $\kappa$ is a smooth bounded function with positive infimum and~$\varepsilon > 0$ is a small parameter. For any integer~$1 \leq m \leq d$, we construct a family of solutions~$(u_\varepsilon)_{\varepsilon}$ which blow up at~$m$ interior distinct points of $I$ and for which~$\varepsilon \int_I \kappa e^{u_\varepsilon} \, \rightarrow 2 m \pi$, as~$\varepsilon \to 0$. Moreover, we show that, when~$d = 2$ and~$m$ is suitably large, no such construction is possible.

\bigbreak

\smallbreak
\medbreak

\end{abstract}


\section{Introduction and main results} \label{Introduction and main results}

\noindent 
Let~$I$ be the union of a finite number~$d \geq 2$ of bounded open intervals with disjoint closure, that is
\begin{equation} \label{I}
I:= \bigcup_{k = 1}^{d} I_k, \quad \textup{ with } \quad I_k := (a_k,b_k) \quad \textup{ and }\quad - \infty < a_1 < b_1 < \cdots < a_d < b_d < + \infty.
\end{equation}
For~$\epsilon \in (0,1)$ and~$\kappa \in C^{2}(\overline{I})$, with~$\inf_{I}\kappa > 0$, we consider the Dirichlet problem
\[ \label{main-problem} \tag{$P_{\varepsilon}$}
\left\{
\begin{aligned}
(-\Delta)^{\frac12} u & = \varepsilon \kappa(x) e^u, \quad && \textup{ in } I, \\
u & > 0, && \textup{ in } I,\\
u & = 0, && \textup{ in } \R \setminus I.
\end{aligned}
\right.
\]
Here,~$(-\Delta)^{\frac{1}{2}}$ is the half-Laplace operator, defined for a bounded and sufficiently smooth function~$u$ as
$$
(-\Delta)^{\frac{1}{2}} u(x) := \frac{1}{\pi} \, \PV \int_{\R} \frac{u(x) - u(x + z)}{z^2} \, dz = \frac{1}{\pi} \, \lim_{\delta \to 0^{+}} \int_{\R \setminus (-\delta, \delta)} \frac{u(x) - u(x + z)}{z^2} \, dz.
$$
The main aim of the present paper is to construct solutions to~\eqref{main-problem} for small values of~$\epsilon$ which become larger and larger and eventually blow up at interior points of~$I$ as~$\epsilon \to 0$.

\medbreak 
The first motivation for considering these types of problems comes from Riemannian Geometry. Let~$(\Sigma,g_0)$ be a Riemannian surface with boundary. A classical problem is the prescription of the Gaussian and geodesic curvatures on~$\Sigma$ and~$\partial\Sigma$ under a conformal change of the metric. In other words, to study if, given~$K_1: \Sigma \to \R$ and~$\kappa_1: \partial \Sigma \to \R$, there exists a conformal metric~$g_1 = e^{2 U} g_0$ on~$\Sigma$ such that~$K_1$ is the Gaussian curvature of~$\Sigma$ and~$\kappa_1$ is the geodesic curvature of~$\partial \Sigma$, both with respect to new metric~$g_1$. This geometric problem can be reformulated in a PDEs framework. Indeed, it is equivalent to find a function~$U$ solving the boundary value problem 
\begin{equation} \label{geometry1}
\left\{
\begin{aligned}
-\Delta_{g_0} U + K_0(x) & =  K_1(x) e^{2U}, \quad && \textup{ in } \Sigma,\\
\partial_{\nu} U + \kappa_0(x) & = \kappa_1(x) e^{U}, \quad && \textup{ on } \partial \Sigma,
\end{aligned}
\right.
\end{equation}
where~$\Delta_{g_0}$ is the Laplace Beltrami operator associated to the metric~$g_0$,~$\nu$ is the outward normal vector to~$\partial \Sigma$,~$K_0$ is the Gaussian curvature of~$\Sigma$ under the metric~$g_0$, and~$\kappa_0$ is geodesic curvature of~$\partial \Sigma$ under the metric~$g_0$. Note that the interested reader can find an exhaustive list of references concerning~\eqref{geometry1} in the introductions of~\cite{CB-R18,B-M-P21,J-LS-M-R22, LS-M-R22}. When~$\Sigma = \R^2_{+} := \{(x,y) \in \R^2: y > 0\}$ and~$g_0$ is the standard Euclidean metric, problem~\eqref{geometry1} is reduced to
\begin{equation} \label{geometry2}
\left\{
\begin{aligned}
-\Delta U & = K_1(x) e^{2U}, \quad && \textup{ in } \R^2_{+},\\
-\partial_{y} U & = \kappa_1(x) e^{U}, \quad && \textup{ on } \partial \R^2_{+}.
\end{aligned}
\right.
\end{equation}
For~$K_1$ and~$\kappa_1$ constants, problem~\eqref{geometry2} has been largely studied and its solutions are completely classified (see for instance~\cite{O00,L-Z95,G-M09,Z03}). The cases where~$K_1$ and/or~$\kappa_1$ are not constants are however far from being fully understood.

Another closely related problem is obtained by prescribing the geodesic curvature only in a portion~$I$ of the boundary~$\R \equiv \partial \R^2_+$ and keeping the metric unchanged in its complement. This corresponds to
\begin{equation} \label{geometry3}
\left\{
\begin{aligned}
-\Delta U & = K_1(x) e^{2 U}, \quad && \textup{ in } \R^2_{+},\\
U & = 0, \quad && \textup{ on } (\R \setminus I) \times \{0\},\\
-\partial_{y} U & = \kappa_1(x) e^{U}, \quad && \textup{ on } I \times \{0\}.
\end{aligned}
\right.
\end{equation}
If we now restrict ourselves to taking~$K_1 \equiv 0$, this problem becomes equivalent to~\eqref{main-problem} with~$\kappa_1 := \epsilon \kappa$. Indeed, by understanding the half-Laplacian as a Dirichlet-to-Neumann operator, it is clear that the harmonic extension of any solution to~\eqref{main-problem} is a solution to~\eqref{geometry3}--note that the positivity condition occurring in~\eqref{main-problem} is actually automatically satisfied thanks to the strong maximum principle. Hence, from a geometric point of view, we can rephrase our main aim as follows: to construct a family of conformal metrics with prescribed Gaussian curvature~$K_1 \equiv 0$ in~$\R^2_{+}$ and prescribed geodesic curvature~$\epsilon \kappa(x)$ on~$I$ which remains unchanged on~$\R \setminus I$ and blows up at interior points of~$I$ as~$\epsilon \to 0$.

Problem~\eqref{main-problem} can also be considered as the one-dimensional analogue of the Liouville type equation
\begin{equation} \label{Liouville2d}
\left\{
\begin{aligned}
-\Delta u & = \epsilon^2 \kappa_2(x) e^{u}, \quad && \textup{ in } \Omega,\\
u & > 0, \quad && \textup{ in } \Omega,\\
u & = 0, \quad && \textup{ on } \partial \Omega,
\end{aligned}
\right.
\end{equation}
where~$\Omega \subset \R^2$ is a bounded open set with smooth boundary,~$\epsilon \in (0,1)$, and~$\kappa_2 \in C^2(\overline{\Omega})$ satisfies~$\inf_{\Omega} \kappa_2 > 0$. Notice that the Laplacian in the plane has many common features to the one-dimensional half-Laplacian--think, e.g., to their fundamental solutions and the criticality of the Sobolev embedding. However, while it suffices to study~\eqref{Liouville2d} for connected~$\Omega$, the nonlocality of the half-Laplacian makes it interesting to deal with problem~\eqref{main-problem} in the case of a general (disconnected) set~$I$.
 
The existence and behaviour of solutions for the two-dimensional problem~\eqref{Liouville2d} has been largely studied. Directly related to the issues we address here are the seminal papers~\cite{B-P98, dP-K-M05, E-G-P05}, where the authors construct blowing-up solutions to~\eqref{Liouville2d}. Taking~$\Omega$ to be a general non-simply connected set,~\cite[Theorem 1.1]{dP-K-M05} establishes that, given any $m \in \N$, there exists a family of solutions~$(u_{\epsilon})_{\epsilon}$ to~\eqref{Liouville2d} which blow up (as $\epsilon \to 0$) at~$m$ distinct points in the interior of~$\Omega$ and such that
$$
\lambda_2(\epsilon) := \epsilon^2 \int_{\Omega} \kappa_2(x) e^{u_{\epsilon}(x)} dx \longrightarrow 8m\pi, \quad \mbox{as } \epsilon \to 0.
$$
It is also possible to construct such a family of solutions for some carefully chosen simply connected domains (see~\cite[Section 5]{E-G-P05}). However, this is not in general the case. For instance, for~$\kappa_2 \equiv 1$ and~$\Omega = B_1(0)$, Pohozaev's identity implies (see~\cite[Proposition 7.1]{C-L-M-P92}) that, if~$u_{\epsilon}$ is a solution to~\eqref{Liouville2d}, then 
$$
\lambda_2(\epsilon) =\epsilon^2 \int_{B_1(0)} e^{u_{\epsilon}(x)} dx < 8\pi.
$$
More generally, if~$\Omega$ is a star-shaped domain, Pohozaev's identity yields an upper bound on the possible values of~$\lambda_2(\epsilon)$ for which solutions exist.

As shown in~\cite{DlT-H-M-S-2020}, in the one-dimensional problem~\eqref{main-problem} the role of~$8 \pi$ is played by~$2 \pi$. Indeed, the authors of~\cite{DlT-H-M-S-2020} proved that the one-dimensional mean-field type equation
\begin{equation} \label{meanfieldhalflap}
\left\{
\begin{aligned}
(-\Delta)^{\frac{1}{2}} v & = \lambda \, \frac{e^v}{\int_{-1}^1 e^v dx}, \quad && \mbox{in } (-1, 1),\\
v & > 0, \quad && \mbox{in } (-1, 1), \\
v & = 0, \quad && \mbox{in } \R \setminus (-1, 1),
\end{aligned}
\right.
\end{equation}
admits a solution~$v_\lambda$ if and only if~$\lambda \in (0, 2 \pi)$. Moreover, such a solution~$v_\lambda$ blows up at~$0$ as~$\lambda \to 2 \pi$. As a result, if~$u_\varepsilon$ is a solution to~\eqref{main-problem} with~$I = (-1, 1)$ and~$\kappa \equiv 1$, then necessarily
$$
\lambda(\epsilon) := \epsilon \int_{-1}^1 e^{u_{\epsilon}(x)} dx < 2\pi.
$$

In view of these results, it seems very natural to analyze whether, for~$I$ as in~\eqref{I} and~$m \in \N$,~$m \geq 2$, there exists a family of solutions~$(u_{\epsilon})_{\epsilon}$ to problem~\eqref{main-problem} which blow up at interior points of~$I$ (as~$\epsilon \to 0$) and such that
$$
\lim_{\epsilon \to 0} \epsilon \int_{I} \kappa(x) e^{u_{\epsilon}(x)} dx = 2 m \pi.
$$
In our main result we provide an affirmative answer to this question, under the additional assumption~$m \leq d$. More precisely, we have the following statement.

\begin{theorem} \label{mainTheoremExistence}
For any integer~$m \in [1, d]$, there exist~$\epsilon_{\star} \in (0,1)$ and a family of solutions~$(u_{\epsilon})_{\epsilon}$ to~\eqref{main-problem}, defined for~$\epsilon \in (0,\epsilon_{\star})$, such that
$$
\lim_{\epsilon \to 0} \epsilon \int_{I} \kappa(x) e^{u_{\epsilon}(x)} dx = 2m\pi.
$$
Moreover, given any infinitesimal sequence~$(\epsilon_n)_n \subset (0,\epsilon_\star)$, there exist~$m$ distinct points~$\xi_1, \ldots, \xi_m \in I$ for which, up to a subsequence and for any~$\delta > 0$, the following is true:
\begin{itemize}
\item[$\circ$] $(u_{\epsilon_n})_{\epsilon_n}$ is uniformly bounded in~$\displaystyle I \setminus \bigcup_{j=1}^m (\xi_j-\delta,\, \xi_j+\delta)$;
\item[$\circ$] $\displaystyle \sup_{(\xi_j-\delta,\, \xi_j+\delta)} u_{\epsilon_n} \to \infty$ as~$n \to \infty$.
\end{itemize}
\end{theorem}

We stress that the solutions $u_{\epsilon}$ constructed in Theorem~\ref{mainTheoremExistence} are classical, in the sense that it is as regular as the function~$\kappa$ allows--namely, if~$\kappa$ is of class~$C^{\ell}(\overline{I})$, $\ell \geq 2$, then~$u_{\epsilon} \in H^{\frac12}(\R) \cap C^{\,0,\frac12}(\R) \cap C^{\,\ell+1}(I)$--and that~\eqref{main-problem} holds pointwise.

\medbreak
Our proof of Theorem~\ref{mainTheoremExistence} is perturbative and based on a Lyapunov-Schmidt reduction inspired by the one pioneered in~\cite{dP-K-M05}. This approach allows for a rather precise description of the blow up of the family of solutions~$(u_\epsilon)_{\epsilon}$ at the points~$\xi_j$'s. Indeed, focusing for simplicity on the case~$\kappa \equiv 1$, the solution $u = u_\epsilon$ will be of the form~$u = \mathscr{U} + \psi$, with
\begin{equation} \label{mathscrUdefintro}
\mathscr{U}(x) := \sum_{j = 1}^m \left( \log \Big( {\frac{2 \mu_j}{\mu_j^2 \epsilon^2 + (x - \xi_j)^2}} \Big) + H_j(x) \right),
\end{equation}
for some appropriately chosen parameters~$\mu_1, \ldots, \mu_m \in (0, +\infty)$ and suitable corrector terms~$H_j$--for more details, see the beginning of Section~\ref{The approximate solution} and in particular formulas~\eqref{ujdef} and~\eqref{mu_j}. The remainder~$\psi$ being small in~$\epsilon$. 

Problem~\eqref{main-problem} will be addressed by linearizing around~$\mathscr{U}$ (in suitable expanded variables) and solving the corresponding problem for~$\psi$. It turns out that this is possible if the points~$\xi_j$'s are chosen in a suitable way. Indeed, letting~$\Gamma$ be the fundamental solution of~$(-\Delta)^{\frac{1}{2}}$,~$G$ its Green function in the set~$I$, and~$H(x, z) := G(x, z) - \Gamma(x - z)$ its regular part (see Section~\ref{The approximate solution} for more precise definitions), the vector~$\xi = (\xi_1, \ldots, \xi_m)$ will (necessarily) be a critical point of a function~$F_\epsilon: I^m \to \R$ satisfying
$$
F_\epsilon(\xi) \approx - 2 \pi m (1 + \log \epsilon) + \pi \, \Xi(\xi), \quad \mbox{ with } \quad \Xi(\xi) := - \sum_{j = 1}^m \Big( {H(\xi_j, \xi_j) + \sum_{i \ne j} G(\xi_j, \xi_i)} \Big),
$$
for~$\epsilon$ small. Our problem is thus further reduced to finding a critical point for~$\Xi$. Such a critical point (actually, a local minimizer) is found at the expense of requiring~$m \le d$ and choosing each~$\xi_j$ in a different interval~$I_k$. 

We note in passing that, even though our construction is carried out for~$d \ge 2$, the same argument works also in the case of a single interval~$I = I_1$ and gives quantitative information on the blow up rate in~\cite[Theorem~1.2]{DlT-H-M-S-2020} (in this case,~$- \Xi$ agrees with the explicit Robin function~$H(\xi_1, \xi_1)$ of~$I$, which has its maximum at the midpoint of~$I$). Let us also mention here \cite{DlT-M-P-2020}, where the authors analyze the somehow related one-dimensional nonlocal sinh-Poisson equation in a single interval using (as we do) a Lyapunov-Schmidt reduction. 

\medbreak
Theorem~\ref{mainTheoremExistence} leaves a few questions open, more preeminently whether solutions with multiple blow up points in the same interval exist or not. The following result goes in the direction of a negative answer, showing that, for~$\kappa \equiv 1$, the construction which we just described cannot be carried out in the case of two intervals ($d = 2$) and a sufficiently large number~$m$ of blow up points.

\begin{proposition} \label{nonExistence}
Let~$\delta_0 \in (0,1/10]$. There exist~$b_\star \ge 1$ and~$m_{\star} > 2$ for which the following holds true: for all~$m \geq m_{\star}$ and all~$\delta \in (0,1/10]$, there exists~$\epsilon_0 \in (0,1)$ such that, if~$\epsilon \in (0,\epsilon_0)$, then, setting~$I_{b_\star} := (-b_{\star}-1,-b_{\star}) \cup (b_{\star}, b_{\star}+1)$, problem
$$
\left\{
\begin{aligned}
(-\Delta)^{\frac12} u & = \varepsilon e^u, \quad && \textup{ in } I_{b_{\star}}, \\
u & > 0, && \textup{ in } I_{b_{\star}},\\
u & = 0, && \textup{ in } \R \setminus I_{b_{\star}},
\end{aligned}
\right.
$$
has no solution of the form~$u = \mathscr{U} + \psi_{\epsilon}$ with~$\mathscr{U}$ as in~\eqref{mathscrUdefintro} and: \smallbreak
\begin{itemize}
\item[$\circ$] $\dist(\xi_j, \R \setminus I_{b_{\star}}) \geq \delta_0$ and $|\xi_j-\xi_i| \geq \delta$, for all $i,j \in \{1,\ldots,m\}$ with $i\neq j$; \smallbreak
\item[$\circ$] $\mu_j(\xi)$ as in \eqref{mu_j} for all $j \in \{1,\ldots,m\}$; \smallbreak
\item[$\circ$] $\psi_{\epsilon} \in L^{\infty}(\R)$ satisfying $\psi_{\epsilon} = 0$ a.e.~in $\R \setminus J_{b_{\star}}$ and $\|\psi_{\epsilon}\|_{L^{\infty}(J_{b_{\star}})} \to 0$, as $\epsilon \to 0$.
\end{itemize}
\end{proposition}

\medbreak

We stress that Proposition~\ref{nonExistence} is just a partial non-existence result, in that only solutions of the type~$u = \mathscr{U} + \psi_\epsilon$ are excluded and, as can be seen from the proof, the upper bound~$m_\star$ on the number of blow up points becomes unbounded as~$\delta_0 \to 0$. We believe it would be interesting to understand if a more precise non-existence result can be established, perhaps, as in~\cite{DlT-H-M-S-2020}, phrased in terms of a bound for~$\lambda$ in the mean-field type formulation~\eqref{meanfieldhalflap}.

\subsection*{Organization of the paper}

The remaining of the paper is organized as follows. The proof of Theorem~\ref{mainTheoremExistence} occupies Sections~\ref{The approximate solution}--\ref{Variational reduction}. More precisely, in Section~\ref{The approximate solution}, we analyze the approximate solution~$\mathscr{U}$; in Section~\ref{Linear theory}, we develop a conditional solvability theory for the linearized operator at~$\mathscr{U}$; in Section~\ref{Nonlinear theory}, we use these results to address a projected version of the original nonlinear problem; and, in Section~\ref{Variational reduction}, we deal with the finite-dimensional reduced problem concluding thus the proof of Theorem \ref{mainTheoremExistence}. On the other hand, Section~\ref{The non-existence result} is devoted to the proof of the non-existence result, namely Proposition~\ref{nonExistence}. The paper is then closed by three appendices: Appendix~\ref{The one dimensional half-Laplacian} gathers some known facts about the half-Laplacian adapted to our framework; Appendix~\ref{App Nondeg} contains a proof of a known non-degeneracy result (namely, the forthcoming Proposition~\ref{prop-nondegeneracy-L}); and Appendix~\ref{App dependence on the xijs} is devoted to the verification of a couple of technical steps in the proof of Theorem~\ref{mainTheoremExistence}.

\subsection*{Acknowledgments} 
This work has received funding from the European Research Council (ERC)
under the European Union's Horizon 2020 research and innovation programme
through the Consolidator Grant agreement 862342 (A.J.F.). The authors wish to thank M. Medina, A. Pistoia, and D. Ru\'iz for helpful discussions concerning the topic of the present paper.

\section{The approximate solution} \label{The approximate solution}

\noindent It is known, see~\cite[Theorem 1.8]{DL-M-R15}, that all solutions to
\begin{equation} \label{liouville}
(-\Delta)^{\frac12} u = e^u, \quad \textup{ in } \R,
\end{equation}
satisfying
$$
\int_\R e^u dx < + \infty,
$$
are of the the form
\begin{equation} \label{bubbles}
\mathcal{U}_{\mu,\xi}(x):= \log \left( \frac{2\mu}{\mu^2+(x-\xi)^2} \right), \quad \textup{ for some } \mu > 0 \textup{ and } \xi \in \R.
\end{equation}
This family of functions, the so-called bubbles, is going to be the basis on which we construct our first approximation for the solution to \eqref{main-problem}. For~$m \in \N$ fixed, let~$\xi_1, \ldots, \xi_m \in I$ and~$\mu_1, \ldots, \mu_m \in (0,+\infty)$ to be chosen later. Define
\begin{equation} \label{ujdef}
u_j(x) := \log \left( \frac{2\mu_j}{ \kappa(\xi_j)\big(\mu_j^2\varepsilon^2 + (x-\xi_j)^2\big)} \right) = \mathcal{U}_{\mu_j,0} \left( \frac{x-\xi_j}{\varepsilon} \right) - 2 \log \epsilon - \log \kappa (\xi_j)\,, 
\end{equation}
for all $ j \in \{1,\ldots,m\}$. One can directly check that, for all $j \in \{1,\ldots,m\}$, $u_j$ is a solution to
\begin{equation} \label{probforuj}
\Lfrac u_j = \epsilon \kappa(\xi_j) e^{u_j}, \quad \textup{ in } \R\,.
\end{equation}
Hence,~$u_j$ seems to be a good approximation for a solution to \eqref{main-problem} near $\xi_j$. It is then natural to guess that~$\sum_{j = 1}^{m} u_j$ would be a good approximation for a solution to~\eqref{main-problem} in the whole~$I$. However, such a function does not satisfy the boundary conditions. To fix this, for all $j \in \{1,\ldots,m\}$, we define $H_j$ to be the solution to
\begin{equation} \label{probforHj}
\left\{
\begin{aligned}
\Lfrac H_j & = 0, \quad && \textup{ in } I,\\
H_j & = - u_j, && \textup{ in } \R \setminus I,
\end{aligned}
\right.
\end{equation}
and the ansatz we consider is 
\begin{equation} \label{ansatz}
\mathscr{U}(x):= \sum_{j=1}^m \mathscr{U}_j(x) = \sum_{j = 1}^m \big(u_j(x) + H_j(x)\big) = \sum_{j=1}^m  \Big( \mathcal{U}_{\mu_j,0} \left( \frac{x-\xi_j}{\varepsilon} \right) - \log ( \kappa(\xi_j) \epsilon^2) + H_j(x) \Big).
\end{equation}

It is not hard to realize (see~\textbf{Case 2} in the proof of the forthcoming Proposition~\ref{error-measure}) that~$\mathscr{U}$ is already a good ansatz away from the~$\xi_j$'s. In order for it to be a good approximation near the concentration points too, we need to choose the~$\mu_j$'s appropriately. To this aim, let~$\Gamma(x):= -2\log|x|$ for all~$x \in \R \setminus \{0\}$ and recall (see for instance~\cite[Theorem 2.3]{B16}) that
$$
\Lfrac \Gamma = 2\pi \delta_0, \quad \textup{ in } \R.
$$
Here, for every~$x_0 \in \R$,~$\delta_{x_0}$ denotes the Dirac delta centered at~$x_0$. Having at hand the fundamental solution~$\Gamma$, for every~$z \in I$ we define~$\mathcal{H}^{z}$ to be the unique solution to
$$
	\left\{
	\begin{aligned}
		\Lfrac \mathcal{H}^{z} & = 0, \quad && \textup{ in } I,\\
		\mathcal{H}^{z} & = -\Gamma (\cdot - z), && \textup{ in } \R \setminus I,
	\end{aligned}
	\right.
$$
and we introduce~$H: \R \times I \to \R$ and~$G: \R \times I \to \R$ given respectively by
$$
	H(x,z) = \mathcal{H}^z(x), \quad \mbox{for } (x,z) \in \R \times I,
$$
and
\begin{equation}  \label{decomposition-Green}
	G(x,z) = H(x,z) + \Gamma(x-z), \quad \mbox{for } (x,z) \in \R \times I \textup{ with } x \neq z.
\end{equation}
We stress that, for every~$z \in I$, the function~$G(\cdot,z)$ is the unique solution to 
$$
	\left\{
	\begin{aligned}
		(-\Delta)^{\frac12} G(\cdot,z)  & = 2\pi \delta_{z}, \quad && \textup{ in } I, \\
		G(\cdot,z) & = 0, && \textup{ in } \R \setminus I.
	\end{aligned}
	\right.
$$
In other words,~$G$ is the Green function for the half Laplacian in~$I$. It is well-known that~$G$ and~$H$ are symmetric in~$I \times I$. We thus understand them to be extended to symmetric functions defined over~$\left( \R \times \R \right) \setminus \left( I \times I \right)$. Now that we have fixed this notation, we can choose
\begin{equation} \label{mu_j}
\mu_j = \mu_j(\xi) := \frac{1}{2} \exp \left( \log \big(\kappa(\xi_j)\big) + H(\xi_j,\xi_j) + \sum_{i \neq j} G(\xi_j, \xi_i) \right), \quad \textup{ for all } j \in \{1, \ldots, m\}.
\end{equation}

To quantify how well ansatz~$\mathscr{U}$ ``solves''~\eqref{main-problem}, it is convenient to express this problem in the expanded variable~$y = x/\varepsilon$. It is immediate to see that~$u$ is a solution of~\eqref{main-problem} if and only if~$v := u(\varepsilon \, \cdot \,) + 2 \log \varepsilon$ is a solution to
\[ \label{main-problem-expanded-variables} \tag{$Q_{\epsilon}$}
\left\{
\begin{aligned}
	\Lfrac v & = \kappa(\epsilon\cdot) e^v, \quad && \textup{ in } I_{\epsilon}, \\
	v & > 2\log \epsilon , && \textup{ in } I_{\epsilon}, \\
	v & = 2\log \epsilon, && \textup{ in } \R \setminus I_{\epsilon},
\end{aligned}
\right.
\]
where~$I_{\epsilon}:= \Big\{ {\displaystyle \frac x \epsilon: x \in I} \Big\}$. Then, we define
\begin{equation} \label{ansatz-expanded-variables}
	\mathscr{V}(y):= \mathscr{U}(\epsilon y) + 2\log \epsilon,
\end{equation}
and analyze how well~$\mathscr{V}$ ``solves''~\eqref{main-problem-expanded-variables}. To this end, we further define~$\eta_j := \epsilon^{-1} \xi_j$ for~$j \in \{1,\ldots,m\}$,
\begin{equation} \label{error}
	\mathscr{E}(y):= \Lfrac \mathscr{V}(y) - \kappa(\epsilon y) e^{\mathscr{V}(y)},
\end{equation}
and introduce the weighted $L^{\infty}$--norm
\begin{equation} \label{norm-star}
	\|v\|_{\star,\,\sigma} := \sup_{y \in I_\varepsilon} \left( \, \varepsilon + \sum\limits_{j = 1}^{m} \frac{1}{(1 + |y - \eta_j|)^{1+\sigma}} \right)^{-1} |v(y)|\,, 
\end{equation}
for some fixed~$\sigma \in (0,1)$. Also, for~$\delta_0 > 0$, we set
$$
	\mathcal{I}_{\delta_0}:= \Big\{ \xi:= (\xi_1, \ldots, \xi_m) \in I^{m}  : \dist( \xi_k, \R \setminus I) \geq \delta_0 \textup{ and } \min_{l \neq k} |\xi_k - \xi_l| \geq \delta_0, \mbox{ for all } k \in \{1, \ldots, m\} \Big\}.
$$
\begin{remark}
Note that, for~$\delta_0$ small enough, the set~$\mathcal{I}_{\delta_0}$ is non-empty. Also, for every~$\delta_0 > 0$ small, there exist two constants~$C_\star \ge c_\star > 0$ such that
\begin{equation} \label{mu-bounded both sides}
c_\star \leq \mu_j  \leq C_\star, \quad \mbox{ for all } \xi \in \mathcal{I}_{\delta_0} \textup{ and } j \in \{1,\ldots,m\}.
\end{equation}
\end{remark}


In the next proposition, which is the main result of this section, we estimate the size of~$\mathscr{E}$ in terms of the weighted~$L^{\infty}$--norm~$\|\cdot\|_{\star,\,\sigma}$. 

\begin{proposition} \label{error-measure}
Let~$\delta_0, \epsilon, \sigma \in (0,1)$ and~$\xi \in \mathcal{I}_{\delta_0}$. There exists a constant~$C > 0$, independent of~$\epsilon$,~$\sigma$, and $\xi$, such that
$$
\|\mathscr{E}\|_{\star,\,\sigma} \leq C \epsilon^{1-\sigma}. 
$$
\end{proposition}


To prove this proposition, we first need the following auxiliary result, containing the expansions of the functions~$H_j$ and~$u_j$ for~$\epsilon$ small.

\begin{lemma}  \label{Hj and uj expansions}
Let~$j \in \{1, \ldots, m\}$,~$\xi_j \in I$,~$\delta_1, \delta_2 > 0$, and~$\xi_j$ be such that~$\dist(\xi_j, \R \setminus I) \ge \delta_1$. Then:
\begin{enumerate}[label=$(\roman*)$]
\item \label{Hjandujlem(i)} For all~$x \in I$, we have~$H_j(x) = H(x,\xi_j) - \log \left( \frac{2\mu_j}{\kappa(\xi_j)} \right) + O(\epsilon^2)$~as~$\epsilon \to 0$.
\item \label{Hjandujlem(ii)}  For all~$x \in I$ such that~$|x - \xi_j| \geq \delta_2$, we have~$u_j(x) + H_j(x) = G(x,\xi_j) + O(\epsilon^2)$ as~$\epsilon \to 0$.
\end{enumerate}
\end{lemma}

As will be clear from the proof, for Lemma~\ref{Hj and uj expansions} to hold we do not actually need~$\mu_j$ to be defined as in~\eqref{mu_j}, but only that it satisfies the uniform bounds~\eqref{mu-bounded both sides}. Indeed, the constants involved in the big~O's depend only on~$\delta_1$,~$\delta_2$,~$I$, and on the constants~$c_\star$,~$C_\star$ appearing in~\eqref{mu-bounded both sides}.

\begin{proof}[Proof of Lemma~\ref{Hj and uj expansions}]
It follows from the definition of~$H_j$ that, for all~$x \in \R \setminus I$,
$$
H_j(x) = -u_j(x) = H(x,\xi_j) - \log \Big( \frac{2\mu_j}{\kappa(\xi_j)} \Big) + \log \Big( 1+ \frac{\epsilon^2 \mu_j^2}{(x-\xi_j)^2} \Big).
$$
Hence, we have that
\begin{equation*}
\left\{
\begin{aligned}
\Lfrac \Big[ H_j - H(\cdot, \xi_j) + \log \Big( \frac{2\mu_j}{\kappa(\xi_j)} \Big) \Big] & = 0, \quad && \textup{ in } I, \\
H_j - H(\cdot,\xi_j) + \log \Big( \frac{2\mu_j}{\kappa(\xi_j)} \Big) & =  \log \Big( 1+ \frac{\epsilon^2 \mu_j^2}{(\cdot - \xi_j)^2} \Big), && \textup{ in } \R \setminus I.
\end{aligned}
\right.
\end{equation*}
Thus, defining~$K=(a_1-3, b_d+3)$ and applying Proposition \ref{boundary regularity half-laplacian}, we get 
\begin{align*}
& \left\|H_j - H(\cdot, \xi_j)+ \log \Big( \frac{2\mu_j}{\kappa(\xi_j)} \Big) \right\|_{C^{0,\frac12}(\overline{I})} \\
& \qquad \leq C \left( \bigg\|  \log \Big( 1+ \frac{\epsilon^2 \mu_j^2}{(\cdot-\xi_j)^2} \Big) \bigg\|_{C^1(\overline{K}\setminus I)} + \bigg\|  \log \Big( 1+ \frac{\epsilon^2 \mu_j^2}{(\cdot-\xi_j)^2} \Big) \bigg\|_{L^\infty(\R \setminus I)} \right).
\end{align*}
for some~$C > 0$ depending only on~$I$. Since~$\dist(\xi_j, \R \setminus I) \ge \delta_1$,
$$
 \bigg\|  \log \Big( 1+ \frac{\epsilon^2 \mu_j^2}{(\cdot-\xi_j)^2} \Big) \bigg\|_{C^1(\overline{K} \setminus I)} + \bigg\|  \log \Big( 1+ \frac{\epsilon^2 \mu_j^2}{(\cdot-\xi_j)^2} \Big) \bigg\|_{L^\infty(\R \setminus I)} \le C \varepsilon^2,
$$
where~$C$ may now depend on~$\delta_1$ and on the constants~$c_\star$,~$C_\star$ occurring in~\eqref{mu-bounded both sides} as well. So, we deduce that, for all $x \in I$,
$$
H_j(x) = H(x,\xi_j) - \log \left( \frac{2\mu_j}{\kappa(\xi_j)} \right) + O(\epsilon^2), \quad \textup{ as } \epsilon \to 0.
$$

\noindent It remains to prove Point $(ii)$. As~$|x - \xi_j| \geq \delta_2$, it holds that
$$
\begin{aligned}
u_j(x) & = \log \left( \frac{1}{\mu_j^2 \epsilon^2 + (x-\xi_j)^2} \right) + \log \left( \frac{2\mu_j}{\kappa(\xi_j)} \right) \\
& = \log \left( \frac{2\mu_j}{\kappa(\xi_j)} \right) + \Gamma (x-\xi_j) - \log \Big( 1+ \frac{\epsilon^2 \mu_j^2}{(x-\xi_j)^2} \Big) \\
& = \log \left( \frac{2\mu_j}{\kappa(\xi_j)} \right) + \Gamma (x-\xi_j) + O(\epsilon^2), \qquad \textup{ as } \epsilon \to 0.
\end{aligned}
$$
Thus, taking into account \eqref{decomposition-Green} and Point $(i)$, the second point follows and the proof is concluded.
\end{proof}

We can now move to the

\begin{proof}[Proof of Proposition~\ref{error-measure}]
We consider separately two cases:
\medbreak
\noindent \textbf{Case 1:} \textit{There exists $j \in \{1,\ldots,m\}$ such that $|y-\eta_j| < \frac{\delta_0}{2\epsilon}$.}
\medbreak
First note that
\begin{equation} \label{eq1-error-measure}
\begin{aligned}
\kappa(\epsilon y) e^{\mathscr{V}(y)} & = \epsilon^2 \kappa(\epsilon y) \exp \left( \sum_{i = 1}^m \big( u_i(\epsilon y) + H_i(\epsilon y) \big) \right) \\
& = \epsilon^2 \kappa(\epsilon y) \exp \left( u_j (\epsilon y)+ H_j(\epsilon y) + \sum_{i \neq j} \big( u_{i}(\epsilon y) + H_i (\epsilon y) \big) \right) \\
& = \frac{2\mu_j \kappa(\epsilon y)}{\kappa(\xi_j)\big(\mu_j^2+(y-\eta_j)^2\big)}  \exp \left( H_j(\epsilon y) + \sum_{i \neq j} \big( u_{i}(\epsilon y) + H_i (\epsilon y) \big) \right).
\end{aligned}
\end{equation}
Now, using Lemma~\ref{Hj and uj expansions}~\ref{Hjandujlem(i)} and a Taylor expansion for~$H$ in the first variable (note that, for all~$x_0 \in I$,~$H(\cdot, x_0) \in C^1(I)$ by Lemma~\ref{regularity-Green}~$(i)$), we get 
$$
\begin{aligned}
H_l(\epsilon y)& = H(\epsilon y,\xi_l) - \log \Big( \frac{2\mu_l}{\kappa(\xi_l)} \Big) + O(\epsilon^2) \\
& = H(\xi_j,\xi_l) - \log \Big( \frac{2\mu_l}{\kappa(\xi_l)} \Big) + O(\epsilon|y-\eta_j|) + O(\epsilon^2), \quad \mbox{ for all } l \in \{1, \ldots,m\}.
\end{aligned}
$$
\noindent Also, since for all~$i \in \{1, \ldots, m\} \setminus \{j\}$ it follows that~$|\epsilon y-\xi_i| \geq \delta_0/2$, by Lemma~\ref{Hj and uj expansions}~\ref{Hjandujlem(ii)} and a Taylor expansion for~$G$ in the first variable (note that~$G(\cdot, \xi_i) \in C^1(I \setminus \{\xi_i\})$ by Lemma~\ref{regularity-Green}), we have 
$$
u_i(\epsilon y) + H_i (\epsilon y) = G(\epsilon y, \xi_i) + O(\epsilon^2) = G(\xi_j,\xi_i) + O(\epsilon|y-\eta_j|) + O(\epsilon^2), \quad \mbox{ for all } i \ne j.
$$
\noindent Substituting these expansions in \eqref{eq1-error-measure} and taking into account \eqref{mu_j}, we deduce that 
\begin{equation} \label{eq2-error-measure}
\begin{aligned}
\kappa(\epsilon y) e^{\mathscr{V}(y)} & = \frac{2\mu_j \kappa(\epsilon y)}{\kappa(\xi_j)\big(\mu_j^2+(y-\eta_j)^2\big)}  \exp \Bigg( -\log(2\mu_j) + \log (\kappa(\xi_j))  \\
& \hspace{5.2cm} + H(\xi_j,\xi_j) + \sum_{i \neq j} G(\xi_j, \xi_i) + O(\epsilon|y-\eta_j|) + O(\epsilon^2) \Bigg)  \\
& = \frac{2\mu_j \kappa(\epsilon y)}{\kappa(\xi_j)\big(\mu_j^2+(y-\eta_j)^2\big)}  \Big( 1+ O(\epsilon|y-\eta_j|)+O(\epsilon^2) \Big) \\
& = \frac{2\mu_j}{\mu_j^2+(y-\eta_j)^2} \Big( 1 + O(\epsilon|y-\eta_j|)+O(\epsilon^2) \Big).
\end{aligned}
\end{equation}
Note that the last identity follows from a Taylor expansion for $\kappa$. 
On the other hand, since $|x-\xi_i| \geq \delta_0/2$ for all $i \in \{1,\ldots,m\} \setminus \{j\}$, it follows that
\begin{equation} \label{eq3-error-measure}
\begin{aligned}
\Lfrac \mathscr{V}(y) & = \Lfrac \left[ \sum_{i=1}^m \bigg( \mathcal{U}_{\mu_i,0}(y-\eta_i) + H_j(\epsilon y) \bigg) \right] \\
& = \sum_{i=1}^m \Lfrac \mathcal{U}_{\mu_i,0}(y-\eta_i) = \frac{2\mu_j}{\mu_j^2 + (y - \eta_j)^2} + \sum_{i\neq j}^m \frac{2\mu_i}{\mu_i^2 + (y - \eta_i)^2} \\
& = \frac{2\mu_j}{\mu_j^2 + (y - \eta_j)^2} + \epsilon^2 \sum_{i\neq j}^m \frac{2\mu_i}{\epsilon^2 \mu_i^2 + (x - \xi_i)^2}  = \frac{2\mu_j}{\mu_j^2 + (y - \eta_j)^2} + O(\epsilon^2).
\end{aligned}
\end{equation}
Thus, combining~\eqref{eq2-error-measure} with~\eqref{eq3-error-measure} and taking into account the definition of~$\mathcal{I}_{\delta_0}$ and~\eqref{mu-bounded both sides}, we conclude that, for all $ y \in I_{\epsilon}$ satisfying $|y-\eta_j| < \frac{\delta_0}{2\epsilon}$, 
\begin{equation} \label{conclussion-case1-error-measure}
|\mathscr{E}(y)| \leq  \frac{C \epsilon}{1+|y-\eta_j|} + O(\epsilon^2)  \leq  \frac{2C \epsilon^{1-\sigma}}{(1+|y-\eta_j|)^{1+\sigma}} + O(\epsilon^2),
\end{equation}
with~$C > 0$ independent of~$\epsilon$,~$\sigma$, and~$\xi$.
\medbreak
\noindent \textbf{Case 2: }\textit{$|y-\eta_j| \geq \frac{\delta_0}{2\epsilon}$ for all $j \in \{1,\ldots,m\}$.}
\medbreak
Since~$|x-\xi_j|\geq \frac{\delta_0}{2}$ for all $j \in \{1,\ldots,m\}$, we can directly use Lemma~\ref{Hj and uj expansions}~\ref{Hjandujlem(ii)} and  get that
\begin{equation} \label{eq4-error-measure}
\kappa(\epsilon y) e^{\mathscr{V}(y)} = \epsilon^2 \kappa(\epsilon y) \exp \left( \sum_{i=1}^m G(x,\xi_i) + O(\epsilon^2) \right) = O(\epsilon^2).
\end{equation}
On the other hand, arguing as in the first case, we obtain
$$
\Lfrac \mathscr{V}(y) = \epsilon^2 \sum_{i = 1}^m \frac{2\mu_i}{\epsilon^2 \mu_i^2 + (x-\xi_i)^2} = O(\epsilon^2).
$$
Hence, for all $y \in I_{\epsilon}$ satisfying $|y-\eta_j| \geq \frac{\delta_0}{2\epsilon}$ for all $j \in \{1,\ldots,m\}$, it follows that
\begin{equation} \label{conclussion-case2-error-measure}
|\mathscr{E}(y)| = O(\epsilon^2).
\end{equation}
Gathering \eqref{conclussion-case1-error-measure} and \eqref{conclussion-case2-error-measure} the result follows.
\end{proof}

Having at hand the approximate solution $\mathscr{V}$ (cf. \eqref{ansatz-expanded-variables}), we look for solutions to \eqref{main-problem-expanded-variables} of the form $v = \mathscr{V} + \phi$. If we want $v$ to be a solution to \eqref{main-problem-expanded-variables}, we need to find $\phi$ solving
\begin{equation} \label{nonlinprobforphi}
\left\{
\begin{aligned}
L(\phi) & = - \mathscr{E} + \mathcal{N}(\phi), \quad &&\textup{ in } I_{\epsilon},\\
\phi & = 0, && \textup{ in } \R \setminus I_{\epsilon},
\end{aligned}
\right.
\end{equation}
where $\mathscr{E}$ is given in \eqref{error} and we set
\begin{equation} \label{L - W and N}
L(\phi) := \Lfrac \phi - W \phi, \qquad W:= \kappa(\epsilon \cdot)e^{\mathscr{V}}, \quad \textup{ and } \quad \mathcal{N}(\phi) :=  W\big( e^{\phi} -1 -\phi \big).
\end{equation}

\noindent A key step to solve a problem of the form \eqref{nonlinprobforphi} for small $\phi$ is to develop a suitable solvability theory for the linear operator $L$. This is the aim of the next section. 

\begin{remark} In the rest of the paper we will use the fact that the potential~$W$ can be written as
\begin{equation} \label{Wasperturb}
W(y) = \sum_{j = 1}^m \frac{2 \mu_j}{\mu_j^2 + (y - \eta_j)^2} + \theta(y),
\end{equation}
with~$\theta$ satisfying
\begin{equation} \label{thetabounds}
|\theta(y)| \le C \varepsilon \sum_{j = 1}^m \frac{1}{1 + |y - \eta_j|}, \quad \mbox{ for all } y \in I_\varepsilon,
\end{equation}
for some constant~$C > 0$ independent of $\epsilon, \sigma$ and $\xi$. This can be readily established via the computations of Lemma~\ref{error-measure} (cf.~\eqref{eq2-error-measure} and~\eqref{eq4-error-measure}).
\end{remark}

\section{Linear theory} \label{Linear theory}

\noindent
We address here the solvability of the Dirichlet problem
\begin{equation} \label{linear-eq}
\left\{
\begin{aligned}
L \phi & = g, \quad && \textup{ in } I_{\epsilon},\\
\phi & = 0, \quad && \textup{ in } \R \setminus I_{\epsilon}.
\end{aligned}
\right.
\end{equation}
under suitable orthogonality assumptions on~$\phi$. A key fact needed to develop such linear theory for~$L$ is the so-called non-degeneracy of the entire solutions to \eqref{liouville} modulo the natural invariance of the equation under translations and dilations. Indeed, note that if~$u$ is a solution to~\eqref{liouville}, then~$u_{\tau}(y):= u(\tau y) + \ln(\tau)$ and~$u_{\eta}(y):= u(y-\eta)$ are also solutions to~\eqref{liouville} for all~$\tau > 0$ and~$\eta \in \R$. Next, observe that the linearized operator for~\eqref{liouville} at a bubble~$\mathcal{U}_{\mu, \eta}$ is given by
\begin{equation} \label{Lmueta}
L_{\mu,\eta}\phi := (-\Delta)^{\frac12} \phi - W_{\mu, \eta} \,\phi, \quad \textup{ with } \quad W_{\mu, \eta}(y):= \frac{2\mu}{\mu^2+(y - \eta)^2},
\end{equation}
and that the functions
\begin{equation} \label{Z0Z1defs}
Z_{0, \mu,\eta}(y) := \partial_{\mu} \mathcal{U}_{\mu,\eta}(y) =  \frac{1}{\mu} - \frac{2\mu}{\mu^2+(y-\eta)^2}, \quad\ Z_{1, \mu,\eta}(y) := \partial_{\eta} \mathcal{U}_{\mu,\eta}(y) = \frac{2(y-\eta)}{\mu^2 + (y-\eta)^2},
\end{equation}
are both bounded solutions of
\begin{equation} \label{linearized-problem}
L_{\mu, \eta} \phi = 0, \quad \textup{ in }  \R.
\end{equation}
Also, recall that~$\mathcal{U}_{\mu,\eta}$ is \textit{non-degenerate} if all bounded solutions to~\eqref{linearized-problem} are a linear combination of~$Z_{0, \mu,\eta}$ and $Z_{1, \mu,\eta}$. In other words,~$\mathcal{U}_{\mu,\eta}$ is \textit{non-degenerate} if~$\ker(L_{\mu, \eta}) \cap L^{\infty}(\R) = \spn\big\{Z_{0, \mu,\eta},Z_{1, \mu,\eta}\big\}$. In the next result we see that this is indeed the case. 

\begin{proposition}{\rm(\hspace{-0.003cm}\cite[Proposition 2.1]{D-dP-M05}).} \label{prop-nondegeneracy-L}
Let~$\mu > 0$ and~$\eta \in \R$. If~$\phi \in L^{\infty}(\R)$ is a solution to~\eqref{linearized-problem} then~$\phi$ is a linear combination of~$Z_{0, \mu,\eta}$ and $Z_{1, \mu,\eta}$.
\end{proposition}

This result can be deduced from \cite[Proposition 2.1]{D-dP-M05} using the so-called Caffarelli-Silvestre extension \cite{C-S07}. An alternative proof without using the extension can be done via~\cite[Theorem 1.4]{S19}. For the benefit of the reader, we provide a detailed and self-contained proof in Appendix~\ref{App Nondeg} in the spirit of~\cite{S19}.

\medbreak

In order to solve problem~\eqref{linear-eq}, we first need to establish some \emph{a priori}~$L^\infty$ bounds. The right-hand side~$g$ will be measured in terms of the weighted $L^\infty$--norm $\|g\|_{\star,\sigma}$ (cf. \eqref{norm-star}) for some fixed~$\sigma \in (0, 1)$.  Before going further, let us introduce some notation. Let~$\xi_1, \ldots, \xi_m \in \mathcal{I}_{\delta_0}$ be arbitrary and~$\mu_1, \ldots, \mu_m$ be defined by~\eqref{mu_j}. We set
\begin{equation} \label{Zij}
Z_{i j} := Z_{i, \mu_j, \eta_j}, \quad \mbox{ for all } i = 0, 1, \ j = 1, \ldots, m.
\end{equation}
In addition, we take~$\bar{R} > 0$ large but fixed and an even cut-off function~$\chi \in C^\infty$ satisfying~$0 \le \chi \le 1$ in~$\R$,~$\chi = 1$ in~$[- \bar{R}, \bar{R}]$, and~$\chi = 0$ in~$\R \setminus (- \bar{R} - 1, \bar{R} + 1)$. Then, for~$j = 1, \ldots, m$, we set
\begin{equation} \label{chij}
\chi_j(y) := \chi(y - \eta_j), \quad \mbox{ for all } y \in \R.
\end{equation}
On the other hand, we also adopt the shortened notation
$$
L_\mu := L_{\mu, 0}
$$
and define $D:= \max\big\{|a_1|, |b_d|\big\}$. Finally, for $\sigma \in (0,1)$, we introduce the Banach space
$$
L^{\infty}_{\star,\,\sigma}(I_{\epsilon}):= \big\{ f \in L^{\infty}(I_{\epsilon}) : \|f\|_{\star,\,\sigma} < + \infty \big\}.
$$

\bigbreak
The main goal of this section is to establish the following result.

\begin{proposition} \label{mainlinearprop}
Let $\delta_0,\, \sigma \in (0,1)$. There exist $\varepsilon_0 \in (0, 1)$ and~$C > 0$ for which the following holds true: if $\xi \in \mathcal{I}_{\delta_0}$, $\varepsilon \in (0, \varepsilon_0]$ and~$g \in L^{\infty}_{\star,\,\sigma}(I_{\epsilon})$, there exists a unique bounded weak solution to
\begin{equation} \label{mainprojlinearprob}
\left\{
\begin{aligned}
L \phi & = g + \sum_{j = 1}^m c_j \chi_j Z_{1 j}, \quad && \mbox{ in } I_\varepsilon, \\
\phi & = 0, \quad && \mbox{ in } \R \setminus I_\varepsilon,
\end{aligned}
\right.
\end{equation}
satisfying the orthogonality conditions
\begin{equation} \label{phiorthtoZ1j}
\int_\R \phi \chi_j Z_{1 j} = 0, \quad \mbox{ for all } j = 1, \ldots, m,
\end{equation}
for some uniquely determined $c_1, \ldots, c_m \in \R$. Moreover,~$\phi$ satisfies
\begin{equation} \label{mainlinearaprioriest}
\| \phi \|_{L^\infty(\R)} \le C \big( {\log \frac{1}{\epsilon}} \big) \| g \|_{\star,\, \sigma}.
\end{equation}
\end{proposition}




\subsection{A priori estimate for solutions orthogonal to the full approximated kernel} $ $ \medbreak

\noindent The aim of this subsection is to prove the following \textit{a priori} bound. Note that this estimate is better than~\eqref{mainlinearaprioriest} in its dependence on~$\epsilon$. However, to establish it we require the solution~$\phi$ to be orthogonal to both the~$Z_{0, j}$'s and the~$Z_{1, j}$'s, not only to the latter functions as in~\eqref{phiorthtoZ1j}.

\begin{lemma} \label{1stLinftyestlem}
Let $\delta_0,\, \sigma \in (0,1)$. There exist $\varepsilon_0 \in (0, 1)$ and~$C_0 > 0$ for which the following holds true: if $\xi \in \mathcal{I}_{\delta_0},\ \varepsilon \in (0, \varepsilon_0]$,~$g \in L^\infty_{\star,\,\sigma}(I_\varepsilon)$ and~$\phi$ is a bounded weak solution to~\eqref{linear-eq} satisfying
$$
\int_\R \phi \chi_j Z_{i j} = 0, \quad \mbox{ for all } i = 0, 1 \mbox{ and } j = 1, \ldots, m,
$$
then
$$
\| \phi \|_{L^\infty(\R)} \le C_0 \| g \|_{\star,\,\sigma}.
$$
\end{lemma}


\medbreak
Let us begin with some preliminary technical lemmas. First of all, we construct a useful barrier.

\begin{lemma} \label{wcomplem}
Let~$\sigma \in (0, 1)$ and~$w_\sigma(y) := (1 + y^2)^{- \frac{\sigma}{2}}$ for~$y \in \R$. Then, there exist two constants~$c_0 \in (0, 1)$ and~$R_0 \ge 1$, depending only on~$\sigma$, such that
$$
(- \Delta)^{\frac{1}{2}} w_\sigma(y) \le - \frac{c_0}{|y|^{1 + \sigma}}, \quad \mbox{ for all } y \in \R \setminus (- R_0, R_0).
$$
\end{lemma}
\begin{proof}
Let~$y \ge 2$ and write~$w = w_\sigma$. We have
$$
(-\Delta)^{\frac{1}{2}} w(y) = \frac{1}{\pi} \int_\R \frac{w(y) - w(y') + w'(y) (y' - y) \mathds{1}_{\left( - \frac{y}{2}, \frac{y}{2}\right)}(y' - y)}{(y' - y)^2} \, dy',
$$
so that, changing variables as~$y' = t y$, it holds 
$$
y^{1 + \sigma} (-\Delta)^{\frac{1}{2}} w(y) = \frac{1}{\pi} \int_\R \frac{\Psi_y(t)}{(t - 1)^2} \, dt, \ \, \textup{with}\ \,  \Psi_y(t) := y^{\sigma} \left\{ w(y) - w(t y) + w'(y) y (t - 1) \mathds{1}_{\left( \frac{1}{2}, \frac{3}{2}\right)}(t) \right\}.
$$
We plan to let $y \rightarrow +\infty$ using Lebesgue's dominated convergence Theorem. For~$y \ge 2$ and~$t \in \left( \frac{1}{2}, \frac{3}{2} \right)$, we compute
\begin{align*}
w'(y) & = - \frac{\sigma y}{(1 + y^2)^{1 + \frac{\sigma}{2}}}, \hspace{-60pt} & w''(y) & = \sigma \frac{(1 + \sigma) y^2 - 1}{(1 + y^2)^{2 + \frac{\sigma}{2}}}, \\
\Psi_y'(t) & = y^{1 + \sigma} (w'(y) - w'(t y)), \hspace{-60pt} & \Psi_y''(t) & = - y^{2 + \sigma} w''(t y).
\end{align*}
Thus, a Taylor expansion yields that, for $y \ge 2$ and~$t \in \left( \frac{1}{2}, \frac{3}{2} \right)$, there exists~$\theta = \theta(y, t) \in [0, 1]$ such that
\begin{align*}
\big| \Psi_y(t) \big| & = \left| \Psi_y(1) + \Psi_y'(1) (t - 1) + \frac{1}{2} \Psi_y'' \big( {1 - \theta + \theta t} \big) (t - 1)^2 \right| \\
& = \frac{y^{2 + \sigma}}{2} w''\big( (1 - \theta + \theta t) y \big) (t - 1)^2 \le 8 (t - 1)^2.
\end{align*}
On the other hand, for~$t \in \R \setminus \left( \frac{1}{2}, \frac{3}{2} \right)$, we simply have
$$
\big| \Psi_y(t) \big| = y^{\sigma} \left| w(y) - w(t y) \right| \le \left( \frac{y^2}{1 + y^2} \right)^{\! \frac{\sigma}{2}} + \left( \frac{y^2}{1 + t^2 y^2} \right)^{\! \frac{\sigma}{2}} \le 1 + |t|^{- \sigma}.
$$
All in all, we have established that
$$
\frac{\big| \Psi_y(t) \big|}{(t - 1)^2} \le g(t) := 8 \mathds{1}_{\left( \frac{1}{2}, \frac{3}{2} \right)}(t) + \frac{1 + |t|^{- \sigma}}{(t - 1)^2} \mathds{1}_{\R \setminus \left( \frac{1}{2}, \frac{3}{2} \right)}(t), \quad \mbox{ for all } t \in \R \mbox{ and } y \ge 2.
$$
As~$\sigma > 0$, we have that~$g \in L^1(\R)$. Thus, by dominated convergence
\begin{align*}
\gamma(\sigma) := & \,\, \lim_{y \rightarrow +\infty} y^{1 + \sigma} (-\Delta)^{\frac{1}{2}} w(y) =  \frac{1}{\pi} \int_\R \lim_{y \rightarrow +\infty} \frac{\Psi_y(t)}{(t - 1)^2} \, dt \\
= & \,\, \frac{1}{\pi} \int_\R \lim_{y \rightarrow +\infty} \left\{ \left( \frac{y^2}{1 + y^2} \right)^{\frac{\sigma}{2}} - \left( \frac{y^2}{1 + t^2 y^2} \right)^{\frac{\sigma}{2}} - \sigma \left( \frac{y^2}{1 + y^2} \right)^{1 + \frac{\sigma}{2}} (t - 1) \mathds{1}_{\left( \frac{1}{2}, \frac{3}{2}\right)}(t) \right\} \frac{dt}{(t - 1)^2} \\
= & \,\,  \frac{1}{\pi} \int_\R \frac{1 - |t|^{- \sigma} - \sigma (t - 1) \mathds{1}_{\left( \frac{1}{2}, \frac{3}{2}\right)}(t)}{(t - 1)^2} \, dt = (- \Delta)^{\frac{1}{2}} \big( {|\cdot|^{- \sigma}} \big)(1).
\end{align*}

\noindent We now claim that~$\gamma(\sigma) < 0$ for all~$\sigma \in (0, 1)$. To see this, set
$$
K(t) := \frac{1}{(t - 1)^2} + \frac{1}{(t + 1)^2}, \quad \mbox{ for } t > 0,
$$
and notice that~$K(t^{-1}) = t^2 K(t)$. By virtue of this and changing variables appropriately, we write
\begin{align*}
\pi \gamma(\sigma) & = \PV \int_\R \frac{1 - |t|^{- \sigma}}{(t - 1)^2} \, dt = \PV  \int_0^{+\infty} \frac{1 - t^{- \sigma}}{(t - 1)^2} \, dt  + \int_{-\infty}^0 \frac{1 - (-t)^{- \sigma}}{(t - 1)^2} \, dt \\
& = \PV \int_0^{+\infty} \frac{1 - t^{- \sigma}}{(t - 1)^2} \, dt +  \int_0^{+\infty} \frac{1 - \tau^{- \sigma}}{(\tau + 1)^2} \, d\tau = \PV \int_0^{+\infty} (1 - t^{- \sigma}) K(t) dt \\
& = \lim_{\delta \to 0^{+}} \left\{ \int_0^{1 - \delta} (1 - t^{- \sigma}) K(t) \, dt + \int_{1 + \delta}^{+\infty} (1 - t^{- \sigma}) K(t) \, dt \right\} \\
& = \lim_{\delta \to 0^{+}} \left\{ \int_0^{1 - \delta} (1 - t^{- \sigma}) K(t) \, dt + \int_0^{\frac{1}{1 + \delta}} (1 - \tau^{\sigma}) K(\tau^{-1}) \tau^{-2} \, d\tau \right\} \\
& = \lim_{\delta \to 0^{+}} \left\{ \int_0^{1 - \delta} (2 - t^{- \sigma} - t^\sigma) K(t) \, dt + \int_{1 - \delta}^{\frac{1}{1 + \delta}} (1 - t^{\sigma}) K(t) \, dt \right\}.
\end{align*}
Now, on one hand, since~$(1 - t^\sigma) K(t) = O(\delta^{-1})$ for all~$t \in \big[ \frac{1}{2}, \frac{1}{1 + \delta} \big)$ and~$\frac{1}{1 + \delta} - 1 + \delta = O(\delta^2)$, we get 
$$
\int_{1 - \delta}^{\frac{1}{1 + \delta}} (1 - t^{\sigma}) K(t) \, dt = O(\delta) \to 0, \quad \textup{ as } \delta \to 0.
$$
On the other hand,~$2 - t^{- \sigma} - t^\sigma = - \left( t^{\frac{\sigma}{2}} - t^{- \frac{\sigma}{2}} \right)^{\! 2}$ and therefore
$$
\gamma(\sigma) = - \int_0^1 \left( t^{\frac{\sigma}{2}} - t^{- \frac{\sigma}{2}} \right)^{\! 2} K(t) \, dt < 0, \quad \mbox{ for every } \sigma \in (0, 1).
$$
Thus, we have proved that there exists~$R > 0$ depending only on~$\sigma$ for which
$$
(-\Delta)^{\frac{1}{2}} w(y) \le - \frac{\pi|\gamma(\sigma)|}{2} \, y^{- 1 - \sigma}, \quad \mbox{ for all } y \ge R.
$$
As, by symmetry, the same holds true also for~$y \le - R$, the proof is complete.
\end{proof}

Having at hand the barrier constructed in the previous lemma, we now prove that the operator~$L$ enjoys the maximum principle outside of a large neighbourhood of the points~$\eta_j$.

\begin{lemma} \label{maxprincforLlem}
Let $\delta_0 \in (0,1)$. There exists $R_1 > 0$ for which the following holds true: given~$\varepsilon$ and~$\xi \in \mathcal{I}_{\delta_0}$, if~$\psi$ is a bounded and continuous weak supersolution to~$L\psi = 0$ in~$I_\varepsilon \setminus \mathscr{B}_{R}$, with~$\mathscr{B}_{R} := \bigcup_{j = 1}^m (\eta_j - R, \eta_j + R)$ and~$R \ge R_1$, which is non-negative in~$\mathscr{B}_R \cup \left( \R \setminus I_\varepsilon \right)$, then~$\psi$ is non-negative in the whole~$\R$.
\end{lemma}


\begin{proof}
In light of Lemma~\ref{maxprincequivlemma}, it suffices to exhibit a positive supersolution~$z$ in~$I_\varepsilon \setminus \mathscr{B}_{R}$. To this aim, let~$w = w_\sigma$ be as in Lemma~\ref{wcomplem}, for some arbitrary but fixed~$\sigma \in (0, 1)$, and set
$$
z(y) := m + 1 - \sum_{j = 1}^{m} w(y - \eta_j).
$$
By definition, it is clear that the function $z$ is non-negative in $\R$ and positive in $\R \setminus \mathscr{B}_R$ for all $R > 0$. Also, from~\eqref{Wasperturb}--\eqref{thetabounds} we estimate that
$$
W(y) \le \sum_{j = 1}^m \frac{2 \mu_j}{\mu_j^2 + (y - \eta_j)^2} + C \varepsilon \sum_{j = 1}^m \frac{1}{1 + |y - \eta_j|} \le \sum_{j = 1}^m \frac{C}{|y - \eta_j|^2}, \quad \mbox{ for all } y \in I_\varepsilon,
$$
for some $C > 0$ independent of $\epsilon$ and $\xi$. Hence, we have that

\begin{equation} \label{Wz}
W(y) z(y) \le (m + 1) W(y) \le \sum_{j = 1}^m \frac{C_0}{|y - \eta_j|^2}, \quad \mbox{ for all } y \in I_\varepsilon,
\end{equation}
for some possibly larger~$C_0 > 0$. On the other hand, by Lemma~\ref{wcomplem}, we have
\begin{equation} \label{fraclapofzislarge}
(-\Delta)^{\frac{1}{2}} z(y) = - \sum_{j = 1}^{m} (-\Delta)^{\frac{1}{2}} w(y - \eta_j) \ge \sum_{j = 1}^m \frac{c_0}{|y - \eta_j|^{1 + \sigma}}, \quad \mbox{for all } y \in \R \setminus \mathscr{B}_R,
\end{equation}
provided~$R \ge R_0$, with~$c_0 \in (0, 1)$ and~$R_0 \ge 1$ depending only on~$\sigma$. We now choose $R_1 \geq \max\{R_0, (2 C_0/c_0)^{1/(1-\sigma)}\} \geq 1$ and fix an arbitrary~$R \geq R_1$. If $I_{\epsilon} \setminus \mathscr{B}_R = \varnothing$, there is nothing to prove. Let us then assume that $I_{\epsilon} \setminus \mathscr{B}_R \neq \varnothing$. Combining \eqref{Wz} with \eqref{fraclapofzislarge}, we get
\begin{align*}
L z(y) & = (-\Delta)^{\frac{1}{2}} z(y) - W(y) z(y) \ge \sum_{j = 1}^m \frac{1}{|y - \eta_j|^{1 + \sigma}} \left( c_0 - \frac{C_0}{|y - \eta_j|^{1 - \sigma}} \right) \\
& \ge \left( c_0 - \frac{C_0}{R_1^{1 - \sigma}} \right) \sum_{j = 1}^m \frac{1}{|y - \eta_j|^{1 + \sigma}} \ge \frac{c_0}{2} \, \sum_{j = 1}^m \frac{1}{|y - \eta_j|^{1 + \sigma}} > 0, \quad \textup{ for all } y \in I_\varepsilon \setminus \mathscr{B}_R.
\end{align*}
Thus,~$z$ is the desired supersolution and the proof is complete.
\end{proof} 
 
Thanks to the above maximum principle we can now prove Lemma~\ref{1stLinftyestlem}.


\begin{proof}[Proof of Lemma \ref{1stLinftyestlem}]
We argue by contradiction. Taking into account the homogeneity of the equation, we suppose that, for any~$n \in \N$, there exist~$\varepsilon_n \in (0, 1)$,~$\xi^{(n)} \in \mathcal{I}_{\delta_0}$, and two functions~$\phi_n, g_n \in L^\infty(\R)$ such that~$\varepsilon_n \to 0^{+}$ as~$n \rightarrow +\infty$, the function~$\phi_n$ weakly satisfies the equation
$$
\left\{
\begin{aligned}
L^{(n)}(\phi_n) & = g_n, \quad && \textup{ in } I_{\epsilon_n},\\
\phi_n & = 0, \quad && \textup{ in } \R \setminus I_{\epsilon_n},
\end{aligned}
\right.
$$
as well as the orthogonality conditions
$$
\int_\R \phi_n \chi_j^{(n)} Z_{i j}^{(n)} = 0, \quad \mbox{ for all } i = 0, 1, \, j = 1, \ldots, m,
$$
but
\begin{equation} \label{phin=1gn=1/n}
\| \phi_n \|_{L^\infty(\R)} = 1 \quad \mbox{and} \quad {\| g_n \|}_{\star, \sigma}^{(n)} < \frac{1}{n}.
\end{equation}
Here, we set~$\eta^{(n)}_j := \varepsilon_n^{-1} \xi^{(n)}_j$,~$\mu_j^{(n)} := \mu_j(\xi^{(n)})$ as in~\eqref{mu_j},~$\chi^{(n)}_j := \chi( \,\cdot\, - \eta_j^{(n)})$,~$Z^{(n)}_{i j} := Z_{i, \mu_j^{(n)}, \eta_j^{(n)}}$,
$$
{\| g \|}_{\star, \sigma}^{(n)} := \sup_{y \in I_{\varepsilon_n}} \, \dfrac{|g(y)|}{ \, \varepsilon_n + \sum\limits_{j = 1}^{m} \big( 1 + \big| y - \eta_j^{(n)} \big| \big)^{- 1 - \sigma}},
$$
and
$$
L^{(n)}(\phi) := (-\Delta)^{\frac{1}{2}} \phi - W^{(n)} \phi,
$$
where, recalling decomposition~\eqref{Wasperturb}--\eqref{thetabounds},
\begin{equation} \label{Wnasperturb}
W^{(n)}(y) = \sum_{j = 1}^m \frac{2 \mu_j^{(n)}}{\big( \mu_j^{(n)} \big)^2 + \big( y - \eta_j^{(n)} \big)^2} + \theta_n(y),
\end{equation}
with
\begin{equation} \label{thetanbounds}
|\theta_n(y)| \le C \varepsilon_n \sum_{j = 1}^m \frac{1}{1 + \big| y - \eta_j^{(n)} \big|}, \quad \mbox{ for all } y \in I_{\varepsilon_n},
\end{equation}
for some constant~$C > 0$ independent of~$n$.

Let~$R > 0$ be a fixed large constant independent of~$n$ and consider the \emph{inner} norm
$$
{\| \phi \|}_{R, n} := \| \phi \|_{L^\infty(\mathscr{B}_{R, n})}, \quad \mbox{with } \, \mathscr{B}_{R, n} := \bigcup_{j = 1}^m \left( \eta_j^{(n)} - R, \eta_j^{(n)} + R \right).
$$
We claim that 
\begin{equation} \label{tailestclaim}
\| \phi_n \|_{L^\infty(\R)} \le C \Big( {\| \phi_n \|}_{R, n} + {\| g_n \|}_{\star,\sigma}^{(n)} \Big),
\end{equation}
provided~$R \ge R_0$ with~$R_0 \ge 1$ large enough and for some constant~$C > 0$, both independent of~$n$. To see this, it suffices to consider the function
$$
\overline{\phi}_n(y) := M \left( m + 1 + \left( 1 - \frac{\varepsilon_n^2 y^2}{(D+1)^2} \right)_{\! +}^{\!\frac{1}{2}} - \sum_{j = 1}^m w_\sigma \big( y - \eta_j^{(n)} \big) \right),
$$
with~$w_\sigma$ as in Lemma~\ref{wcomplem} and for some~$M > 0$ to be chosen large. Arguing as in the proof of Lemma~\ref{maxprincforLlem} and taking advantage of the fact that~$(-\Delta)^{\frac{1}{2}} \left( 1 - (D+1)^{-2}\varepsilon_n^2 y^2 \right)_{+}^{\frac{1}{2}} = \bar{c} \, \varepsilon_n$ in~$(- (D + 1) \varepsilon_n^{-1}, (D + 1) \varepsilon_n^{-1})$ for some numerical constant~$\bar{c} > 0$ (see, e.g.,~\cite[Theorem 1]{Dy2012}), one can directly check that~$\overline{\phi}_n$ satisfies
$$
L^{(n)} \, \overline{\phi}_n(y) \ge \frac{M}{C_1} \left( \varepsilon_n + \sum_{j = 1}^m \frac{1}{\big( 1 + \big| y - \eta_j^{(n)} \big| \big)^{1 + \sigma}} \right), \quad \mbox{for all } y \in I_{\varepsilon_n} \setminus \mathscr{B}_R,
$$
for some constant~$c_1 > 0$ depending only on~$\sigma$ and provided~$R \ge R_0$, with~$R_0 \ge 1$ depending only on~$\delta_0$,~$m$,~$D$, and~$\sigma$. Hence,~$L^{(n)} \overline{\phi}_n \ge g_n$ in~$I_{\varepsilon_n} \setminus \mathscr{B}_R$, provided~$M \ge C_1 {\| g_n \|}_{\star, \sigma}^{(n)}$. In addition,~$\overline{\phi}_n \ge \phi_n$ in~$\mathscr{B}_R \cup \left( \R \setminus I_{\varepsilon_n} \right)$, provided~$M \ge {\| \phi_n \|}_{R, n}$. We then choose~$M := {\| \phi_n \|}_{R, n} + M_0 {\| g_n \|}_{\star, \sigma}^{(n)}$ and, by Lemma~\ref{maxprincforLlem}, we obtain that
$$
\phi_n(y) \le \overline{\phi}_n(y) \le (m + 2) \Big( {\| \phi_n \|}_{R, n} + C_1 {\| g_n \|}_{\star, \sigma}^{(n)} \Big) \le C \Big( {\| \phi_n \|}_{R, n} + {\| g_n \|}_{\star,\sigma}^{(n)} \Big), \quad \mbox{for all } y \in \R,
$$
for some constant~$C > 0$ independent of~$n$. Since an analogous lower bound can be easily established by considering~$-\overline{\phi}_n$ instead of~$\overline{\phi}_n$, we conclude that claim~\eqref{tailestclaim} holds true.

From~\eqref{tailestclaim} and~\eqref{phin=1gn=1/n}, we get the existence of a small constant~$\kappa > 0$ independent of~$n$ such that
$$
{\| \phi_n \|}_{R, n} \ge 2 \kappa, \quad \mbox{ for every large } n.
$$
Hence, for any large~$n$, there exists a point~$p_n \in \mathscr{B}_{R, n}$ at which
$$
|\phi_n(p_n)| \ge \kappa.
$$
Up to considering a subsequence, there exists~$k \in \{ 1, \ldots, m \}$ such that~$p_n \in \big( \eta_k^{(n)} - R, \eta_k^{(n)} + R \big)$ for every large~$n$. That is,~$p_n - \eta_k^{(n)} \in (- R, R)$ for every large~$n$. Up to a further subsequence,~$p_n - \eta_k^{(n)}$ thus converges to some point~$p \in [- R, R]$ as~$n \rightarrow +\infty$. Consider the translations~$\psi_n := \phi_n(\,\cdot\, + p_n - p)$. Clearly,~$\psi_n$ satisfies
$$
\begin{dcases}
(-\Delta)^{\frac{1}{2}} \psi_n - V^{(n)} \psi_n = f_n, & \quad \mbox{ in } \mathscr{I}_{\epsilon_n}, \\
\int_\R \psi_n(y) \chi \big( y + p_n - \eta_k^{(n)} - p \big) Z_{i k}^{(n)}(y + p_n - p) \, dy = 0, & \quad \mbox{ for } i = 0, 1,\\
\vphantom{\int_\R \psi_n(y) \chi \big( y + p_n - \eta_k^{(n)} - p \big) Z_{i k}^{(n)}(y + p_n - p) \, dy = 0} \, \| \psi_n \|_{L^\infty(\R)} = 1, \quad \mbox{ and } \quad |\psi_n(p)| \ge \kappa, & \\
\end{dcases}
$$
for every large~$n$, where~$f_n := g_n (\,\cdot\, + p_n - p)$,~$V^{(n)} := W^{(n)}(\,\cdot\, + p_n - p)$, and~$\mathscr{I}_{\epsilon_n} := I_{\varepsilon_n} - p_n + p$. Note that, as~$\xi^{(n)} \in \mathcal{I}_{\delta_0}$, it holds
$$
J_n := \left( - \frac{\delta_0}{4 \varepsilon_n}, \frac{\delta_0}{4 \varepsilon_n} \right) \subset \left( \eta_k^{(n)} - p_n + p - \frac{\delta_0}{\varepsilon_n}, \eta_k^{(n)} - p_n + p + \frac{\delta_0}{\varepsilon_n} \right) \subset \mathscr{I}_{\epsilon_n},
$$
for~$n$ sufficiently large. In addition,~$\| f_n \|_{L^\infty(J_n)} \le (m + 1)/n$. Thanks to~\eqref{mu-bounded both sides} and up to taking another subsequence, we also have that~$\mu_k^{(n)} \rightarrow \mu \in (0, +\infty)$ as~$n \rightarrow +\infty$. Recalling~\eqref{Wnasperturb}--\eqref{thetanbounds}, we estimate
\begin{equation} \label{Vnlimit}
\begin{aligned}
\left| V^{(n)}(y) - \frac{2 \mu}{\mu^2 + y^2} \right| & \le \left| \frac{2 \mu_k^{(n)}}{\big( \mu_k^{(n)} \big)^2 + \big( y + p_n - \eta_k^{(n)} - p \big)^2} - \frac{2 \mu}{\mu^2 + y^2} \right| \\
& \quad + \sum_{\substack{j = 1, \ldots, m\\ j \ne k}} \frac{2 \mu_j^{(n)}}{\big( \mu_j^{(n)} \big)^2 + \big( y + p_n - \eta_j^{(n)} - p \big)^2} + |\theta_n(y + p_n - p)| \\
& \le C \left( |\mu_k^{(n)} - \mu| + |p_n - \eta_k^{(n)} - p| + \varepsilon_n \right),
\end{aligned}
\end{equation}
for every~$y \in J_n$, for some constant~$C > 0$ independent of~$n$, and for~$n$ large enough.

We now take the limit as~$n \rightarrow +\infty$. Let~$M > 0$ and observe that~$(-M-1, M+1) \subset J_n$, provided~$n$ is sufficiently large. As~$\| \psi_n \|_{L^\infty(\R)} = 1$,~$\| f_n \|_{L^\infty(-M-1, M+1)} \le (m + 1)/n$, and~$\| V^{(n)} \|_{L^\infty(-M-1, M+1)}$ is uniformly bounded in~$n$ (and~$M$), by standard regularity theory for the half-Laplacian (see, e.g.,~Proposition~\ref{interior-regularity}) we infer that
$$
\| \psi_n \|_{C^\alpha(- M, M)} \le C,
$$
for any~$\alpha \in (0, 1)$ and for some constant~$C > 0$ independent of~$n$. Applying then  \cite[4.44 Arzel\`a-Ascoli Theorem II]{folland} we conclude that, up to a subsequence,~$\psi_n$ converges to some function~$\psi$ in~$C^\alpha_\loc(\R)$ as~$n \rightarrow +\infty$. Thanks to~\eqref{Vnlimit}, one can directly see that~$\psi$ satisfies
$$
\begin{dcases}
(-\Delta)^{\frac{1}{2}} \psi - W_\mu \psi = 0, & \quad \mbox{ in } \R, \\
\int_\R \psi \chi Z_{i, \mu, 0} = 0, & \quad \mbox{ for } i = 0, 1,\\
\vphantom{\int_\R \psi(y) \chi(y - \bar{y}) Z_{i, \mu, \bar{y}}(y) \, dy = 0} \, \| \psi \|_{L^\infty(\R)} \le 1, \quad \mbox{ and } \quad |\psi(p)| \ge \kappa, & \\
\end{dcases}
$$
with~$W_\mu := W_{\mu, 0}$ as in~\eqref{Lmueta}. The fact that the equation is satisfied (in the weak sense), can be easily justified by taking~$\alpha > 1/2$, which ensures that~$\psi \in H^{\frac{1}{2}}_\loc(\R)$.

From this, we immediately arrive at a contradiction. Indeed, Proposition~\ref{prop-nondegeneracy-L} gives that~$\psi$ is a linear combination of~$Z_{0, \mu, 0}$ and~$Z_{1, \mu, 0}$, i.e.,
$$
\psi(y) = a_0 Z_{0, \mu, 0}(y) + a_1 Z_{1, \mu, 0}(y), \quad \mbox{ for all } y \in \R,
$$
for some coefficients~$a_0, a_1 \in \R$. Recall now that~$Z_{0, \mu, 0}$ and~$\chi$ are even functions, while~$Z_{1, \mu, 0}$ is odd. As a result, the orthogonality conditions readily give that~$a_0 = a_1 = 0$. That is,~$\psi = 0$ in~$\R$, in contradiction with the fact that~$|\psi(p)| \ge \kappa > 0$. The lemma is thus proved.
\end{proof}

\subsection{A priori estimates for solutions orthogonal only to the~$Z_{1 j}$'s} $ $ 
\medbreak
\noindent We now establish a worse~$L^\infty$ estimate (worse in its dependence on~$\varepsilon$), which holds however for all solutions which are orthogonal to the~$Z_{1 j}$'s alone. First, we have a technical lemma.

\begin{lemma} \label{ztilde0lem}
Let~$\delta_0, \sigma \in (0, 1)$,~$\mu_0 \ge 1$,~$R \ge 2 \mu_0$, and~$\varepsilon \in \big( 0, \frac{\delta_0}{16 R} \big]$. For any~$\mu \in \big[ \frac{1}{\mu_0}, \mu_0 \big]$, there exists an even function~$\tilde{z}_0 = \tilde{z}_0^{(\mu)} \in C^\infty_c(\R)$ satisfying
\begin{equation} \label{z0tilderequirements1}
\tilde{z}_0 = Z_{0, \mu, 0} \mbox{ in } [- R, R], \quad
0 \le \tilde{z}_0 \le \frac{1}{\mu} \mbox{ in } \R \setminus [-R, R], \quad \tilde{z}_0 = 0 \mbox{ in } \R \setminus \left( - \frac{\delta_0}{3 \varepsilon}, \frac{\delta_0}{3 \varepsilon} \right),
\end{equation}
\begin{equation} \label{z0tildeL1est}
\| \tilde{z}_0 \|_{L^1(\R)} \le C \left( R + \frac{1}{\varepsilon |{\log \varepsilon}|} \right),
\end{equation}
\begin{equation} \label{estimateforL0z0}
|L_{\mu} \tilde{z}_0(y)| \le \frac{C}{|{\log \varepsilon}|} \left( \varepsilon + \frac{R^\sigma}{1 + |y|^{1 + \sigma}} \right), \quad \mbox{ for all } y \in \R,
\end{equation}
and
\begin{equation} \label{energylowerbound}
\int_\R \tilde{z}_0(y) L_{\mu} \tilde{z}_0(y) \, dy \ge \frac{1}{C |{\log \varepsilon}|},
\end{equation}
for some constant~$C \ge 1$, depending only on~$\delta_0$,~$\sigma$,~$\mu_0$, provided that~$R$ is sufficiently large, in dependence of~$\delta_0$,~$\sigma$,~$\mu_0$ alone, and that~$\varepsilon$ is sufficiently small, also in dependence of~$R$.
\end{lemma}
\begin{proof}
Let~$z_0(y) := Z_{0, \mu, 0}(y) = \frac{1}{\mu}\frac{y^2 - \mu^2}{y^2 + \mu^2}$ for~$y \in \R$ and consider its harmonic extension~$Z_0$ to the upper half-space~$\R^2_+ = \left\{ (y, \lambda) \in \R^2 : \lambda > 0 \right\}$, i.e.,
$$
Z_0(y, \lambda) = \frac{1}{\mu} \frac{y^2 + \lambda^2 - \mu^2}{y^2 + (\lambda + \mu)^2}, \quad \mbox{ for } y \in \R, \, \lambda \ge 0.
$$
We modify~$Z_0$ by means of a logarithmic cutoff inspired by the proof of~\cite[Lemma~4.4]{D-dP-M05}. Let~$\varrho \in C^\infty(\R)$ be a function satisfying~$0 \le \varrho \le 1$ in~$\R$,~$\varrho = 1$ in~$\left( -\infty, \frac{1}{10} \right]$,~$\varrho = 0$ in~$\left[ \frac{9}{10}, +\infty \right)$,~$\frac{1}{2} \mathds{1}_{\big[ \frac{1}{4}, \frac{3}{4} \big]} \le - \varrho' \le 2$ in~$\R$, and~$|\varrho''| \le 100$ in~$\R$. Given~$R \ge 2 \mu_0$ and~$\delta_0, \varepsilon \in (0, 1)$, let
$$
\varrho_1(Y) := \varrho \left( \frac{|Y|}{R} - 1 \right), \quad \varrho_2(Y) := \varrho \left( \frac{12 \varepsilon |Y|}{\delta_0} - 3 \right), \quad \mbox{ for } Y \in \R^2.
$$
Taking~$\varepsilon \le \frac{\delta_0}{16 R}$, we define
$$
h(Y) := \begin{dcases}
1, & \quad \mbox{ for } Y \in \overline{B}_R, \\
1 - \frac{\log \left( \frac{|Y|}{R} \right)}{\log \left( \frac{\delta_0}{3 \varepsilon R} \right)}, & \quad \mbox{ for } Y \in B_{\frac{\delta_0}{3\varepsilon}} \setminus \overline{B}_R,\\
0, & \quad \mbox{ for } Y \in \R^2 \setminus B_{\frac{\delta_0}{3\varepsilon}},
\end{dcases}
$$
and
$$
\widetilde{Z}_0(Y) := \Big( {\varrho_1(Y) + \big( {1 - \varrho_1(Y)} \big) \varrho_2(Y) h(Y)} \Big) Z_0(Y), \quad \mbox{ for } Y = (y, \lambda) \in \overline{\R^2_+}.
$$
Clearly,~$\widetilde{Z}_0 \in C^\infty(\overline{\R^2_+})$ and is supported in the closure of~$B^+_{\delta_0 / (3 \varepsilon)}$--here,~$B^+_R := B_R \cap \R^2_+$. We denote by~$\tilde{z}_0$ the restriction of~$\widetilde{Z}_0$ to~$\R \equiv \partial \R^2_+$, i.e.,~$\tilde{z}_0(y) := \widetilde{Z}_0(y, 0)$ for~$y \in \R$. Clearly,~$\tilde{z}_0$ satisfies requirements~\eqref{z0tilderequirements1}. We now show that it also fulfills~\eqref{z0tildeL1est},~\eqref{estimateforL0z0},~\eqref{energylowerbound}, provided that~$R$ is sufficiently large, in dependence of~$\sigma$,~$\delta_0$,~$\mu_0$, and that~$\varepsilon$ is sufficiently small, in dependence of the said quantities and~$R$.

Understanding the half-Laplacian as a Dirichlet-to-Neumann operator, we write
$$
L_{\mu} \tilde{z}_0 = (-\Delta)^{\frac{1}{2}} \tilde{z}_0 - W_\mu \tilde{z}_0 = - \partial_\lambda \widehat{Z}_0(\cdot, 0) - W_\mu \widetilde{Z}_0(\cdot,0), \quad \mbox{ in } \R,
$$
where $\widehat{Z}_0$ is the harmonic extension of~$\tilde{z}_0$ to~$\R^2_+$. As~$L_\mu z_0 = - \partial_\lambda Z_0(\cdot, 0) - W_\mu Z_0(\cdot,0) = 0$ in~$\R$ and~$\partial_\lambda \varrho_1 = \partial_\lambda \varrho_2 = \partial_\lambda h = 0$ on~$\partial \R^2_+$, by their rotational symmetry, we have
$$
\partial_\lambda \widetilde{Z}_0 = \Big( {\varrho_1 + \big( {1 - \varrho_1} \big) \varrho_2 h} \Big) \partial_\lambda Z_0 = - W_\mu \widetilde{Z}_0, \quad \textup{ on } \partial \R^2_{+}.
$$
Hence, it follows that
\begin{equation} \label{L0asD2N}
L_\mu \tilde{z}_0(y) = \partial_\lambda \big( {\widetilde{Z}_0 - \widehat{Z}_0} \big)(y, 0), \quad \mbox{ for all } y \in \R. 
\end{equation}
Since~$h$ is harmonic in~$B_{\delta_0/(3 \varepsilon)} \setminus \overline{B}_R$ and~$Z_0$ in~$\R^2_+$, using the properties of~$\varrho_1$ and~$\varrho_2$, we find that
$$
\Delta \widetilde{Z}_0 = f, \quad \mbox{ in } \R^2_+,
$$
with
\begin{equation} \label{fdef}
\begin{aligned}
f & := 2 \Big( {\big( {1 - h} \big) \nabla \varrho_1 + h \nabla \varrho_2 + \big( {1 - \varrho_1} \big) \varrho_2 \nabla h} \Big) \cdot \nabla Z_0 \\
& \quad\,\, + \Big( { \big( {1 - h} \big) \Delta \varrho_1 - 2 \nabla \varrho_1 \cdot \nabla h + h \Delta \varrho_2 + 2 \nabla \varrho_2 \cdot \nabla h } \Big) Z_0.
\end{aligned}
\end{equation}
In particular, as~$\widehat{Z}_0$ is harmonic in~$\R^2_+$, the difference~$\widetilde{Z}_0 - \widehat{Z}_0$ solves the Dirichlet problem
$$
\left\{
\begin{aligned}
\Delta \big( {\widetilde{Z}_0 - \widehat{Z}_0} \big) & = f, \quad && \mbox{ in } \R^2_+, \\
\widetilde{Z}_0 - \widehat{Z}_0 & = 0, \quad && \mbox{ on } \partial \R^2_+,
\end{aligned}
\right.
$$
and has therefore the Green function representation
\begin{align*}
\big( {\widetilde{Z}_0 - \widehat{Z}_0} \big)(y, \lambda) & = \frac{1}{4 \pi} \int_{\R^2_+} f(y', \lambda') \log \left( \frac{(y - y')^2 + (\lambda - \lambda')^2}{(y - y')^2 + (\lambda + \lambda')^2} \right) dy' d\lambda' \\
& = - \frac{1}{4 \pi} \int_{-\lambda}^{+\infty} \left( \int_\R f(y', \lambda + t) \log \left( \frac{(y - y')^2 + (2 \lambda + t)^2}{(y - y')^2 + t^2} \right) dy' \right) dt
\end{align*}
for all~$y \in \R$ and~$\lambda > 0$. Differentiating under the integral sign and recalling~\eqref{L0asD2N}, we get
\begin{equation} \label{L0ztilde0repr}
L_\mu \tilde{z}_0(y) = - \frac{1}{\pi} \int_{\R^2_+} \frac{f(y', \lambda)}{(y - y')^2 + \lambda^2} \, \lambda \, dy' d\lambda, \quad \mbox{ for all } y \in \R.
\end{equation}

We now estimate the right-hand side of the above identity. To do it, it is useful to recall that
\begin{align*}
& \rho_1 = 1 \mbox{ in } B_R, \quad \rho_1 = 0 \mbox{ in } \R^2 \setminus B_{2 R}, \quad |D^j \! \rho_1| \le \frac{C}{R^j} \mbox{ in } \R^2, \\
& \rho_2 = 1 \mbox{ in } B_{\frac{\delta_0}{4 \varepsilon}}, \quad \rho_2 = 0 \mbox{ in } \R^2 \setminus B_{\frac{\delta_0}{3 \varepsilon}}, \quad |D^j \! \rho_2| \le C \varepsilon^j \mbox{ in } \R^2,
\end{align*}
for~$j = 1, 2$, and to observe that
$$
|h(Y) - 1| \le C \, \frac{\log \left( \frac{|Y|}{R} \right)}{|{\log \varepsilon}|} , \quad |h(Y)| \le - C \, \frac{\log \left( \frac{3 \varepsilon |Y|}{\delta_0} \right)}{|{\log \varepsilon}|} , \quad |\nabla h(Y)| \le \frac{C}{|{\log \varepsilon}| |Y|}, \quad \mbox{ for all } Y \in B_{\frac{\delta_0}{3 \varepsilon}} \setminus B_R
$$
and
$$
|Z_0(Y)| \le C, \quad |\nabla Z_0(Y)| \le \frac{C}{1 +|Y|^2}, \quad \mbox{ for all } Y \in \R^2_+,
$$
for some constant~$C \ge 1$ depending only on~$\delta_0$,~$\mu_0$ and provided~$\varepsilon$ is small enough. Integrating in polar coordinates we compute
\begin{align*}
& \left| \int_{\R^2_+} \frac{h(y', \lambda) \nabla \varrho_2(y', \lambda) \cdot \nabla Z_0(y', \lambda)}{(y - y')^2 + \lambda^2} \, \lambda \, dy' d\lambda \right| \le \frac{C \varepsilon^3}{|{\log \varepsilon}|} \int_{B_{\frac{\delta_0}{3 \varepsilon}}^+ \setminus B_{\frac{\delta_0}{4 \varepsilon}}^+} \frac{\lambda \, dy' d\lambda}{(y - y')^2 + \lambda^2} \\
& \hspace{10pt} = \frac{C \varepsilon^3}{|{\log \varepsilon}|} \int_{\frac{\delta_0}{4 \varepsilon}}^{\frac{\delta_0}{3 \varepsilon}} \left( \int_0^\pi \frac{\sin \theta}{r^2 + y^2 - 2 r y \cos \theta} \, d\theta \right) r^2 \, dr = \frac{C \varepsilon^3 |y|}{|{\log \varepsilon}| } \int_{\frac{\delta_0}{4 \varepsilon |y|}}^{\frac{\delta_0}{3 \varepsilon |y|}}  \log \left( \frac{s + 1}{|s - 1|} \right) s \, ds \\
& \hspace{10pt} \le \frac{C \varepsilon^3 |y|}{|{\log \varepsilon}| } \left\{ \mathds{1}_{\big( {0, \frac{\delta_0}{8 \varepsilon} } \big)} \! (|y|) \int_{2}^{\frac{\delta_0}{3 \varepsilon |y|}} ds + \mathds{1}_{\big[ {\frac{\delta_0}{8 \varepsilon}, \frac{\delta_0}{\varepsilon}} \big]} \! (|y|) \int_{\frac{1}{4}}^{3} \log \left( \frac{4}{|s - 1|} \right) ds + \mathds{1}_{\big( {\frac{\delta_0}{\varepsilon}, +\infty} \big)} \! (|y|) \int_{0}^{\frac{\delta_0}{3 \varepsilon |y|}} s^2 \, ds \right\} \\
& \hspace{10pt} \le \frac{C \varepsilon^2}{|{\log \varepsilon}|},
\end{align*}
for all~$y \in \R$. With the same computation, we estimate
\begin{align*}
& \left| \int_{\R^2_+} \frac{Z_0(y', \lambda) \big( {h(y', \lambda) \Delta \varrho_2(y', \lambda) + 2 \nabla \varrho_2(y', \lambda) \cdot \nabla h(y', \lambda)} \big)}{(y - y')^2 + \lambda^2} \, \lambda \, dy' d\lambda \right| \\
& \hspace{200pt} \le \frac{C \varepsilon^2}{|{\log \varepsilon}|} \int_{B_{\frac{\delta_0}{3 \varepsilon}}^+ \setminus B_{\frac{\delta_0}{4 \varepsilon}}^+} \frac{\lambda \, dy' d\lambda}{(y - y')^2 + \lambda^2} \le \frac{C \varepsilon}{|{\log \varepsilon}|},
\end{align*}
while, arguing analogously, we get
\begin{align*}
& \left| \int_{\R^2_+} \frac{(1 - \varrho_1(y', \lambda)) \varrho_2(y', \lambda) \nabla h(y', \lambda) \cdot \nabla Z_0(y', \lambda)}{(y - y')^2 + \lambda^2} \, \lambda \, dy' d\lambda \right| \\
& \hspace{10pt} \le \frac{C}{|{\log \varepsilon}|} \int_{B_{\frac{\delta_0}{3 \varepsilon}}^+ \setminus B_{R}^+} \frac{\lambda \, dy' d\lambda}{\left( (y - y')^2 + \lambda^2 \right) \left( y'^2 + \lambda^2 \right)^{\frac{3}{2}}} = \frac{C}{|{\log \varepsilon}|} \int_R^{\frac{\delta_0}{3 \varepsilon}} \left( \int_0^\pi \frac{\sin \theta}{r^2 + y^2 - 2 r y \cos \theta} \, d\theta \right) \frac{dr}{r} \\
& \hspace{10pt} = \frac{C}{|{\log \varepsilon}| \, y^2} \int_{\frac{R}{|y|}}^{\frac{\delta_0}{3 \varepsilon|y|}} \log \left( \frac{s + 1}{|s - 1|} \right) \frac{ds}{s^2} \\
& \hspace{10pt} \le \frac{C}{|{\log \varepsilon}| \, y^2} \left\{ \mathds{1}_{\big( 0, \frac{R}{2} \big)} \! (|y|) \int_{\frac{R}{|y|}}^{+\infty} \frac{ds}{s^3} + \mathds{1}_{\big[ \frac{R}{2}, +\infty \big)} \! (|y|) \int_{\frac{1}{2}}^{+\infty} \log \left( \frac{s + 1}{|s - 1|} \right) \frac{ds}{s^2} + \mathds{1}_{\big[ 2 R, +\infty \big)} \! (|y|) \int_{\frac{R}{|y|}}^{\frac{1}{2}} \frac{ds}{s} \right\} \\
& \hspace{10pt} \le \frac{C}{|{\log \varepsilon}| \, y^2} \left\{ \mathds{1}_{\big( 0, \frac{R}{2} \big)} \! (|y|) \frac{y^2}{R^2} + \mathds{1}_{\big[ \frac{R}{2}, +\infty \big)} \! (|y|) + \mathds{1}_{\big[ 2 R, +\infty \big)} \! (|y|) \log \left( \frac{|y|}{2 R} \right) \right\} \le C \frac{(R + |y|)^{-1 - \sigma}}{ R^{1 - \sigma} |{\log \varepsilon}|},
\end{align*}
for any fixed~$\sigma \in (0, 1)$ and where~$C$ may now depend on~$\sigma$ as well. As a result, recalling \eqref{fdef} and \eqref{L0ztilde0repr}, we obtain
\begin{equation} \label{mainL0z0terms}
\begin{aligned}
L_\mu \tilde{z}_0(y) & = - \frac{1}{\pi} \int_{\R^2_+} \frac{\Big\{ {2 (1 - h) \nabla \varrho_1 \cdot \nabla Z_0 + \Big( { (1 - h) \Delta \varrho_1 - 2 \nabla \varrho_1 \cdot \nabla h } \Big) Z_0} \Big\}(y', \lambda)}{(y - y')^2 + \lambda^2} \, \lambda \, dy' d\lambda \\
& \quad\, + O \! \left( \frac{\varepsilon}{|\log \varepsilon|} \right) + O \! \left( \frac{(R + |y|)^{-1 - \sigma}}{ R^{1 - \sigma} |{\log \varepsilon}|} \right).
\end{aligned}
\end{equation}

\noindent We proceed to estimate the integral above. Arguing as before, we find
\begin{equation} \label{estwithonemoreR}
\begin{aligned}
& \left| \int_{\R^2_+} \frac{\big( {1 - h(y', \lambda)} \big) \nabla \varrho_1(y', \lambda) \cdot \nabla Z_0(y', \lambda)}{(y - y')^2 + \lambda^2} \, \lambda \, dy' d\lambda \right| \le \frac{C}{R^3 |{\log \varepsilon}|} \int_{B_{2 R}^+ \setminus B_R^+} \frac{\lambda \, dy' d\lambda}{(y - y')^2 + \lambda^2} \\
& \hspace{25pt} = \frac{C}{R^3 |{\log \varepsilon}|} \int_R^{2 R} \left( \int_0^\pi \frac{\sin \theta}{r^2 + y^2 - 2 r y \cos \theta} \, d\theta \right) \! r^2 \, dr = \frac{C |y|}{R^3 |{\log \varepsilon}|} \int_{\frac{R}{|y|}}^{\frac{2 R}{|y|}} \log \left( \frac{s + 1}{|s - 1|} \right) \! s \, ds \\
& \hspace{25pt} \le \frac{C |y|}{R^3 |{\log \varepsilon}|} \left\{ \mathds{1}_{\big( 0, \frac{R}{2} \big)} \! (|y|) \frac{R}{|y|} + \mathds{1}_{\big[ \frac{R}{2}, 4 R \big]} \! (|y|) + \mathds{1}_{\big( 4 R, +\infty \big)} \! (|y|) \frac{R^3}{|y|^3} \right\} \le \frac{C (R + |y|)^{-2}}{|{\log \varepsilon}|}
\end{aligned}
\end{equation}
and
\begin{equation} \label{estwithonelessR}
\begin{aligned}
& \left| \int_{\R^2_+}  \frac{\big( {1 - h(y', \lambda)} \big) Z_0(y', \lambda) \Delta \varrho_1(y', \lambda)}{(y - y')^2 + \lambda^2} \, \lambda \, dy' d\lambda \right| + \left| \int_{\R^2_+}  \frac{Z_0(y', \lambda) \nabla \varrho_1(y', \lambda) \cdot \nabla h(y', \lambda)}{(y - y')^2 + \lambda^2} \, \lambda \, dy' d\lambda \right| \\
& \hspace{188pt} \le \frac{C}{R^2 |{\log \varepsilon}|} \int_{B_{2 R}^+ \setminus B_R^+} \frac{\lambda \, dy' d\lambda}{(y - y')^2 + \lambda^2} \le \frac{C R (R + |y|)^{-2}}{|{\log \varepsilon}|}.
\end{aligned}
\end{equation}
From these and~\eqref{mainL0z0terms}, estimate~\eqref{estimateforL0z0} follows at once.

We now address the~$L^1$ estimate~\eqref{z0tildeL1est} and the lower bound~\eqref{energylowerbound}. Observing that
\begin{equation} \label{z0ests}
|z_0(y)| \le C \quad \mbox{and} \quad |z_0'(y)| \le \frac{C |y|}{1 + y^4}, \quad \mbox{ for all } y \in \R,
\end{equation}
we have
\begin{align*}
\int_\R |\tilde{z}_0(y)| \, dy & \le C \left( \int_0^{2 R} \, dy - \frac{1}{|{\log \varepsilon}|} \int_{2 R}^{\frac{\delta_0}{3 \varepsilon}} \log \left( \frac{3 \varepsilon y}{\delta_0} \right) dy \right) \\
& \le C \left( R - \frac{1}{\varepsilon |{\log \varepsilon}|} \int_{0}^1 \log t \, dt \right) \le C \left( R + \frac{1}{\varepsilon |{\log \varepsilon}|} \right),
\end{align*}
which is actually~\eqref{z0tildeL1est}, and
\begin{equation} \label{z0againstdecay}
\begin{aligned}
\int_\R \frac{|\tilde{z}_0(y)|}{(R + |y|)^{1 + \sigma}} \, dy & \le C \left( \frac{1}{R^{1 + \sigma}} \int_0^{2 R} dy - \frac{1}{|{\log \varepsilon}|} \int_{2 R}^{\frac{\delta_0}{3 \varepsilon}} \frac{\log \left( \frac{3 \varepsilon y}{\delta_0} \right)}{y^{1 + \sigma}} \, dy \right) \\
& \le C \left( \frac{1}{R^\sigma} - \frac{\varepsilon^\sigma}{|{\log \varepsilon}|} \int_{\frac{2 R \varepsilon}{\delta}}^1 \frac{\log t}{t^{1 + \sigma}} \, dt \right) \le \frac{C}{R^\sigma},
\end{aligned}
\end{equation}
where the last inequality follows after integration by parts. Consequently, we deduce from~\eqref{mainL0z0terms} that
\begin{equation} \label{z0againstLz0}
\begin{aligned}
& \hspace{-8pt} \int_\R \tilde{z}_0(y) L_\mu \tilde{z}_0(y) \, dy \\
& = - \frac{1}{\pi} \int_\R \tilde{z}_0(y) \bigg\{ { \int_{\R^2_+} \frac{\Big\{ {2 (1 - h) \nabla \varrho_1 \cdot \nabla Z_0 + \Big( { (1 - h) \Delta \varrho_1 - 2 \nabla \varrho_1 \cdot \nabla h } \Big) Z_0} \Big\}(y', \lambda)}{(y - y')^2 + \lambda^2} \, \lambda \, dy' d\lambda} \bigg\} \, dy \\
& \quad\, + O \! \left( \frac{1}{R |{\log \varepsilon}|} \right),
\end{aligned}
\end{equation}
if~$\varepsilon$ is small enough, in dependence of~$R$ and~$\delta_0$. We plan to integrate by parts the term involving the Laplacian of~$\varrho_1$. Set~$\Psi(y, \lambda) := \log |(y, \lambda)|$ and notice that~$\partial_\lambda \Psi(y, \lambda) = \frac{\lambda}{y^2 + \lambda^2}$ for all~$(y, \lambda) \in \R^2 \setminus \{ (0, 0) \}$. For~$y \in \R$ and~$t > 0$, set~$\C_t(y) := (y - t, y + t) \times (0, +\infty)$. We then write 
\begin{equation} \label{ItJtstart}
\int_\R \tilde{z}_0(y) \left\{ \int_{\R^2_+} \frac{\big\{ {(1 - h) Z_0 \Delta \varrho_1} \big\}(y', \lambda)}{(y - y')^2 + \lambda^2} \, \lambda \, dy' d\lambda \right\} dy = \lim_{t \searrow 0} \big( {Y_t + \Lambda_t} \big),
\end{equation}
where
\begin{align*}
Y_t & := \int_\R \tilde{z}_0(y) \bigg\{ {\int_{\R^2_+ \setminus \C_t(y)} \big( {1 - h(y', \lambda)} \big) Z_0(y', \lambda) \partial^2_{y' y'} \varrho_1(y', \lambda) \partial_\lambda \Psi(y - y', \lambda) \, dy' d\lambda} \bigg\} \, dy, \\
\Lambda_t & := \int_\R \tilde{z}_0(y) \bigg\{ {\int_{\R^2_+ \setminus \C_t(y)} \big( {1 - h(y', \lambda)} \big) Z_0(y', \lambda) \partial^2_{\lambda \lambda} \varrho_1(y', \lambda) \partial_\lambda \Psi(y - y', \lambda) \, dy' d\lambda} \bigg\} \, dy.
\end{align*}
Changing variables and integrating by parts with respect to~$y$ in~$Y_t$, we get
\begin{align*}
Y_t & = \int_{\R^2_+ \setminus \C_t(0)} \left\{ \int_\R \tilde{z}_0(y) \big( {1 - h(y + \tilde{y}, \lambda)} \big) Z_0(y + \tilde{y}, \lambda) \partial^2_{y y} \varrho_1(y + \tilde{y}, \lambda) \, dy \right\} \partial_\lambda \Psi(- \tilde{y}, \lambda) \, d\tilde{y} d\lambda \\
& = - \int_{\R^2_+ \setminus \C_t(0)} \bigg\{ {\int_\R \bigg( \tilde{z}_0(y) \Big( {\big( {1 - h(y + \tilde{y}, \lambda)} \big) \partial_y Z_0(y + \tilde{y}, \lambda) - \partial_y h(y + \tilde{y}, \lambda) Z_0(y + \tilde{y}, \lambda)} \Big)} \\
& \quad\,\, {+ \tilde{z}_0'(y) \big( {1 - h(y + \tilde{y}, \lambda)} \big) Z_0(y + \tilde{y}, \lambda) \bigg) \partial_{y} \varrho_1(y + \tilde{y}, \lambda) \, dy} \bigg\} \partial_\lambda \Psi(- \tilde{y}, \lambda) \, d\tilde{y} d\lambda.
\end{align*}
On the other hand, integrating by parts with respect to~$\lambda$ in~$\Lambda_t$,
\begin{align*}
\Lambda_t & = - \int_\R \tilde{z}_0(y) \bigg\{ {\int_{\R^2_+ \setminus \C_t(y)} \bigg( {- \partial_\lambda h(y', \lambda) Z_0(y', \lambda) \partial_\lambda \Psi(y - y', \lambda)}} \\
& \quad\,\, {{+ \big( {1 - h(y', \lambda)} \big) \Big( {\partial_\lambda Z_0(y', \lambda) \partial_\lambda \Psi(y - y', \lambda) + Z_0(y', \lambda) \partial^2_{\lambda \lambda} \Psi(y - y', \lambda)} \Big)} \bigg) \partial_{\lambda} \varrho_1(y', \lambda) \, dy' d\lambda} \bigg\} \, dy.
\end{align*}
Putting these two identities together, switching back variables in~$Y_t$, and recognizing the inner products between gradients, we obtain
\begin{align*}
Y_t + \Lambda_t & = - \int_\R \tilde{z}_0(y) \bigg\{ {\int_{\R^2_+ \setminus \C_t(y)} \Big\{ {\Big( {(1 - h) \nabla Z_0 - Z_0 \nabla h} \Big) \cdot \nabla \varrho_1} \Big\}(y', \lambda) \partial_\lambda \Psi(y - y', \lambda) \, dy' d\lambda} \bigg\} \, dy \\
& \quad\, - \int_\R \tilde{z}_0'(y) \bigg\{ {\int_{\R^2_+ \setminus \C_t(y)} \Big\{ {(1 - h) Z_0 \partial_{y'} \varrho_1} \Big\}(y', \lambda) \partial_\lambda \Psi(y - y', \lambda) \, dy' d\lambda} \bigg\} \, dy \\
& \quad\, - \int_\R \tilde{z}_0(y) \bigg\{ {\int_{\R^2_+ \setminus \C_t(y)} \Big\{ {(1 - h) Z_0 \partial_{\lambda} \varrho_1} \Big\}(y', \lambda) \partial^2_{\lambda \lambda} \Psi(y - y', \lambda) \, dy' d\lambda} \bigg\} \, dy.
\end{align*}
Exchanging the order of integration, noticing that~$\partial^2_{\lambda \lambda} \Psi = - \partial^2_{y y} \Psi$ in~$\R^2 \setminus \{ (0, 0) \}$, and integrating by parts with respect to~$y$, we rewrite the integral on the third line as
\begin{align*}
& \int_\R \tilde{z}_0(y) \bigg\{ {\int_{\R^2_+ \setminus \C_t(y)} \Big\{ {(1 - h) Z_0 \partial_{\lambda} \varrho_1} \Big\}(y', \lambda) \partial^2_{\lambda \lambda} \Psi(y - y', \lambda) \, dy' d\lambda} \bigg\} \, dy \\
& \hspace{50pt} = - \int_{\R^2_+} \Big\{ {(1 - h) Z_0 \partial_{\lambda} \varrho_1} \Big\}(y', \lambda) \bigg\{ {\int_{\R \setminus (y' - t, y' + t)} \tilde{z}_0(y) \partial^2_{y y} \Psi(y - y', \lambda) \, dy} \bigg\} \, dy' d\lambda \\
& \hspace{50pt} = \int_{\R^2_+} \Big\{ {(1 - h) Z_0 \partial_{\lambda} \varrho_1} \Big\}(y', \lambda) \bigg\{ {\int_{\R \setminus (y' - t, y' + t)} \tilde{z}_0'(y) \partial_{y} \Psi(y - y', \lambda) \, dy} \bigg\} \, dy' d\lambda \\
& \hspace{50pt} \quad + \int_{\R^2_+} \Big\{ {(1 - h) Z_0 \partial_{\lambda} \varrho_1} \Big\}(y', \lambda) \Big( {\tilde{z}_0(y' + t) + \tilde{z}_0(y' - t)} \Big) \partial_{y} \Psi(t, \lambda) \, dy' d\lambda,
\end{align*}
where for the last identity we also took advantage of the oddness of~$\partial_y \Psi(\,\cdot\,, \lambda)$. Observe that, by dominated convergence, the last integral goes to zero as~$t \searrow 0$. Indeed, its integrand converges to zero a.e.~in~$\R^2_+$ and, recalling the definition of~$\varrho_1$, it satisfies the uniform-in-$t$ bound
\begin{align*}
& \left| \Big\{ {(1 - h) Z_0 \partial_{\lambda} \varrho_1} \Big\}(y', \lambda) \Big( {\tilde{z}_0(y' + t) + \tilde{z}_0(y' - t)} \Big) \partial_{y} \Psi(t, \lambda) \right| \\
& \hspace{60pt} \le \frac{C}{R^2 |{\log \varepsilon}|} \frac{t \lambda}{t^2 + \lambda^2} \, \mathds{1}_{B^+_{2 R} \setminus B_{R}^+}(y', \lambda) \le \frac{C}{R^2 |{\log \varepsilon}|} \, \mathds{1}_{B^+_{2 R} \setminus B_{R}^+}(y', \lambda) \in L^1(\R^2_+),
\end{align*}
thanks to Young's inequality. As a result,~\eqref{ItJtstart} becomes
\begin{align*}
& \int_\R \tilde{z}_0(y) \left\{ \int_{\R^2_+} \frac{\big\{ {(1 - h) Z_0 \Delta \varrho_1} \big\}(y', \lambda)}{(y - y')^2 + \lambda^2} \, \lambda \, dy' d\lambda \right\} dy \\
& \hspace{30pt} = - \int_\R \tilde{z}_0(y) \bigg\{ {\int_{\R^2_+} \Big\{ {(1 - h) \nabla Z_0 \cdot \nabla \varrho_1 - Z_0 \nabla h \cdot \nabla \varrho_1} \Big\}(y', \lambda) \partial_\lambda \Psi(y - y', \lambda) \, dy' d\lambda} \bigg\} \, dy \\
& \hspace{30pt} \quad\, - \lim_{t \searrow 0} \int_\R \tilde{z}_0'(y) \bigg\{ {\int_{\R^2_+ \setminus \C_t(y)} \big( {1 - h(y', \lambda)} \big) Z_0(y', \lambda) \Big( {\partial_{y'} \varrho_1(y', \lambda) \partial_\lambda \Psi(y - y', \lambda)}} \\
& \hspace{30pt} \quad\,\, {{+ \partial_\lambda \varrho_1(y', \lambda) \partial_{y} \Psi(y - y', \lambda)} \Big) \, dy' d\lambda} \bigg\} \, dy.
\end{align*}
Recalling once again the definition of~$\varrho_1$, we compute
\begin{align*}
& \partial_{y'} \varrho_1(y', \lambda) \partial_\lambda \Psi(y - y', \lambda) + \partial_\lambda \varrho_1(y', \lambda) \partial_{y} \Psi(y - y', \lambda) \\
& \hspace{80pt} = \frac{1}{|(y', \lambda)| R} \, \frac{1}{(y - y')^2 + \lambda^2} \, \varrho' \! \left( \frac{|(y', \lambda)|}{R} - 1 \right) \left( y' \lambda + \lambda (y - y') \right) \\
& \hspace{80pt} = \frac{1}{|(y', \lambda)| R} \, \frac{\lambda y}{(y - y')^2 + \lambda^2} \, \varrho' \! \left( \frac{|(y', \lambda)|}{R} - 1 \right),
\end{align*}
and thus, also expressing~$\partial_\lambda \Psi$ explicitly,
\begin{equation} \label{Laplrho1termgood}
\begin{aligned}
& \int_\R \tilde{z}_0(y) \left\{ \int_{\R^2_+} \frac{\big\{ {(1 - h) Z_0 \Delta \varrho_1} \big\}(y', \lambda)}{(y - y')^2 + \lambda^2} \, \lambda \, dy' d\lambda \right\} dy \\
& \hspace{30pt} = - \int_\R \tilde{z}_0(y) \bigg\{ {\int_{\R^2_+} \frac{\Big\{ {(1 - h) \nabla Z_0 \cdot \nabla \varrho_1 - Z_0 \nabla h \cdot \nabla \varrho_1} \Big\}(y', \lambda)}{(y - y')^2 + \lambda^2} \, \lambda \, dy' d\lambda} \bigg\} \, dy \\
& \hspace{30pt} \quad\, - \frac{1}{R} \int_\R y \tilde{z}_0'(y) \bigg\{ {\int_{\R^2_+} \frac{\big( {1 - h(y', \lambda)} \big) Z_0(y', \lambda)}{|(y', \lambda)| \big( {(y - y')^2 + \lambda^2} \big)} \, \varrho' \! \left( \frac{|(y', \lambda)|}{R} - 1 \right) \lambda \, dy' d\lambda} \bigg\} \, dy.
\end{aligned}
\end{equation}
We now estimate this last term. Recalling the rightmost inequality in~\eqref{estwithonelessR}, for the inner integral we have the bound
\begin{align*}
\left| \int_{\R^2_+} \frac{\big( {1 - h(y', \lambda)} \big) Z_0(y', \lambda)}{|(y', \lambda)| \big( {(y - y')^2 + \lambda^2} \big)} \, \varrho' \! \left( \frac{|(y', \lambda)|}{R} - 1 \right) \lambda \, dy' d\lambda \right| & \le \frac{C}{R |{\log \varepsilon}|} \int_{B_{2 R}^+ \setminus B_R^+} \frac{\lambda \, dy' d\lambda}{(y - y')^2 + \lambda^2} \\
& \le \frac{C R^2 (R + |y|)^{-2}}{|{\log \varepsilon}|}.
\end{align*}
Since it also holds
\begin{align*}
\tilde{z}_0'(y) = z_0'(y) \Big\{ {\varrho_1 + (1 - \varrho_1) \varrho_2 h} \Big\}(y, 0) + z_0(y) \Big\{ {(1 - h) \partial_y \varrho_1 + h \partial_y \varrho_2 + (1 - \varrho_1) \varrho_2 \partial_y h} \Big\}(y, 0),
\end{align*}
using~\eqref{z0ests} we compute
\begin{align*}
& \left| \frac{1}{R} \int_\R y \tilde{z}_0'(y) \bigg\{ {\int_{\R^2_+} \frac{\big( {1 - h(y', \lambda)} \big) Z_0(y', \lambda)}{|(y', \lambda)| \big( {(y - y')^2 + \lambda^2} \big)} \, \varrho' \! \left( \frac{|(y', \lambda)|}{R} - 1 \right) \lambda \, dy' d\lambda} \bigg\} \, dy \right| \\
& \hspace{140pt} \le \frac{C R}{|{\log \varepsilon}|} \Bigg( {\frac{1}{R^2} \int_0^{2 R} \frac{y^2}{1 + y^4} \, dy - \frac{1}{|{\log \varepsilon}|} \int_{R}^{\frac{\delta_0}{3 \varepsilon}} \frac{\log \left(\frac{3 \varepsilon y}{\delta_0} \right)}{y^4} \, dy} \\
& \hspace{140pt} \quad\, {+ \frac{1}{R^2 |{\log \varepsilon}|} \int_R^{2 R} dy + \frac{\varepsilon}{|{\log \varepsilon}|} \int_{\frac{\delta_0}{4 \varepsilon}}^{\frac{\delta_0}{3 \varepsilon}} \frac{dy}{y} + \frac{1}{|{\log \varepsilon}|} \int_R^{\frac{\delta_0}{3 \varepsilon}} \frac{dy}{y^2}} \Bigg) \\
& \hspace{140pt} \le \frac{C}{|{\log \varepsilon}|} \left( \frac{1}{R} + \frac{1}{|{\log \varepsilon}|} + \frac{\varepsilon R}{|{\log \varepsilon}|} \right) \le \frac{C}{R |{\log \varepsilon}|}.
\end{align*}
Putting together~\eqref{z0againstLz0},~\eqref{Laplrho1termgood}, and this, we arrive at
\begin{align*}
\int_\R \tilde{z}_0(y) L_\mu \tilde{z}_0(y) \, dy & = - \frac{1}{\pi} \int_\R \tilde{z}_0(y) \bigg\{ { \int_{\R^2_+} \frac{\Big\{ {(1 - h) \nabla \varrho_1 \cdot \nabla Z_0 - Z_0 \nabla \varrho_1 \cdot \nabla h} \Big\}(y', \lambda)}{(y - y')^2 + \lambda^2} \, \lambda \, dy' d\lambda} \bigg\} \, dy \\
& \quad\, + O \! \left( \frac{1}{R |{\log \varepsilon}|} \right).
\end{align*}
Recalling~\eqref{estwithonemoreR} and~\eqref{z0againstdecay}, we see that the integral involving the term~$(1 - h) \nabla \varrho_1 \cdot \nabla Z_0$ is also of order~$(R |{\log \varepsilon}|)^{-1}$. Thus,
$$
\int_\R \tilde{z}_0(y) L_\mu \tilde{z}_0(y) \, dy = \frac{1}{\pi} \int_\R \tilde{z}_0(y) \bigg\{ { \int_{\R^2_+} \frac{Z_0(y', \lambda) \nabla \varrho_1(y', \lambda) \cdot \nabla h(y', \lambda)}{(y - y')^2 + \lambda^2} \, \lambda \, dy' d\lambda} \bigg\} \, dy + O \! \left( \frac{1}{R |{\log \varepsilon}|} \right).
$$
Note now that, by~\eqref{estwithonelessR},
$$
\left| \int_{-2 \mu}^{2 \mu} \tilde{z}_0(y) \bigg\{ { \int_{\R^2_+} \frac{Z_0(y', \lambda) \nabla \varrho_1(y', \lambda) \cdot \nabla h(y', \lambda)}{(y - y')^2 + \lambda^2} \, \lambda \, dy' d\lambda} \bigg\} \, dy \right| \le \frac{C}{R |{\log \varepsilon}|}.
$$
On the other hand, observe that~$\tilde{z}_0 \ge \frac{1}{4 \mu}$ on $(-2R,2R) \setminus (-2 \mu, 2 \mu)$ (if~$\epsilon$ is small enough). In addition, as both~$\varrho_1$ and~$h$ are radially non-increasing, we have that~$\nabla \varrho_1 \cdot \nabla h = |\nabla \varrho_1| |\nabla h| \ge 0$ in~$\R^2_+$. As~$Z_0 \ge \frac{1}{2 \mu}$ on the support of~$\nabla \varrho_1$ (if~$R$ is large enough) and~$|\nabla \varrho_1| \ge \frac{1}{2 R}$ in~$B_{\frac{7R}{4}} \setminus B_{\frac{5 R}{4}}$ (thanks to our hypotheses on~$\varrho$), we have
\begin{align*}
& \int_{\R \setminus (- 2 \mu, 2 \mu)} \tilde{z}_0(y) \bigg\{ { \int_{\R^2_+} \frac{Z_0(y', \lambda) \nabla \varrho_1(y', \lambda) \cdot \nabla h(y', \lambda)}{(y - y')^2 + \lambda^2} \, \lambda \, dy' d\lambda} \bigg\} \, dy \\
& \hspace{100pt} \ge \frac{1}{2^4 \mu^2 R \log \left( \frac{\delta_0}{3 \varepsilon R} \right)} \int_R^{2 R} \bigg\{ { \int_{B^+_{\frac{7R}{4}} \setminus B_{\frac{5R}{4}}^+} \frac{\lambda \, dy' d\lambda}{|(y', \lambda)||(y', \lambda) - (y, 0)|^2}} \bigg\} \, dy \\
& \hspace{100pt} \ge \frac{1}{2^8 \mu_0^2 R^3 \log \left( \frac{\delta_0}{3 \varepsilon R} \right)} \int_{\frac{5 R}{4}}^{\frac{7 R}{4}} \left( \int_0^\pi \sin \theta \, d\theta \right) r^2 dr \ge \frac{c}{|{\log \varepsilon}|},
\end{align*}
for some small constant~$c > 0$ depending only on~$\mu_0$. Combining the last three formulas, we easily conclude that~\eqref{energylowerbound} holds true, provided we take~$R$ sufficiently large in dependence of~$\sigma$,~$\delta$, and~$\mu_0$.
\end{proof}

Thanks to the previous result, we may now prove the desired \textit{a priori} estimate for solutions to~\eqref{linear-eq} which satisfy the orthogonality conditions~\eqref{phiorthtoZ1j}.

\begin{lemma} \label{2ndLinftyestlem}
Let $\delta_0,\, \sigma \in (0,1)$. There exist $\varepsilon_1 \in (0, 1)$ and~$C_1 > 0$ for which the following holds true:  if $\xi \in \mathcal{I}_{\delta_0}$, $\epsilon \in (0,\varepsilon_1]$, $g \in L^\infty_{\star,\sigma}(I_\varepsilon)$ and~$\phi$ is a bounded weak solution to~\eqref{linear-eq} satisfying~\eqref{phiorthtoZ1j}, then
\begin{equation} \label{logepsestforphi}
\| \phi \|_{L^\infty(\R)} \le C_1 \big( {\log \frac{1}{\epsilon}} \big) \| g \|_{\star,\,\sigma}.
\end{equation}
\end{lemma}


\begin{proof}
Let~$R \ge \bar{R} + 1$. For~$j = 1, \ldots, m$, let~$\tilde{z}_{0 j} := \tilde{z}_0^{(\mu_j)}(\,\cdot\, - \eta_j)$, with~$\tilde{z}_0^{(\mu_j)}$ being the function constructed in Lemma~\ref{ztilde0lem}. Note that this function exists, provided that~$R$ is large enough, in dependence of~$\delta_0$,~$\sigma$--recall~\eqref{mu-bounded both sides}--and that~$\varepsilon$ is small enough, in dependence of~$\delta_0$,~$\sigma$, and~$R$. We know that~$\tilde{z}_{0 j} \in C^\infty_c(\R)$ is symmetric with respect to~$\eta_j$ and that it satisfies~$\tilde{z}_{0 j} = Z_{0 j}$ in~$[\eta_j - R, \eta_j + R]$,~$0 \le \tilde{z}_{0 j} \le C$ in~$\R \setminus [\eta_j - R, \eta_j + R]$, and~$\tilde{z}_{0 j} = 0$ in~$\R \setminus \big( {\eta_j - \frac{\delta_0}{3 \varepsilon}, \eta_j + \frac{\delta_0}{3 \varepsilon}} \big)$, for some constant~$C \ge 1$ depending only on~$m$,~$D$,~$\delta_0$, $\sup_{I} \kappa$, and $\inf_{I} \kappa$. Furthermore, observe that
$$
L \tilde{z}_{0 j}(y) = L_{\mu_j, \eta_j} \tilde{z}_{0 j}(y) + \left( \frac{2 \mu_j}{\mu_j^2 + (y - \eta_j)^2} - W(y) \right) \! \tilde{z}_{0 j}(y), \quad \mbox{ for all } y \in \R.
$$
Recalling~\eqref{Wasperturb}--\eqref{thetabounds} and the fact that~$\tilde{z}_{0 j}$ is supported in~$\big( {\eta_j - \frac{\delta_0}{3 \varepsilon}, \eta_j + \frac{\delta_0}{3 \varepsilon}} \big)$, we see that
\begin{align*}
\left| \left( \frac{2 \mu_j}{\mu_j^2 + (y - \eta_j)^2} - W(y) \right) \! \tilde{z}_{0 j}(y) \right| & \le C \mathds{1}_{\big( {\eta_j - \frac{\delta_0}{3 \varepsilon}, \eta_j + \frac{\delta_0}{3 \varepsilon}} \big)} \! (y) \left( \sum_{k \ne j}\frac{2 \mu_k}{\mu_k^2 + (y - \eta_k)^2} + \varepsilon \sum_{k = 1}^m \frac{1}{1 + |y - \eta_k|} \right) \\
& \le C \varepsilon^{1 - \sigma} \mathds{1}_{\big( {\eta_j - \frac{\delta_0}{3 \varepsilon}, \eta_j + \frac{\delta_0}{3 \varepsilon}} \big)} \! (y) \left( \varepsilon + \sum_{k = 1}^m \frac{1}{\big( {1 + |y - \eta_k|} \big)^{1 + \sigma}} \right),
\end{align*}
with~$C$ now depending also on~$\sigma$. As~$L_{\mu_j, \eta_j} \tilde{z}_{0 j} = \big( {L_{\mu_j} \tilde{z}_0^{(\mu_j)}} \big)(\,\cdot\, - \eta_j)$, we deduce from~\eqref{z0tildeL1est}--\eqref{energylowerbound} that
\begin{equation} \label{estsfortilde0j}
\| \tilde{z}_{0 j} \|_{L^1(\R)} \le \frac{C}{\varepsilon |{\log \varepsilon}|}, \quad \| L \tilde{z}_{0 j} \|_{\star,\,\sigma} \le \frac{C}{|{\log \varepsilon}|}, \quad \mbox{and} \quad \int_\R \tilde{z}_{0 j}(y) L \tilde{z}_{0 j}(y) \, dy \ge \frac{1}{C |{\log \varepsilon}|},
\end{equation}
if~$\varepsilon$ is small enough. Notice, in addition, that~\eqref{estimateforL0z0} also gives
\begin{equation} \label{otheroestfortilde0j}
|L \tilde{z}_{0 j}(y)| \le \frac{C \varepsilon}{|{\log \varepsilon}|}, \quad \mbox{ for all } y \in \R \setminus \left( \eta_j - \frac{\delta_0}{3 \varepsilon}, \eta_j + \frac{\delta_0}{3 \varepsilon} \right).
\end{equation}

Let now~$\phi$ be a bounded solution of~\eqref{linear-eq} satisfying the orthogonality conditions~\eqref{phiorthtoZ1j}. We modify it to obtain a new function orthogonal to both the~$Z_{0 j}$'s and the~$Z_{1 j}$'s. Let
$$
\widetilde{\phi} := \phi + \sum_{k = 1}^m d_k \tilde{z}_{0 k},
$$
for some~$d_1, \ldots, d_m \in \R$ to be determined. Note that~$\widetilde{\phi}$ is always orthogonal to the~$Z_{1 j}$'s, thanks to~\eqref{phiorthtoZ1j} and the properties of~$\tilde{z}_{0 k}$, namely, its symmetry with respect to~$\eta_k$ and the fact that it is supported in~$\big( {\eta_k - \frac{\delta_0}{3 \varepsilon}, \eta_k + \frac{\delta_0}{3 \varepsilon}} \big)$. The choice
$$
d_j := - \dfrac{\displaystyle \int_{\R} \phi \chi_j Z_{0 j}}{\displaystyle \int_\R \chi_j Z_{0 j}^2}, \quad \mbox{ for } j = 1, \ldots, m,
$$
yields that~$\widetilde{\phi}$ is also orthogonal to the~$Z_{0 j}$'s. Here, we used that~$\tilde{z}_{0 j} = Z_{0 j}$ on the support of~$\chi_j$ (as we took~$R \ge \bar{R} + 1$) and that~$\tilde{z}_{0 k}$ and~$\chi_j$ have disjoint supports if~$k \ne j$. In addition to this,~$\widetilde{\phi} = 0$ outside of~$I_\varepsilon$, since~$\xi \in \mathcal{I}_{\delta_0}$. Hence, it satisfies
$$
\left\{
\begin{aligned}
L \widetilde{\phi} & = g + \sum_{k = 1}^m d_k L \tilde{z}_{0 k}, \quad && \mbox{ in } I_\varepsilon, \\
\widetilde{\phi} & = 0, \quad && \mbox{ in } \R \setminus I_\varepsilon, \\
\int_\R \widetilde{\phi} \chi_j Z_{i j} & = 0, \quad && \mbox{ for all } i = 0, 1 \mbox{ and } j = 1, \ldots, m,
\end{aligned}
\right.
$$
and we may apply Lemma~\ref{1stLinftyestlem} to deduce that
\begin{equation} \label{Linftyestforphitilde}
\big\| {\widetilde{\phi} \,} \big\|_{L^\infty(\R)} \le C_0 \left( \| g \|_{\star,\,\sigma} + \sum_{k = 1}^m |d_k| \, \big\| {L \tilde{z}_{0 k}} \big\|_{\star,\,\sigma} \right),
\end{equation}
provided~$\varepsilon$ is sufficiently small. We now test the equation for~$\widetilde{\phi}$ against~$\tilde{z}_{0 j}$ obtaining that
\begin{equation} \label{testingagainstZtilde}
\int_{\R} \widetilde{\phi} \, L \tilde{z}_{0 j} = \int_\R g \tilde{z}_{0 j} + \sum_{k = 1}^m d_k \int_\R \tilde{z}_{0 j} L \tilde{z}_{0 k}, \quad \mbox{ for all } j = 1, \ldots, m.
\end{equation}
Taking advantage of~\eqref{estsfortilde0j}, we estimate
\begin{align*}
\left| \int_{\R} \widetilde{\phi} \, L \tilde{z}_{0 j} \right| & \le \big\| {\widetilde{\phi} \,} \big\|_{L^\infty(\R)} \, \big\| {L \tilde{z}_{0 j}} \big\|_{\star,\,\sigma} \int_{- \frac{D}{\varepsilon}}^{\frac{D}{\varepsilon}} \Bigg\{ \varepsilon + \sum\limits_{k = 1}^{m} \frac{1}{(1 + |y - \eta_k|)^{1 + \sigma}} \Bigg\} \, dy \le \frac{C}{|{\log \varepsilon}|} \, \big\| {\widetilde{\phi} \,} \big\|_{L^\infty(\R)}, \\
\left| \int_{\R} g \tilde{z}_{0 j} \right| & \le \| g \|_{\star,\,\sigma} \big\| {\tilde{z}_{0 j}} \big\|_{L^\infty(\R)} \int_{- \frac{D}{\varepsilon}}^{\frac{D}{\varepsilon}} \Bigg\{ \varepsilon + \sum\limits_{k = 1}^{m} \frac{1}{(1 + |y - \eta_k|)^{1 + \sigma}} \Bigg\} \, dy \le C \| g \|_{\star,\,\sigma},
\end{align*}
for all~$j = 1, \ldots, m$, while, using~\eqref{otheroestfortilde0j} as well,
\begin{align*}
\left| \int_{\R} \tilde{z}_{0 j} L \tilde{z}_{0 k} \right| \le \int_{\eta_j - \frac{\delta_0}{3 \varepsilon}}^{\eta_j + \frac{\delta_0}{3 \varepsilon}} |\tilde{z}_{0 j}| |L \tilde{z}_{0 k}| \le \| L \tilde{z}_{0 k} \|_{L^\infty \big( {\R \setminus \big( {\eta_k - \frac{\delta_0}{3 \varepsilon}, \eta_k + \frac{\delta_0}{3 \varepsilon}} \big)} \big)} \| \tilde{z}_{0 j} \|_{L^1(\R)} \le \frac{C}{|{\log \varepsilon}|^2},
\end{align*}
for all~$k \ne j$. Combining these three estimates with~\eqref{Linftyestforphitilde} and~\eqref{testingagainstZtilde}, we easily arrive at
$$
|d_j| \left| \int_\R \tilde{z}_{0 j} \, L \tilde{z}_{0 j} \right| \le C \left( \| g \|_{\star,\,\sigma} + \frac{1}{|{\log \varepsilon}|^2} \sum_{k = 1}^m |d_k| \right), \quad \mbox{ for all } j = 1, \ldots, m.
$$
Summing up in~$j$, multiplying both sides by~$|{\log \varepsilon}|$, and recalling the last inequality in~\eqref{estsfortilde0j}, we get
$$
\sum_{k = 1}^m |d_k| \le C \left( |{\log \varepsilon}| \, \| g \|_{\star,\,\sigma} + \frac{1}{|{\log \varepsilon}|} \sum_{k = 1}^m |d_k| \right).
$$
By taking~$\varepsilon$ sufficiently small, we can reabsorb in the left-hand side the second term on the right and obtain
$$
\sum_{k = 1}^m |d_k| \le C |{\log \varepsilon}| \, \| g \|_{\star,\,\sigma}.
$$
Going back to the definition of~$\widetilde{\phi}$ and using again~\eqref{estsfortilde0j} and~\eqref{Linftyestforphitilde}, we are immediately led to the desired estimate~\eqref{logepsestforphi}.
\end{proof}

\subsection{Solvability of problem (\ref{mainprojlinearprob})-(\ref{phiorthtoZ1j}). Proof of Proposition \ref{mainlinearprop}} $ $ 
\medbreak
\noindent We finally tackle here the proof of Proposition~\ref{mainlinearprop}. In order to do it, we need one more technical lemma.

\begin{lemma} \label{ztilde1lem}
Let~$\delta_0 \in (0, 1)$,~$\mu_0 \ge 1$, and~$\varepsilon \in \big( {0, \frac{\delta_0}{8}} \big]$. For any~$\mu \in \big[ {\frac{1}{\mu_0}, \mu_0} \big]$, there exists an odd function~$\tilde{z}_1 = \tilde{z}_1^{(\mu)} \in C^\infty_c(\R)$ satisfying
\begin{equation} \label{tildez1props}
\tilde{z}_1 = Z_{1, \mu, 0} \mbox{ in } \left[ - \frac{\delta_0}{4 \varepsilon}, \frac{\delta_0}{4 \varepsilon} \right], \quad 0 \le \tilde{z}_1 \le Z_{1, \mu, 0} \mbox{ in } \left( \frac{\delta_0}{4 \varepsilon}, +\infty \right), \quad \tilde{z}_1 = 0 \mbox{ in } \R \setminus \left(- \frac{\delta_0}{3 \varepsilon}, \frac{\delta_0}{3 \varepsilon} \right),
\end{equation}
and
\begin{equation} \label{L0tildez1est}
|L_{\mu} \tilde{z}_1(y)| \le \frac{C \varepsilon}{1 + |y|},  \quad \mbox{ for all } y \in \R,
\end{equation}
for some constant~$C > 0$ depending only on $\mu_0$ and $\delta_0$.
\end{lemma}
\begin{proof}
Let~$\varrho_2 \in C^\infty_c(\R)$ be the restriction to~$\R = \partial \R^2_+$ of the cutoff function considered in the proof of Lemma~\ref{ztilde0lem}, fulfilling in particular~$\varrho_2(- y) = \varrho_2(y)$ for all~$y \in \R$,
$$
\rho_2 = 1 \mbox{ in }  \left[ - \frac{\delta_0}{4 \varepsilon}, \frac{\delta_0}{4 \varepsilon} \right], \quad \rho_2 = 0 \mbox{ in } \R \setminus  \left( - \frac{\delta_0}{3 \varepsilon}, \frac{\delta_0}{3 \varepsilon} \right), \quad |\rho_2^{(j)}| \le C_2 \varepsilon^j \mbox{ in } \R,
$$
for~$j = 1, 2$ and for some constant~$C_2 > 0$ depending only on~$\delta_0$. Define~$\tilde{z}_1 := \varrho_2 Z_{1, \mu, 0}$. It is clear that~$\tilde{z}_1$ satisfies~\eqref{tildez1props}.

We now show the validity of~\eqref{L0tildez1est}. Of course, we can restrict ourselves to verify it only for~$y > 0$. First, recall that~$Z_{1, \mu, 0}$ is odd and observe that it satisfies
$$
|Z_{1, \mu, 0}(y)| \le \frac{C}{1 + |y|} \quad \mbox{ and } \quad \big| {Z_{1, \mu, 0}(y + \tilde{y}) - Z_{1, \mu, 0}(y - \tilde{y})} \big| \le \frac{C \, |\tilde{y}|}{\big( {1 + |y + \tilde{y}|} \big) \big( {1 + |y - \tilde{y}|} \big)},
$$
for all~$y, \tilde{y} \in \R$ and for some constant~$C > 0$ depending only on~$\mu_0$. Next, using that~$L_{\mu} Z_{1, \mu, 0} = 0$ in~$\R$, we write  
\begin{equation} \label{L0ztilde1}
L_\mu \tilde{z}_1(y) = \frac{1}{\pi}\,\PV \int_{\R} \frac{Z_{1, \mu, 0}(y') \big( {\varrho_2(y) - \varrho_2(y')} \big)}{(y - y')^2} \, dy', \quad \mbox{ for all } y > 0.
\end{equation}
We distinguish between the three cases~$y \in (0, 1]$,~$y \in \big( {1, \frac{\delta_0}{2 \varepsilon}} \big)$, and~$y \in \big[ {\frac{\delta_0}{2 \varepsilon}, +\infty} \big)$. To deal with the first situation, we change variables and notice that~$\varrho_2'(y) = 0$, expressing~$L_\mu \tilde{z}_1(y)$ as
$$
L_\mu \tilde{z}_1(y) = - \frac{1}{\pi}\, \int_\R \frac{Z_{1, \mu, 0}(y + \tilde{y}) \big( { \varrho_2(y + \tilde{y}) - \varrho_2(y) - \varrho_2'(y) \tilde{y} \mathds{1}_{(-2, 2)}(\tilde{y})} \big)}{\tilde{y}^2} \, d\tilde{y}.
$$
Then, using the properties of~$\varrho_2$ and~$Z_{1, \mu, 0}$ we estimate that
\begin{equation} \label{L0tildezest-1}
|L_\mu \tilde{z}_1(y)| \le C \varepsilon^2 \int_{-2}^2 d\tilde{y} + C \varepsilon \int_{\R \setminus (-2, 2)} \frac{d\tilde{y}}{|\tilde{y} + y| |\tilde{y}|} \le C \varepsilon^2 + C \varepsilon \int_2^{+\infty} \frac{d\tilde{y}}{\tilde{y}^2} \le C \varepsilon, \quad \mbox{ for all } y \in (0, 1].
\end{equation}

To address the latter two cases, we go back to~\eqref{L0ztilde1} and, taking advantage of the symmetry properties of~$Z_{1, \mu, 0}$ and~$\varrho_2$, we rewrite~$L_\mu \tilde{z}_1(y)$ as
\begin{align*}
L_\mu \tilde{z}_1(y) & = \frac{1}{\pi}\, \PV \int_{0}^{+\infty} \frac{Z_{1, \mu, 0}(y') \big( {\varrho_2(y) - \varrho_2(y')} \big)}{(y - y')^2} \, dy' + \frac{1}{\pi}\, \int_0^{+\infty} \frac{Z_{1, \mu, 0}(- y'') \big( {\varrho_2(y) - \varrho_2(-y'')} \big)}{(y + y'')^2} \, dy'' \\
& = \frac{1}{\pi}\, \PV \int_0^{+\infty} Z_{1, \mu, 0}(y') \big( {\varrho_2(y) - \varrho_2(y')} \big) K(y, y') \, dy'
\end{align*}
for every~$y > 0$ and with~$K(y, y') := \frac{4 y y'}{(y - y')^2 (y + y')^2}$. The case~$y \in \big[ {\frac{\delta_0}{2 \varepsilon}, +\infty} \big)$ is then easily handled from here. Indeed,~$\varrho_2(y) = 0$ and therefore
\begin{equation} \label{L0tildezest-2}
|L_\mu \tilde{z}_1(y)| = \frac{1}{\pi} \left| \int_0^{\frac{\delta_0}{3 \varepsilon}} Z_{1, \mu, 0}(y') \varrho_2(y') K(y, y') \, dy' \right| \le \frac{C}{y^3} \int_0^{\frac{\delta_0}{3 \varepsilon}} dy' \le \frac{C}{y^2}, \quad \mbox{ for all } y \in \left[ \frac{\delta_0}{2 \varepsilon}, +\infty \right) \!,
\end{equation}
where~$C$ depends now on~$\delta_0$ as well. To deal with~$y \in \big( {1, \frac{\delta_0}{2 \varepsilon}} \big)$, we change variables and symmetrize, obtaining that
$$
L_\mu \tilde{z}_1(y) = \frac{1}{\pi}\, \big( \ell_1(y) + \ell_2(y) + \ell_3(y) \big),
$$
with
\begin{align*}
\ell_1(y) & := - \int_{-\frac{y}{2}}^{\frac{y}{2}} Z_{1, \mu, 0}(y + \tilde{y}) \big( { \varrho_2(y + \tilde{y}) - \varrho_2(y) - \varrho_2'(y) \tilde{y}} \big) K(y, y + \tilde{y}) \, d\tilde{y}, \\
\ell_2(y) & := \int_{\big( {-y, - \frac{y}{2}} \big) \cup \big( \frac{y}{2}, +\infty \big)} Z_{1, \mu, 0}(y + \tilde{y}) \big( {\varrho_2(y)- \varrho_2(y + \tilde{y})} \big) K(y, y + \tilde{y}) \, d\tilde{y}, \\
\ell_3(y) & := \frac{\varrho_2'(y)}{2} \int_{-\frac{y}{2}}^{\frac{y}{2}} \Big( {Z_{1, \mu, 0}(y - \tilde{y}) K(y, y - \tilde{y}) - Z_{1, \mu, 0}(y + \tilde{y}) K(y, y + \tilde{y})} \Big) \tilde{y} \, d\tilde{y}.
\end{align*}
We readily estimate
$$
|\ell_1(y)| \le C \varepsilon^2 \| Z_{1, \mu, 0} \|_{L^\infty \big( \frac{y}{2}, \frac{3 y}{2} \big)} \int_{-\frac{y}{2}}^{\frac{y}{2}} \frac{y (y + \tilde{y})}{(2 y + \tilde{y})^2} \, d\tilde{y} \le C \varepsilon^2
$$
and
$$
|\ell_2(y)| \le C \varepsilon \int_{\big( {-y, - \frac{y}{2}} \big) \cup \big( \frac{y}{2}, +\infty \big)} \frac{y |y + \tilde{y}|}{\big( {1 + |y + \tilde{y}|} \big) |\tilde{y}| (2 y + \tilde{y})^2} \, d\tilde{y} \le C \varepsilon \int_{\frac{y}{2}}^{+\infty} \frac{d\tilde{y}}{\tilde{y}^2} \le \frac{C \varepsilon}{y}.
$$
To control~$\ell_3(y)$, we observe that, for all~$y \in \big( {1, \frac{\delta_0}{2 \varepsilon}} \big)$ and~$\tilde{y} \in \big( {- \frac{y}{2}}, \frac{y}{2} \big)$, it holds
\begin{align*}
& \big| {Z_{1, \mu, 0}(y - \tilde{y}) K(y, y - \tilde{y}) - Z_{1, \mu, 0}(y + \tilde{y}) K(y, y + \tilde{y})} \big| \\
& \hspace{20pt} \le |Z_{1, \mu, 0}(y - \tilde{y})| \big| {K(y, y - \tilde{y}) - K(y, y + \tilde{y})} \big| + K(y, y + \tilde{y}) \big| {Z_{1, \mu, 0}(y - \tilde{y}) - Z_{1, \mu, 0}(y + \tilde{y})} \big| \\
& \hspace{20pt} \le \frac{C y}{\big( {1 + y - \tilde{y}} \big) (2 y + \tilde{y})^2} \left( \frac{|\tilde{y}|}{(2 y - \tilde{y})^2} + \frac{y + \tilde{y}}{|\tilde{y}| \big( {1 + y + \tilde{y}} \big)} \right) \le \frac{C}{|\tilde{y}| y^2}.
\end{align*}
Consequently,
$$
|\ell_3(y)| \le \frac{C \varepsilon}{y^2} \int_{-\frac{y}{2}}^{\frac{y}{2}} d\tilde{y} \le \frac{C \varepsilon}{y},
$$
and we end up with
$$
|L_\mu \tilde{z}_1(y)| \le \frac{C \varepsilon}{y} \quad \mbox{for all } y \in \left( 1, \frac{\delta_0}{2 \varepsilon} \right).
$$
Combining this last estimate with~\eqref{L0tildezest-1} and~\eqref{L0tildezest-2}, we arrive at the claimed~\eqref{L0tildez1est}.
\end{proof}

We can now deal with the

\begin{proof}[Proof of Proposition~\ref{mainlinearprop}]
We begin by establishing the a priori estimate~\eqref{mainlinearaprioriest}. If~$\phi$ solves~\eqref{mainprojlinearprob}--\eqref{phiorthtoZ1j} for some~$c_1, \ldots, c_m \in \R$, then, by Lemma~\ref{2ndLinftyestlem}, it follows that
\begin{equation} \label{phiesttechprop}
\| \phi \|_{L^\infty(\R)} \le C_1 |{\log \varepsilon}| \left( \| g \|_{\star,\,\sigma} + \sum_{j = 1}^m |c_j| \, \big\| {\chi_j Z_{1 j}} \big\|_{\star,\,\sigma} \right) \le C |{\log \varepsilon}| \left( \| g \|_{\star,\,\sigma} + \sum_{j = 1}^m |c_j| \right),
\end{equation}
for some constant~$C \ge 1$, independent of $\epsilon$ and $\xi$, and provided~$\varepsilon$ is small enough.

We now estimate the constants~$|c_j|$. To do it, we let~$\tilde{z}_1^{(\mu_k)}$ be the function constructed in Lemma~\ref{ztilde1lem} and test the equation for~$\phi$ against~$\tilde{z}_{1 k} := \tilde{z}_1^{(\mu_k)}(\,\cdot\, - \eta_k)$. Taking advantage of the properties of~$\tilde{z}_{1 k}$--namely, that~$\supp(\tilde{z}_{1 k}) \cap \supp(\chi_j) = \varnothing$ for every~$k \ne j$ and that~$\tilde{z}_{1 k} = Z_{1 k}$ in~$\supp(\chi_k)$--, we get
$$
\int_\R \phi L \tilde{z}_{1 k} = \int_\R g \tilde{z}_{1 k} + \sum_{j = 1}^{m} c_j \int_\R \chi_j Z_{1 j} \tilde{z}_{1 k} = \int_\R g \tilde{z}_{1 k} + c_k \int_\R \chi_k Z_{1 k}^2.
$$
That is,
$$
c_k = \dfrac{\displaystyle \int_\R \phi L \tilde{z}_{1 k} - \int_\R g \tilde{z}_{1 k}}{\displaystyle \int_\R \chi_k Z_{1 k}^2}, \quad \mbox{ for all } k = 1, \ldots, m.
$$
Notice that
$$
\int_\R \chi_k Z_{1 k}^2 = \int_{- \bar{R} - 1}^{\bar{R}+ 1} \chi Z_{1, \mu_k, 0}^2 \ge \frac{1}{C}
$$
and
$$
\left| \int_\R g \tilde{z}_{1 k} \right| \le \| g \|_{\star,\,\sigma} \| \tilde{z}_{1 k} \|_{L^\infty(\R)} \int_{\eta_k - \frac{\delta_0}{3 \varepsilon}}^{\eta_k - \frac{\delta_0}{3 \varepsilon}} \left( \varepsilon + \sum_{j = 1}^m \frac{1}{\big( {1 + |y - \eta_j|} \big)^{1 + \sigma}} \right) dy \le C \| g \|_{\star,\,\sigma}.
$$
On the other hand,
$$
L \tilde{z}_{1 k}(y) = L_{\mu_k, \eta_k} \tilde{z}_{1 k}(y) + \left( \frac{2 \mu_k}{\mu_k^2 + (y - \eta_k)^2} - W(y) \right) \tilde{z}_{1 k}(y), \quad \mbox{ for all } y \in \R,
$$
so that, by~\eqref{Wasperturb}--\eqref{thetabounds}, and Lemma~\ref{ztilde1lem}, we have
\begin{align*}
|L \tilde{z}_{1 k}(y)| & \le |L_{\mu_k, \eta_k} \tilde{z}_{1 k}(y)| + \left| \frac{2 \mu_k}{\mu_k^2 + (y - \eta_k)^2} - W(y) \right| |\tilde{z}_{1 k}(y)| \\
& \le \big| {\big( {L_{\mu_k} \tilde{z}_1^{(\mu_k)}} \big)(y - \eta_k)} \big| + \left( \sum_{j \ne k} \frac{2 \mu_j}{\mu_j^2 + (y - \eta_j)^2} + |\theta(y)| \right) |\tilde{z}_{1}^{(\mu_k)}(y - \eta_k)| \le \frac{C \varepsilon}{1 + |y - \eta_k|}.
\end{align*}
As a result,
$$
\left| \int_\R \phi L \tilde{z}_{1 k} \right| \le C \varepsilon \| \phi \|_{L^\infty(\R)} \int_{- \frac{D}{\varepsilon}}^{\frac{D}{\varepsilon}} \frac{dy}{1 + |y - \eta_k|} \le C \varepsilon |{\log \varepsilon}| \| \phi \|_{L^\infty(\R)}.
$$
By virtue of these estimates, we obtain
\begin{equation} \label{ckestimates}
|c_k| \le C \left( \varepsilon |{\log \varepsilon}| \| \phi \|_{L^\infty(\R)} + \| g \|_{\star,\,\sigma} \right), \quad \mbox{ for all } k = 1, \ldots, m,
\end{equation}
which, when combined with~\eqref{phiesttechprop}, yields that
$$
\| \phi \|_{L^\infty(\R)} \le C \Big( { |{\log \varepsilon}|\| g \|_{\star,\,\sigma} + \varepsilon |{\log \varepsilon}|^2 \| \phi \|_{L^\infty(\R)}} \Big).
$$
By taking~$\varepsilon$ suitably small, we can reabsorb in the left-hand side the~$L^\infty$ norm of~$\phi$ appearing on the right and conclude the validity of~\eqref{mainlinearaprioriest}.

Having established~\eqref{mainlinearaprioriest}, the unique solvability of~\eqref{mainlinearprop}--\eqref{phiorthtoZ1j} is now a simple consequence of the Fredholm theory. Indeed, consider the Hilbert space
$$
H := \left\{ \phi \in H^{\frac{1}{2}}(\R) : \phi = 0 \mbox{ a.e.~in~} \R\setminus I_\varepsilon \mbox{ and } \int_\R \phi \chi_k Z_{1 k} = 0 \mbox{ for every } k = 1, \ldots, m  \right\},
$$
endowed with the inner product
$$
\langle \phi, \psi \rangle_{H} := \frac{1}{2 \pi} \int_\R \int_\R \frac{\big( {\phi(y) - \phi(y')} \big)\big( {\psi(y) - \psi(y')} \big)}{(y - y')^2} \, dy dy'.
$$
Thanks to the fractional Poincar\'e inequality, this product yields a norm equivalent on~$H$ to the full norm~$\| \cdot \|_{L^2(\R)} + \sqrt{\langle \, \cdot \,, \, \cdot \, \rangle_{H}}$.
It is easy to see that~$\phi \in L^{\infty}(\R)$ weakly solves~\eqref{mainprojlinearprob}--\eqref{phiorthtoZ1j} for some~$c_1, \ldots, c_m \in \R$ if and only it belongs to~$H$ and satisfies
\begin{equation} \label{weak-formulation-on-H}
\langle \phi, \psi \rangle_{H} = \int_{\R} \left( W \phi + g \right) \psi \, dy,  \quad \mbox{ for all } \psi \in H.
\end{equation}
Note that, if $h \in L^2(I_{\epsilon})$, then $F: H \to H$ given by
$$
F(\psi) = \int_{I_{\epsilon}} h \psi\,  dy
$$
is a linear continuous functional on~$H$. Thus, Riesz's representation theorem (see, e.g.,~\cite[Theorem 5.7]{gilbarg-trudinger}) implies that there exists a unique~$\phi \in H$ such that
\begin{equation} \label{inverse-half-Laplacian-Riesz}
\langle \phi, \psi \rangle_{H} = \int_{I_{\epsilon}} h \psi\, dy, \quad \textup{ for all } \psi \in H.
\end{equation}
It is then clear that $A:  L^2(I_{\epsilon}) \to H$ given by $A(h) = \phi$ with $\phi \in H$ the unique solution to \eqref{inverse-half-Laplacian-Riesz} is a continuous linear map. Also, using the compactness of the embedding of $H$ into $L^2(I_{\epsilon})$, one can check that $K: H \to H$ given by $K(\phi) = A(W\phi)$ is compact. Problem \eqref{weak-formulation-on-H} can then be formulated as
\begin{equation} \label{Fredholm-formulation}
\phi = K(\phi) + A(g), \quad \phi \in H,
\end{equation}
Since~$K$ is compact, Fredholm's alternative (see, e.g.,~\cite[Theorem 5.3]{gilbarg-trudinger}) guarantees the unique solvability of \eqref{Fredholm-formulation} for every $g \in L^2(I_{\epsilon})$, provided the homogeneous problem (corresponding to~$g \equiv 0$) has no non-trivial solutions. As this fact is guaranteed by the a priori estimate~\eqref{mainlinearaprioriest}, the proof is complete.
\end{proof}

\subsection{Dependence on the~$\xi_j$'s} $ $
\medbreak
\noindent
For~$\xi = \epsilon \eta \in \mathcal{I}_{\delta_0}$, we denote by~$L_{\eta}^{-1}$ the operator that associates to any~$g \in L^{\infty}_{\star,\,\sigma}(I_{\epsilon})$ the unique solution to~\eqref{mainprojlinearprob}-\eqref{phiorthtoZ1j}. Note that, by Proposition~\ref{mainlinearprop},~$L_{\eta}^{-1} : L^{\infty}_{\star,\,\sigma}(I_{\epsilon}) \to L^{\infty}(\R)$ is a linear continuous operator, provided~$\epsilon$ is sufficiently small.  

For later purposes, it is also important to understand the differentiability of~$L_{\eta}^{-1}$ with respect to~$\eta$. This is the content of the next result.

\begin{proposition}  \label{prop1derivative}
Let~$\delta_0, \sigma \in (0, 1)$. There exist two constants~$\epsilon_0 \in (0, 1)$ and~$C > 0$ for which the following holds true: given~$\epsilon \in (0, \epsilon_0)$, an open set~$\mathcal{V} \subset \epsilon^{-1} \mathcal{I}_{\delta_0}$, and a~$C^1$ map~$\mathcal{V} \ni \eta \mapsto g_\eta \in L^\infty_{\star, \sigma}$, the map~$\mathcal{V} \ni \eta \mapsto L_\eta^{-1}(g_\eta) \in L^\infty(\R)$ is also of class~$C^1$ and
\begin{equation} \label{boundDerivative}
\big\| {\partial_{\eta_{l}} ( L_\eta^{-1}(g_\eta))} \big\|_{L^{\infty}(\R)} \leq C \big( {\log \frac1\epsilon} \big) \left( \big( {\log \frac1\epsilon} \big) \|g_\eta\|_{\star,\sigma} + \big\| {\partial_{\eta_l} g_\eta} \big\|_{\star,\sigma} \right),
\end{equation}
for all~$\eta \in \mathcal{V}$ and~$l = 1, \ldots, m$.
\end{proposition}


We postpone the proof of this result to Appendix~\ref{App dependence on the xijs}.
 
\section{Nonlinear theory} \label{Nonlinear theory}

\noindent Let us consider the nonlinear projected problem
\begin{equation} \label{nonlinprojprobforphi}
\left\{
\begin{aligned}
L \phi & = - \mathscr{E} + \mathcal{N}(\phi) + \sum_{j = 1}^m c_j \chi_j Z_{1 j}, \quad && \mbox{ in } I_\varepsilon, \\
\phi & = 0, \quad && \mbox{ in } \R \setminus I_\varepsilon,
\end{aligned}
\right.
\end{equation}
along with the orthogonality conditions
\begin{equation} \label{phiortnonlin}
\int_\R \phi \chi_j Z_{1 j} = 0, \quad \mbox{ for all } j = 1, \ldots, m.
\end{equation}
The main goal of this section is to establish the existence of a bounded solution~$\phi$ to~\eqref{nonlinprojprobforphi}--\eqref{phiortnonlin} for suitable coefficients~$c_1, \ldots, c_m$. Note that we keep here the notation introduced in Sections~\ref{The approximate solution} and~\ref{Linear theory}. In particular, we recall that~$\mathscr{E}$ is given in \eqref{error},~$L$ and~$\mathcal{N}$ in~\eqref{L - W and N}, the~$Z_{1 j}$'s in~\eqref{Zij}, and the~$\chi_j$'s in~\eqref{chij}.

\medbreak
The main result of this section reads as follows.

\begin{proposition} \label{propnonlineartheory}
Let~$\delta_0, \sigma \in (0,1)$. There exist~$\epsilon_0 \in (0,1)$ and~$C > 0$ for which the following holds true: if~$\xi \in \mathcal{I}_{\delta_0}$ and~$\epsilon \in (0, \epsilon_0]$, then there exists a unique bounded solution~$\phi$ to \eqref{nonlinprojprobforphi}--\eqref{phiortnonlin}, for some uniquely determined~$c_1, \ldots, c_m \in \R$, which satisfies
\begin{equation} \label{estfornonlinearphi}
\|\phi\|_{L^{\infty}(\R)} \leq C \, \epsilon^{1-\sigma} \big( {\log \frac{1}{\epsilon}} \big).
\end{equation}
\end{proposition}


\begin{proof}
First of all, having at hand~$L^{-1}_\eta$ (see Proposition~\ref{prop1derivative}), we reformulate~\eqref{nonlinprojprobforphi}--\eqref{phiortnonlin} as the fixed-point problem
$$
\phi = T(\phi) := L_\eta^{-1} \big( {-\mathscr{E} + \mathcal{N}(\phi)} \big), \quad \phi \in L^{\infty}(\R).
$$
By Proposition~\ref{mainlinearprop}--in particular, by the a priori estimate~\eqref{mainlinearaprioriest}--, we know there exist two constants~$C_1 > 0$ and~$\epsilon_1 \in (0,1)$ such that, for all~$\epsilon \in (0,\epsilon_1]$,
\begin{equation*}
\|T(\phi)\|_{L^{\infty}(\R)} \leq C_1 |{\log \epsilon}| \Big( \|\mathcal{N}(\phi)\|_{\star,\,\sigma} + \|\mathscr{E}\|_{\star,\,\sigma} \Big).
\end{equation*}
On the one hand, observe that, for all~$\phi \in L^{\infty}(\R)$ with~$\|\phi\|_{L^{\infty}(\R)} \leq 1$,  
\begin{equation} \label{nstarestimate}
\|\mathcal{N}(\phi)\|_{\star,\,\sigma} \leq C_2 \|\phi\|_{L^{\infty}(\R)}^2,
\end{equation}
for some~$C_2 > 0$ (independent of~$\epsilon$ and~$\xi$). On the other hand, by Proposition \ref{error-measure}, we know there exists~$C_3 > 0$  (independent of~$\epsilon$,~$\sigma$, and~$\xi$) such that 
$$
\|\mathscr{E}\|_{\star,\,\sigma} \leq C_3\, \epsilon^{1-\sigma}.
$$
Hence, for all $\phi \in L^{\infty}(\R)$ with $\|\phi\|_{L^{\infty}(\R)} \leq 1$, it follows that 
\begin{equation} \label{Tphi-before-ball}
\|T(\phi)\|_{L^{\infty}(\R)} \leq C_1 |{\log\epsilon}| \big( C_2 \|\phi\|_{L^{\infty}(\R)}^2 + C_3\,  \epsilon^{1-\sigma} \big).
\end{equation}
We also point out that there exists~$C_4 > 0$ (independent of~$\epsilon$ and~$\xi$) such that, for all~$\phi_1, \phi_2 \in L^{\infty}(\R)$ with~$\|\phi_i\|_{L^{\infty}(\R)} \leq 1$ ($i = 1, 2$), it holds
\begin{equation} \label{Nphi2-Nphi1}
\|\mathcal{N}(\phi_2) - \mathcal{N}(\phi_1) \|_{\star,\,\sigma} \leq C_4 \big( \|\phi_1\|_{L^{\infty}(\R)} + \|\phi_2\|_{L^{\infty}(\R)} \big) \|\phi_2 - \phi_1\|_{L^{\infty}(\R)}.
\end{equation}
Now, we set
$$
B_{\epsilon}:= \left\{ \phi \in L^{\infty}(\R) : \|\phi\|_{L^{\infty}(\R)} \leq 2C_1 C_3\,  \epsilon^{1-\sigma} |{\log \epsilon}| \right\} \! ,
$$
and choose~$\epsilon_0 \in (0,\epsilon_1]$ in a way that
\begin{equation} \label{epsilon0nonlinear}
\max \Big\{ {2 C_1 C_3 \,\epsilon^{1-\sigma} |{\log \epsilon}|,\, 4 C_1^2 C_2 C_3 \, \epsilon^{1-\sigma} |{\log \epsilon}|^2,\, 8 C_1^2 C_3 C_4 \, \epsilon^{1-\sigma} |{\log \epsilon}|^2} \Big\} \leq 1, \quad \textup{ for all } \epsilon \in (0,\epsilon_0].
\end{equation}
We conclude the proof by showing that, for all~$\epsilon \in (0,\epsilon_0]$,~$T$ is a contraction mapping on~$B_{\epsilon}$. To this end, let~$\epsilon \in (0, \epsilon_0]$. First, we prove that~$T(B_{\epsilon}) \subset B_{\epsilon}$. Indeed, taking into account~\eqref{Tphi-before-ball}, the definition of~$B_{\epsilon}$, and~\eqref{epsilon0nonlinear}, it is immediate to check that
$$
\|T(\phi)\|_{L^{\infty}(\R)} \leq C_1 C_3\, \epsilon^{1-\sigma} |{\log \epsilon}| \big( {4C_1^2C_2C_3\, \epsilon^{1-\sigma} |{\log \epsilon}|^2 + 1} \big) \leq 2 C_1 C_3 \epsilon^{1-\sigma} |{\log \epsilon}|, \quad \textup{ for all } \phi \in B_{\epsilon}.
$$
Thus,~$T(B_{\epsilon}) \subset B_{\epsilon}$ and it only remains to show the existence of~$L \in (0,1)$ for which
\begin{equation*}
\|T(\phi_2) - T(\phi_1) \|_{L^{\infty}(\R)} \leq L \|\phi_1-\phi_2\|_{L^{\infty}(\R)}, \quad \textup{ for all } \phi_1, \phi_2 \in B_{\epsilon}.
\end{equation*}
Taking advantage of estimates~\eqref{mainlinearaprioriest} and~\eqref{Nphi2-Nphi1}, the definition of~$B_{\epsilon}$, and~\eqref{epsilon0nonlinear}, it is clear that
\begin{equation*}
\begin{aligned}
\|T(\phi_2) - T(\phi_1) \|_{L^{\infty}(\R)} & = \| L_\eta^{-1} \! \left( \mathcal{N}(\phi_2) - \mathcal{N}(\phi_1) \right) \|_{L^{\infty}(\R)} \leq C_1 |{\log \epsilon}| \,  \|\mathcal{N}(\phi_2) - \mathcal{N}(\phi_1)\|_{\star, \sigma} \\
& \leq 4C_1^2 C_3 C_4 \epsilon^{1-\sigma} |{\log \epsilon}|^2 \|\phi_2-\phi_1\|_{L^{\infty}(\R)} \\
& \leq \frac{1}{2} \, \|\phi_2-\phi_1\|_{L^{\infty}(\R)}, \qquad \textup{for all } \phi_1, \phi_2 \in B_{\epsilon}.
\end{aligned}
\end{equation*}
We thus conclude that, for all $\epsilon \in (0,\epsilon_0]$, $T$ is a contraction mapping on $B_{\epsilon}$ and the proof is complete. 
\end{proof}

\subsection{Dependence on the $\xi_j$'s} $ $
\medbreak
\noindent
For~$\xi = \epsilon \eta \in \mathcal{I}_{\delta_0}$, we consider the map~$\Phi: \eta \mapsto \Phi(\eta) = \phi$, where~$\phi$ is the unique bounded solution to~\eqref{nonlinprojprobforphi}--\eqref{phiortnonlin}. We analyze here its differentiability.

\begin{proposition} \label{prop2derivative}
Let~$\delta_0, \sigma \in (0, 1)$. There exist~$\epsilon_0 \in (0, 1)$ such that, if~$\epsilon \in (0, \epsilon_0)$, then the map~$\Phi$ is of class~$C^1$. Moreover, there exists a constant~$C > 0$ for which
$$
\big\|\partial_{\eta_l} (\Phi(\eta)) \big\|_{L^{\infty}(\R)} \leq C \epsilon^{1-\sigma} \big( {\log \frac1\epsilon}\big)^2,
$$
for all~$\eta \in \epsilon^{-1} \mathcal{I}_{\delta_0}$ and~$l = 1, \ldots, m$.
\end{proposition}

We postpone the proof of this result to Appendix~\ref{App dependence on the xijs}.

\section{Variational reduction} \label{Variational reduction}

\noindent Our goal in this section is to show that~$\xi = (\xi_1, \ldots, \xi_m)$ can be chosen in such a way that the coefficients~$c_1, \ldots, c_m$ appearing in~\eqref{nonlinprojprobforphi} are all vanishing. In this way, we would have found a solution~$\phi$ to~\eqref{nonlinprobforphi} and thus a solution~$v$ to~\eqref{main-problem}. 

Given~$\xi \in \mathcal{I}_{\delta_0}$ and~$\epsilon \in (0,\epsilon_0]$, let~$\phi = \phi[\xi]$ be the solution to~\eqref{nonlinprojprobforphi}--\eqref{phiortnonlin} provided by Proposition \ref{propnonlineartheory}. Letting~$\psi := \phi \big( { \frac{\cdot}{\varepsilon} } \big)$, we consider the function~$F_\epsilon: \mathcal{I}_{\delta_0} \to \R$ defined by
$$
F_\varepsilon(\xi) := \mathcal{J}_\varepsilon(\mathscr{U} + \psi), \quad \mbox{ for } \xi \in \mathcal{I}_{\delta_0},
$$
where
$$
\mathcal{J}_\varepsilon(u) := \frac{1}{4\pi} \int_\R \int_\R \frac{|u(x) - u(x')|^2}{|x - x'|^2} \, dx dx' - \varepsilon \int_I \kappa(x) e^{u(x)} \, dx,
$$
and introduce the shortened notation $c = c[\xi] := \big( {c_1[\xi], \ldots, c_m[\xi]} \big)$.

First we have the following preliminary lemma, concerning the derivative of~$\mathscr{V}$ with respect to the~$\eta_j$'s.

\begin{lemma} \label{derofVwrtetalem}
For any~$j = 1, \ldots, m$, it holds that~$\partial_{\eta_j} \mathscr{V} = Z_{1 j} + R_j$, with~$\| R_j \|_{L^\infty(I_\varepsilon)} \le C \varepsilon$, for some constant~$C > 0$ independent of $\epsilon$ and $\xi$.
\end{lemma}


\begin{proof}
Recalling~\eqref{bubbles},~\eqref{ansatz},~\eqref{ansatz-expanded-variables}, and the fact that~$\xi = \varepsilon \eta$, we have
$$
\mathscr{V}(y) = \sum_{k = 1}^m \Big( {\mathcal{U}_{\mu_k(\xi), \eta_k}(y) - \log \kappa(\xi_k) + H_k(\varepsilon y)} \Big) - 2 (m - 1) \log \varepsilon,
$$
for all~$y \in I_\varepsilon$. As a result, taking into account the definitions in~\eqref{Z0Z1defs}, we find
\begin{align*}
\partial_{\eta_j} \mathscr{V}(y) & = \partial_{b} \, \mathcal{U}_{\mu_j(\xi), b}(y) \Big|_{b = \eta_j} - \varepsilon \frac{\kappa'(\xi_j)}{\kappa(\xi_j)} + \varepsilon \sum_{k = 1}^m \Big( {\partial_{a} \, \mathcal{U}_{a, \eta_k}(y) \Big|_{a = \mu_k(\xi)} \partial_{\xi_j} \mu_k(\xi) + \partial_{\xi_j} H_k(\varepsilon y)} \Big) \\
& = Z_{1, \mu_j(\xi), \eta_j}(y) - \varepsilon \frac{\kappa'(\xi_j)}{\kappa(\xi_j)} + \varepsilon \sum_{k = 1}^m \Big( {Z_{0, \mu_k(\xi), \eta_k}(y) \partial_{\xi_j} \mu_k(\xi) + \partial_{\xi_j} H_k(\varepsilon y)} \Big).
\end{align*}
We go back to definition~\eqref{mu_j} of~$\mu_k$ and compute
$$
\partial_{\xi_j} \mu_k(\xi) = \Bigg\{ {\delta_{j k} \Bigg( {\frac{\kappa'(\xi_k)}{\kappa(\xi_k)} + 2 \partial_a H(a, \xi_k) \Big|_{a = \xi_k} + \sum_{i \ne k} \partial_a G(a, \xi_i) \Big|_{a = \xi_k}} \Bigg) + (1 - \delta_{j k}) \partial_b G(\xi_k, b) \Big|_{b = \xi_j}} \Bigg\} \mu_k(\xi).
$$
From the regularity properties of~$\kappa$,~$H$,~$G$, the fact that~$\inf_I \kappa > 0$, and the uniform bounds~\eqref{mu-bounded both sides} on~$\mu_k$, we infer that~$|\partial_{\xi_j} \mu_k(\xi)| \le C$ for all~$\xi \in \mathcal{I}_{\delta_0}$ and for some constant~$C > 0$ (indep. of $\epsilon$ and $\xi$). We now have a look at the term~$\partial_{\xi_j} H_k$. Recalling~\eqref{ujdef},~\eqref{probforHj} and arguing as before, we see that it solves 
$$
\left\{
\begin{aligned}
\Lfrac \big( {\partial_{\xi_j} H_k} \big) & = 0, \quad && \textup{ in } I,\\
\partial_{\xi_j} H_k & = f_{j k}, \quad && \textup{ in } \R \setminus I,
\end{aligned}
\right.
$$
with
$$
f_{j k}(x) := - Z_{0, \mu_k(\xi), 0}\left( \frac{x - \xi_k}{\varepsilon} \right) \partial_{\xi_j} \mu_k(\xi) - \delta_{j k} \left\{ \frac{1}{\varepsilon} \, Z_{1, \mu_k(\xi), 0} \left( \frac{x - \xi_k}{\varepsilon} \right) + \frac{\kappa'(\xi_k)}{\kappa(\xi_k)} \right\}, \quad \mbox{ for } x \in \R \setminus I.
$$
By~\eqref{Z0Z1defs} and the fact that~$\xi \in \mathcal{I}_{\delta_0}$, we estimate
$$
\left| \frac{1}{\varepsilon} \, Z_{1, \mu_k(\xi), 0} \left( \frac{x - \xi_k}{\varepsilon} \right) \right| = \frac{1}{|x - \xi_k|} \left| \frac{x - \xi_k}{\varepsilon} \, Z_{1, \mu_k(\xi), 0} \left( \frac{x - \xi_k}{\varepsilon} \right) \right| \le \frac{2}{|x - \xi_k|} \le \frac{2}{\delta_0}, \quad \mbox{ for all } x \in \R \setminus I.
$$
Consequently, we easily deduce that~$|f_{j k}| \le C$ in~$\R \setminus I$. Hence, Proposition \ref{boundary regularity half-laplacian} gives that~$|\partial_{\xi_j} H_k(x)| \le C$ for all~$x \in I$. The combination of all these facts leads us to the claim of the lemma.
\end{proof}

The next result shows the connection between the coefficients~$c_1, \ldots, c_m$ and the function~$F_\varepsilon$.

\begin{lemma} \label{c=0iffstationarylem}
The function~$F_\varepsilon$ is of class~$C^1$ in~$\mathcal{I}_{\delta_0}$ and
$$
c[\xi] = 0 \quad \mbox{if and only if} \quad \nabla F_\varepsilon(\xi) = 0,
$$
provided~$\varepsilon$ is small enough.
\end{lemma}
\begin{proof}
Through a change of variables and recalling the definition~\eqref{ansatz-expanded-variables} of~$\mathscr{V}$, it is immediate to see that~$F_\varepsilon(\xi) = \widehat{\mathcal{J}}_\varepsilon(\mathscr{V} + \phi)$, with
$$
\widehat{\mathcal{J}}_\varepsilon(v) := \frac{1}{4\pi} \int_\R \int_\R  \frac{|v(y) - v(y')|^2}{|y - y'|^2} \, dy dy' - \int_{I_\varepsilon} \kappa(\varepsilon y) e^{v(y)} \, dy.
$$
By this and the fact that~$\eta = \varepsilon^{-1} \xi$, setting~$v := \mathscr{V} + \phi$ we have that, for~$j = 1, \ldots, m$, 
\begin{align*}
\partial_{\xi_j} F_\varepsilon(\xi) & = \frac{1}{\varepsilon} \, \partial_{\eta_j} \widehat{\mathcal{J}}_\varepsilon(\mathscr{V} + \phi) \\
& = \frac{1}{\varepsilon} \left\{ \frac{1}{2\pi} \int_\R \int_\R  \frac{\left( v(y) - v(y') \right) \left( \partial_{\eta_j} v(y) - \partial_{\eta_j} v(y') \right)}{|y - y'|^2} \, dy dy' - \int_{I_\varepsilon} \kappa(\varepsilon y) e^{v(y)} \partial_{\eta_j} v(y) \, dy \right\}.
\end{align*}
Using that~$v$ solves
$$
\left\{
\begin{aligned}
(- \Delta)^{\frac{1}{2}} v & = \kappa(\varepsilon \, \cdot \,) e^v + \sum_{k = 1}^m c_k \chi_k Z_{1 k}, \quad && \mbox{ in } I_\varepsilon, \\
v & = 2 \log \varepsilon, \quad && \mbox{ in } \R \setminus I_\varepsilon,
\end{aligned}
\right.
$$
we deduce that
$$
\partial_{\xi_j} F_\varepsilon(\xi) = \frac{1}{\varepsilon} \sum_{k = 1}^m c_k \int_{I_\varepsilon} \chi_k Z_{1 k} \left( \partial_{\eta_j} \mathscr{V} + \partial_{\eta_j} \phi \right).
$$
Applying Proposition~\ref{prop2derivative} and Lemma~\ref{derofVwrtetalem}, we see that~$\big\| {\partial_{\eta_j} \mathscr{V} + \partial_{\eta_j} \phi - Z_{1 j}} \big\|_{L^\infty(I_\varepsilon)} \le C \varepsilon^{1 - \sigma} |{\log \varepsilon}|^2$. Hence,
$$
\partial_{\xi_j} F_\varepsilon(\xi) = \frac{1}{\varepsilon} \sum_{k = 1}^m M_{j k} c_k, \quad \mbox{ with } \quad M_{j k} := \int_{I_\varepsilon} \chi_k Z_{1 k} Z_{1 j} + O \big( {\varepsilon^{1 - \sigma} |{\log \varepsilon}|^2} \big).
$$
Note that
\begin{equation} \label{Mjj}
M_{j j} = \int_{\R} \chi Z_{1, \mu_j, 0}^2 - C \varepsilon^{1 - \sigma} |{\log \varepsilon}|^2 \ge 8 \int_0^{\bar{R}} \frac{t^2}{\big( {\mu_j^2 + t^2} \big)^2} \, dt  - C \varepsilon^{1 - \sigma} |{\log \varepsilon}|^2 \ge \frac{1}{C},
\end{equation}
while, for~$k \ne j$,
\begin{align*}
|M_{j k}| & \le 4 \int_{- \bar{R} - 1}^{\bar{R} + 1} \frac{|t|}{\big( {\mu_k^2 + t^2} \big)} \frac{|t + \eta_k - \eta_j|}{\big( {\mu_j^2 + (t + \eta_k - \eta_j)^2} \big)} \, dt + C \varepsilon^{1 - \sigma} |{\log \varepsilon}|^2 \\
& \le C \left( \varepsilon + \varepsilon^{1 - \sigma} |{\log \varepsilon}|^2 \right) \le C \varepsilon^{1 - \sigma} |{\log \varepsilon}|^2,
\end{align*}
provided~$\varepsilon$ is sufficiently small. Consequently, the matrix~$M = \big( {M_{j k}} \big)$ is invertible for~$\varepsilon$ small enough and the conclusion of the lemma follows.
\end{proof}

Thanks to the previous result, the problem of making the~$c_j$'s vanish is reduced to finding the stationary points of~$F_\epsilon$. As~$\psi = \phi(\frac{\cdot}{\epsilon})$ is small in~$\epsilon$, the natural first step is to understand~$F_\varepsilon$ as a perturbation of~$\mathcal{J}_\epsilon(\mathscr{U})$--which is the content of the following lemma.

\begin{lemma} \label{FepsasJepspertlem}
It holds
$$
F_\varepsilon(\xi) = \mathcal{J}_\varepsilon(\mathscr{U}) +  \theta_{1, \varepsilon}(\xi), \quad \mbox{ for } \xi \in \mathcal{I}_{\delta_0},
$$
with~$\theta_{1, \varepsilon}: \mathcal{I}_{\delta_0} \to \R$ satisfying
$$
|\theta_{1, \varepsilon}(\xi)| \le C \varepsilon^{2 (1 - \sigma)} |{\log \varepsilon}|, \quad \mbox{ for all } \xi \in \mathcal{I}_{\delta_0},
$$
for some constant~$C > 0$ independent of~$\epsilon$ and~$\xi$, provided~$\epsilon$ is sufficiently small.
\end{lemma}
\begin{proof}
Adopting the notation introduced in the proof of Lemma~\ref{c=0iffstationarylem}, we write
$$
F_\varepsilon(\xi) - \mathcal{J}_\varepsilon(\mathscr{U}) = \widehat{\mathcal{J}}_\varepsilon(\mathscr{V}+ \phi) - \widehat{\mathcal{J}}_\varepsilon(\mathscr{V}) = \int_0^1 \frac{d}{d t} \widehat{\mathcal{J}}_\varepsilon(\mathscr{V} + t \phi) \, dt.
$$
As, by Lemma~\ref{error-measure} and Proposition \ref{propnonlineartheory},
\begin{align*}
\left| \frac{d}{d t} \widehat{\mathcal{J}}_\varepsilon(\mathscr{V} + t \phi) \Big|_{t = 0} \right| & = \left| \frac{1}{2\pi} \int_\R \int_\R  \frac{\big( {\mathscr{V}(y)- \mathscr{V}(y')} \big) \big( {\phi(y) - \phi(y')} \big)}{|y - y'|^2} \, dy dy' - \int_{I_\varepsilon} \kappa(\varepsilon y) e^{\mathscr{V}(y)} \phi(y) \, dy \right| \\
& = \left| \int_{I_\varepsilon} \mathscr{E}(y) \phi(y) \, dy \right| \le \| \mathscr{E} \|_{\star, \sigma} \| \phi \|_{L^\infty(\R)} \int_{I_\varepsilon} \left( \varepsilon + \sum_{j = 1}^m \frac{1}{\big( {1 + |y - \eta_j|} \big)^{1 + \sigma}} \right) \, dy \\
& \le C \varepsilon^{2 (1 - \sigma)} |{\log \varepsilon}|,
\end{align*}
we obtain that
$$
F_\varepsilon(\xi) - \mathcal{J}_\varepsilon(\mathscr{U}) = \int_0^1 (1 - t) \, \frac{d^2}{dt^2} \widehat{\mathcal{J}}_\varepsilon(\mathscr{V} + t \phi) \, dt + O \! \left( \varepsilon^{2 (1 - \sigma)} |{\log \varepsilon}| \right).
$$
Now, using that~$\phi$ solves~\eqref{nonlinprojprobforphi}--\eqref{phiortnonlin}, we find
\begin{align*}
\frac{d^2}{dt^2} \widehat{\mathcal{J}}_\varepsilon(\mathscr{V} + t \phi) & = \frac{1}{2\pi} \int_\R \int_\R  \frac{|\phi(y) - \phi(y')|^2}{|y - y'|^2} \, dy dy' - \int_{I_\varepsilon} \kappa(\varepsilon y) e^{\mathscr{V}(y) + t \phi(y)} \phi(y)^2 \, dy \\
& = \int_{\R} \bigg( {- \mathscr{E} + \mathcal{N}(\phi) + \sum_{j = 1}^m c_j \chi_j Z_{1 j}} \bigg) \phi  - \int_{I_\varepsilon} W \big( {e^{t \phi} - 1} \big) \phi^2 \\
& = \int_{I_\varepsilon} \big( {- \mathscr{E} + \mathcal{N}(\phi)} \big) \phi - \int_{I_\varepsilon} W \big( {e^{t \phi} - 1} \big) \phi^2.
\end{align*}
Hence, by Lemma~\ref{error-measure},~\eqref{Wasperturb}--\eqref{thetabounds}, Proposition~\ref{propnonlineartheory}, and~\eqref{nstarestimate}, we get
\begin{align*}
& \left| \frac{d^2}{dt^2} \widehat{\mathcal{J}}_\varepsilon(\mathscr{V} + t \phi) \right| \\
& \hspace{20pt} \le C \| \phi \|_{L^\infty(\R)} \left\{ \| \mathscr{E} \|_{\star,\,\sigma} + \| \mathcal{N}(\phi) \|_{\star,\,\sigma} + \| \phi \|_{L^\infty(\R)}^2 \sum_{j = 1}^m \int_{I_\varepsilon} \frac{1}{1 + |y - \eta_j|} \left( \frac{1}{1 + |y - \eta_j|} + \varepsilon \right) dy \right\} \\
& \hspace{20pt} \le C \varepsilon^{1 - \sigma} |{\log \varepsilon}| \Big( {\varepsilon^{1 - \sigma} + \varepsilon^{2 (1 - \sigma)} |{\log \varepsilon}|^2} \Big) \le C \varepsilon^{2(1 - \sigma)} |{\log \varepsilon}|,
\end{align*}
which gives the claim of the lemma.
\end{proof}

The next lemma takes care of the expansion in~$\epsilon$ of~$\mathcal{J}_\varepsilon(\mathscr{U})$

\begin{lemma} \label{JepsasPhipertlem}
It holds
$$
\mathcal{J}_\varepsilon(\mathscr{U}) = - 2 \pi m \big( {1 + \log \varepsilon} \big) + \pi \, \Xi(\xi) + \theta_{2, \varepsilon}(\xi), \quad \mbox{ for } \xi \in \mathcal{I}_{\delta_0},
$$
with
\begin{equation} \label{Phidef}
\Xi(\xi) := - \sum_{j = 1}^m \Big( {2 \log \kappa(\xi_j) + H(\xi_j, \xi_j)}  +\sum_{i\neq j} G(\xi_j, \xi_i) \Big),
\end{equation}
and~$\theta_{2, \varepsilon}: \mathcal{I}_{\delta_0} \to \R$ satisfying
$$
|\theta_{2, \varepsilon}(\xi)| \le C \varepsilon |{\log \varepsilon}|, \quad \mbox{ for all } \xi \in \mathcal{I}_{\delta_0},
$$
for some constant~$C > 0$ independent of~$\epsilon$ and~$\xi$, provided~$\varepsilon$ is sufficiently small.
\end{lemma}
\begin{proof}
Recalling \eqref{ansatz}, we write
$$
\mathcal{\mathcal{J}}_\varepsilon(\mathscr{U}) = \sum_{i, j = 1}^m J^{(1)}_{i j} - J^{(2)},
$$
where
\begin{align*}
J^{(1)}_{i j} & := \frac{1}{4\pi} \int_\R \int_\R \frac{\big( {\mathscr{U}_i(x) - \mathscr{U}_i(x')} \big) \big( {\mathscr{U}_j(x) - \mathscr{U}_j(x')} \big)}{|x - x'|^2} \, dx dx', && \hspace{-30pt} \mbox{for } i, j \in \{ 1, \ldots, m\}, \\
J^{(2)} & := \varepsilon \int_I \kappa(x) \prod_{i = 1}^m e^{\mathscr{U}_i(x)} \, dx. &&
\end{align*}

We begin by expanding the~$J^{(1)}_{i j}$'s. Using that~$\mathscr{U}_j$ solves
$$
\left\{
\begin{aligned}
(-\Delta)^{\frac{1}{2}} \mathscr{U}_j & = \varepsilon \kappa(\xi_j) e^{u_j}, \quad && \mbox{ in } I, \\
\mathscr{U}_j & = 0, \quad && \mbox{ in } \R \setminus I,
\end{aligned}
\right.
$$
we obtain that
$$
J^{(1)}_{i j} = \frac{\varepsilon}{2} \, \kappa(\xi_j) \int_I \mathscr{U}_i(x) e^{u_j(x)} \, dx.
$$
We first consider the case~$i = j$. By definitions~\eqref{bubbles} and~\eqref{ujdef}, Lemma~\ref{Hj and uj expansions}~\ref{Hjandujlem(ii)}, and a change of variables, we have
\begin{align*}
\frac{\varepsilon}{2} \, \kappa(\xi_j) \int_{I \setminus \big( {\xi_j - \frac{\delta_0}{2}, \xi_j + \frac{\delta_0}{2}} \big)} \mathscr{U}_j(x) e^{u_j(x)} \, dx & = \frac{1}{2 \mu_j \varepsilon} \int_{I \setminus \big( {\xi_j - \frac{\delta_0}{2}, \xi_j + \frac{\delta_0}{2}} \big)} \Big( G(x, \xi_j) + O(\varepsilon^2) \Big) e^{\mathcal{U}_{1, 0} \left( \frac{x - \xi_j}{\mu_j \varepsilon} \right)} dx \\
& = \int_{\frac{I - \xi_j}{\mu_j \varepsilon} \setminus {\big(- \frac{\delta_0}{2 \mu_j \varepsilon}, \frac{\delta_0}{2 \mu_j \varepsilon}} \big)} \Big( G(\xi_j + \mu_j \varepsilon y, \xi_j) + O(\varepsilon^2) \Big) \frac{dy}{1 + y^2}.
\end{align*}
As~$G(\xi_j + x, \xi_j) = \Gamma(x) + H(\xi_j + x, \xi_j) = - 2 \log |x| + O(1)$ uniformly with respect to~$x \in I - \xi_j$, we get
\begin{equation} \label{J1tech1}
\begin{aligned}
\left| \frac{\varepsilon}{2} \, \kappa(\xi_j) \int_{I \setminus \big( {\xi_j - \frac{\delta_0}{2}, \xi_j + \frac{\delta_0}{2}} \big)} \mathscr{U}_j(x) e^{u_j(x)} \, dx \right| & = \left| \int_{\frac{I - \xi_j}{\mu_j \varepsilon} \setminus {\big(- \frac{\delta_0}{2 \mu_j \varepsilon}, \frac{\delta_0}{2 \mu_j \varepsilon}} \big)} \Big( {- 2 \log \big( {\mu_j \varepsilon |y|} \big) + O(1)} \Big) \frac{dy}{1 + y^2} \right| \\
& \le C \int_{\frac{\delta_0}{2 \mu_j \varepsilon}}^{+\infty} \Big( {\big| {\log \big( {\mu_j \varepsilon y} \big)} \big| + 1} \Big) \frac{dy}{y^2} \le C \varepsilon.
\end{aligned}
\end{equation}
To deal with the integral over~$\big( {\xi_j - \frac{\delta_0}{2}, \xi_j + \frac{\delta_0}{2}} \big)$, we write~$\mathscr{U}_j = u_j + H_j$ and treat the resulting two integrals separately. On the one hand, by arguing as before we see that
\begin{align*}
\frac{\varepsilon}{2} \, \kappa(\xi_j) \int_{\xi_j - \frac{\delta_0}{2}}^{\xi_j + \frac{\delta_0}{2}} u_j(x) e^{u_j(x)} \, dx & = \frac{1}{2 \mu_j \varepsilon} \int_{\xi_j - \frac{\delta_0}{2}}^{\xi_j + \frac{\delta_0}{2}} \left\{ \mathcal{U}_{1, 0} \left( \frac{x - \xi_j}{\mu_j \varepsilon} \right) - \log \big( {\kappa(\xi_j) \mu_j \varepsilon^2} \big) \right\} e^{\mathcal{U}_{1, 0} \left( \frac{x - \xi_j}{\mu_j \varepsilon} \right)} dx \\
& = - \int_{- \frac{\delta_0}{2 \mu_j \varepsilon}}^{\frac{\delta_0}{2 \mu_j \varepsilon}} \left\{ \log (1 + y^2) + 2 \log (\mu_j \varepsilon) + \log \left( \frac{\kappa(\xi_j)}{2 \mu_j} \right) \right\} \frac{dy}{1 + y^2}.
\end{align*}
Since
\begin{equation} \label{integralMexp}
\int_{-M}^M \! \frac{dy}{1 + y^2} = \pi + O \! \left( \frac{1}{M} \right) \quad \mbox{and} \quad \int_{-M}^M \frac{\log(1 + y^2)}{1 + y^2} \, dy = 2 \pi \log 2 + O \! \left( \frac{\log M}{M} \right)
\end{equation}
as~$M \rightarrow +\infty$--the latter fact can be readily established via the residue theorem--, we get
\begin{equation} \label{J1tech2}
\frac{\varepsilon}{2} \, \kappa(\xi_j) \int_{\xi_j - \frac{\delta_0}{2}}^{\xi_j + \frac{\delta_0}{2}} u_j(x) e^{u_j(x)} \, dx = - 2 \pi \log (\mu_j \varepsilon) - 2 \pi \log 2 - \pi \log \left( \frac{\kappa(\xi_j)}{2 \mu_j} \right) + O \big( {\varepsilon |{\log \varepsilon}|} \big).
\end{equation}
On the other hand, by Lemma~\ref{Hj and uj expansions}~\ref{Hjandujlem(i)},
\begin{align*}
\frac{\varepsilon}{2} \, \kappa(\xi_j) \int_{\xi_j - \frac{\delta_0}{2}}^{\xi_j + \frac{\delta_0}{2}} H_j(x) e^{u_j(x)} \, dx & = \frac{1}{2 \mu_j \varepsilon} \int_{\xi_j - \frac{\delta_0}{2}}^{\xi_j + \frac{\delta_0}{2}} \left\{ H(x,\xi_j) - \log \left( \frac{2\mu_j}{\kappa(\xi_j)} \right) + O(\epsilon^2) \right\} e^{\mathcal{U}_{1, 0} \left( \frac{x - \xi_j}{\mu_j \varepsilon} \right)} dx \\
& = \int_{- \frac{\delta_0}{2 \mu_j \varepsilon}}^{\frac{\delta_0}{2 \mu_j \varepsilon}} \left\{ H(\xi_j + \mu_j \varepsilon y, \xi_j) + \log \left( \frac{\kappa(\xi_j)}{2\mu_j} \right) + O(\epsilon^2) \right\} \frac{dy}{1 + y^2}.
\end{align*}
As~$H(\xi_j + x, \xi_j) = H(\xi_j, \xi_j) + O(|x|)$ for all~$x \in \big[ {- \frac{\delta_0}{2}, \frac{\delta_0}{2}} \big]$, the previous identity becomes
\begin{align*}
\frac{\varepsilon}{2} \, \kappa(\xi_j) \int_{\xi_j - \frac{\delta_0}{2}}^{\xi_j + \frac{\delta_0}{2}} H_j(x) e^{u_j(x)} \, dx & = \int_{- \frac{\delta_0}{2 \mu_j \varepsilon}}^{\frac{\delta_0}{2 \mu_j \varepsilon}} \left\{ H(\xi_j, \xi_j) + O(\varepsilon |y|) + \log \left( \frac{\kappa(\xi_j)}{2\mu_j} \right) + O(\epsilon^2) \right\} \frac{dy}{1 + y^2} \\
& = \pi H(\xi_j, \xi_j) + \pi \log \left( \frac{\kappa(\xi_j)}{2\mu_j} \right) + O \big( {\varepsilon |{\log \varepsilon}|} \big).
\end{align*}
By putting together this,~\eqref{J1tech1}, and~\eqref{J1tech2}, we conclude that
\begin{equation} \label{J1expansion}
J^{(1)}_{j j} = - 2 \pi \log (\mu_j \varepsilon) + \pi H(\xi_j, \xi_j) - 2 \pi \log 2 + O \big( {\varepsilon |{\log \varepsilon}|} \big), \quad \mbox{ for all } j = 1, \ldots, m.
\end{equation}

We now expand~$J^{(1)}_{i j}$ for~$i \ne j$. We write~$I = A_{i} + A_j + B_{i j}$, where~$A_k := \big( \xi_k - \frac{\delta_0}{2}, \xi_k + \frac{\delta_0}{2} \big)$ for~$k = i, j$ and~$B_{i j} := I \setminus \big( {A_i \cup A_j} \big)$ are disjoint sets. We first deal with the integral over~$B_{i j}$. Recalling~\eqref{bubbles},~\eqref{ujdef}, and Lemma~\ref{Hj and uj expansions}~\ref{Hjandujlem(ii)}, we have
$$
\frac{\varepsilon}{2} \, \kappa(\xi_j) \int_{B_{i j}} \mathscr{U}_i(x) e^{u_j(x)} \, dx = \mu_j \varepsilon \int_{B_{i j}} \frac{G(x, \xi_i) + O(\varepsilon^2)}{\mu_j^2 \varepsilon^2 + (x - \xi_j)^2} \, dx.
$$
As~$G(x, \xi_i) = - 2 \log |x - \xi_i| + O(1)$ uniformly with respect to~$x \in I$, translating variables we compute
\begin{equation} \label{Jijtech1}
\left| \frac{\varepsilon}{2} \, \kappa(\xi_j) \int_{B_{i j}} \mathscr{U}_i(x) e^{u_j(x)} \, dx \right| \le C \varepsilon \int_{B_{i j}} \frac{1 + \big| {\log |x - \xi_i|} \big|}{\mu_j^2 \varepsilon^2 + (x - \xi_j)^2} \, dx \le C \varepsilon \int_{\frac{\delta_0}{2}}^{2 D} \Big( {1 + \big| {\log |x'|} \big|} \Big) \, dx' \le C \varepsilon.
\end{equation}
Next, we estimate the integral over~$A_i$. Using this time Lemma~\ref{Hj and uj expansions}~\ref{Hjandujlem(i)}, the fact that~$H(x, \xi_i) = O(1)$ uniformly with respect to~$x \in I$, and changing variables appropriately, we easily obtain
$$
\frac{\varepsilon}{2} \, \kappa(\xi_j) \int_{A_i} \mathscr{U}_i(x) e^{u_j(x)} \, dx = -
\int_{\frac{A_i - \xi_j}{\mu_j \varepsilon}} \left\{ \log \left( \mu_i^2 \varepsilon^2 + \big( {\xi_j - \xi_i + \mu_j \varepsilon y} \big)^2 \right) + O(1) \right\} \frac{dy}{1 + y^2}.
$$
From here, we let~$y' := \frac{\xi_j - \xi_i}{\mu_i \varepsilon} + \frac{\mu_j}{\mu_i} y$ and estimate
\begin{equation} \label{Jijtech2}
\begin{aligned}
\left| \frac{\varepsilon}{2} \, \kappa(\xi_j) \int_{A_i} \mathscr{U}_i(x) e^{u_j(x)} \, dx \right| & = \frac{\mu_i}{\mu_j} \Bigg| {\int_{- \frac{\delta_0}{2 \mu_i \varepsilon}}^{\frac{\delta_0}{2 \mu_i \varepsilon}} \frac{\log \left( \mu_i^2 \varepsilon^2 (1 + y'^2) \right) + O(1)}{1 + \left( \frac{\xi_i - \xi_j}{\mu_j \varepsilon} + \frac{\mu_i}{\mu_j} y' \right)^2} \, dy'} \Bigg| \\
& \le C \varepsilon^2 \int_{- \frac{\delta_0}{2 \mu_i \varepsilon}}^{\frac{\delta_0}{2 \mu_i \varepsilon}} \frac{|{\log \varepsilon}| + \log (1 + y'^2)}{\left( \xi_i - \xi_j + \mu_i \varepsilon y' \right)^2} \, dy' \le C \varepsilon |{\log \varepsilon}|,
\end{aligned}
\end{equation}
where for the last inequality we used the fact that~$|\xi_i - \xi_j + \mu_i \varepsilon y'| \ge \frac{\delta_0}{2}$ for all~$y' \in \big[ {- \frac{\delta_0}{2 \mu_i \varepsilon}, \frac{\delta_0}{2 \mu_i \varepsilon}} \big]$. Finally, we address the integral over~$A_j$. By arguing similarly as before, we express it as
\begin{align*}
& \frac{\varepsilon}{2} \, \kappa(\xi_j) \int_{A_j} \mathscr{U}_i(x) e^{u_j(x)} \, dx \\
& \hspace{30pt} = \int_{- \frac{\delta_0}{2 \mu_j \varepsilon}}^{\frac{\delta_0}{2 \mu_j \varepsilon}} \left\{ H(\xi_j + \mu_j \varepsilon y, \xi_i) - \log \left( \mu_i^2 \varepsilon^2 + \big( {\xi_j - \xi_i + \mu_j \varepsilon y} \big)^2 \right) + O(\varepsilon^2) \right\} \frac{dy}{1 + y^2}.
\end{align*}
As $H(\cdot, \xi_i)$ is $C^1$ away from the boundary of $I$, we have that~$H(\xi_j + x, \xi_i) = H(\xi_j, \xi_i) + O(|x|) = G(\xi_j, \xi_i) + 2 \log |\xi_i - \xi_j| + O(|x|)$ uniformly with respect to~$x \in \big( {- \frac{\delta_0}{2}, \frac{\delta_0}{2}} \big)$. Hence, also recalling~\eqref{integralMexp},
\begin{equation} \label{Jijtech3}
\begin{aligned}
& \frac{\varepsilon}{2} \, \kappa(\xi_j) \! \int_{A_j} \mathscr{U}_i(x) e^{u_j(x)} \, dx \\
& \hspace{30pt} = \int_{- \frac{\delta_0}{2 \mu_j \varepsilon}}^{\frac{\delta_0}{2 \mu_j \varepsilon}} \left\{ G(\xi_j, \xi_i) - \log \left( \frac{\mu_i^2 \varepsilon^2}{(\xi_i - \xi_j)^2} + \Big( {1 + \frac{\mu_j \varepsilon y}{\xi_j - \xi_i}} \Big)^2 \right) + O(\varepsilon |y|) + O(\varepsilon^2) \right\} \frac{dy}{1 + y^2} \\
& \hspace{30pt} = \pi G(\xi_j, \xi_i) + O \big( {\varepsilon |{\log \varepsilon}|} \big) - \int_{- \frac{\delta_0}{2 \mu_j \varepsilon}}^{\frac{\delta_0}{2 \mu_j \varepsilon}} \log \left( \frac{\mu_i^2 \varepsilon^2}{(\xi_i - \xi_j)^2} + \Big( {1 + \frac{\mu_j \varepsilon y}{\xi_j - \xi_i}} \Big)^2 \right) \frac{dy}{1 + y^2}.
\end{aligned}
\end{equation}
Since
$$
\left| \int_{- \frac{\delta_0}{2 \mu_j \varepsilon}}^{\frac{\delta_0}{2 \mu_j \varepsilon}} \log \left( \frac{\mu_i^2 \varepsilon^2}{(\xi_i - \xi_j)^2} + \Big( {1 + \frac{\mu_j \varepsilon y}{\xi_j - \xi_i}} \Big)^{\! 2} \right) \frac{dy}{1 + y^2} \right| \le C \int_{0}^{\frac{\delta_0}{2 \mu_j \varepsilon}} \frac{\varepsilon^2 (1 + y^2) + \varepsilon y}{1 + y^2} \, dy \le C \varepsilon |{\log \varepsilon}|,
$$
we conclude from~\eqref{Jijtech1},~\eqref{Jijtech2}, and~\eqref{Jijtech3} that
\begin{equation} \label{Jijexpansion}
J^{(1)}_{i j} = \pi G(\xi_j, \xi_i) + O \big( {\varepsilon |{\log \varepsilon}|} \big), \quad \mbox{ for all } i, j \in \{1, \ldots, m\} \mbox{ such that } i \ne j.
\end{equation}

At last, we expand~$J^{(2)}$. We decompose~$I$ as~$I = B \cup \bigcup_{j = 1}^m A_j$, with~$A_j := \big( {\xi_j - \frac{\delta_0}{2}, \xi_j + \frac{\delta_0}{2}} \big)$ and~$B := I \setminus \bigcup_{j = 1}^m A_j$. The integral over~$B$ is readily estimated. Indeed, by~\eqref{decomposition-Green} and Lemma~\ref{Hj and uj expansions}~\ref{Hjandujlem(ii)},
$$
\mathscr{U}_i(x) = G(x, \xi_i) + O(\varepsilon^2) = H(x, \xi_i) - 2 \log |x - \xi_i| + O(\varepsilon^2) = O(1)
$$
for every~$i = 1, \ldots, m$, uniformly with respect to~$x \in B$. Hence, we have
\begin{equation} \label{J2techest}
\Bigg| {\varepsilon \int_B \kappa(x) \prod_{i = 1}^m e^{\mathscr{U}_i(x)} \, dx} \Bigg| \le C \varepsilon.
\end{equation}
We now consider the integral over~$A_j$, for some~$j \in\{ 1, \ldots, m\}$. By Lemma~\ref{Hj and uj expansions}, it holds
\begin{align*}
\mathscr{U}_i(x) & = G(x, \xi_i) + O(\varepsilon^2) \qquad\qquad \mbox{for every } i \ne j, \\
\mathscr{U}_j(x) & = - \log \left( 1 + \left( \frac{x - \xi_j}{\mu_j \varepsilon} \right)^{\! 2} \right) + H(x, \xi_j) - 2 \log (\mu_j \varepsilon) + O(\varepsilon^2),
\end{align*}
uniformly with respect to~$x \in A_j$. Hence, changing variables,
$$
\varepsilon \int_{A_j} \kappa(x) \prod_{i = 1}^m e^{\mathscr{U}_i(x)} \, dx = \frac{1}{\mu_j} \int_{- \frac{\delta_0}{2 \mu_j \varepsilon}}^{\frac{\delta_0}{2 \mu_j \varepsilon}} \big( {1 + O(\varepsilon^2)} \big) \kappa(\xi_j + \mu_j \varepsilon y) e^{H(\xi_j + \mu_j \varepsilon y, \xi_j)} \prod_{i \ne j} e^{G(\xi_j + \mu_j \varepsilon y, \xi_i)} \frac{dy}{1 + y^2}.
$$
Now, as~$\kappa$,~$H(\cdot, \xi_j)$, and~$G(\cdot, \xi_i)$ (for~$i \ne j$) are of class~$C^1$ in~$A_j$, we see that
\begin{align*}
\kappa(\xi_j + \mu_j \varepsilon y) & = \kappa(\xi_j) + O(\varepsilon |y|) = \kappa(\xi_j) \big( {1 + O(\varepsilon |y|)} \big), \\
H(\xi_j + \mu_j \varepsilon y, \xi_j) & = H(\xi_j, \xi_j) + O(\varepsilon |y|), \\
G(\xi_j + \mu_j \varepsilon y, \xi_i) & = G(\xi_j, \xi_i) + O(\varepsilon |y|),
\end{align*}
uniformly with respect to~$y \in \big( {- \frac{\delta_0}{2 \mu_j \varepsilon}, \frac{\delta_0}{2 \mu_j \varepsilon}} \big)$. As a result,
$$
\varepsilon \int_{A_j} \kappa(x) \prod_{i = 1}^m e^{\mathscr{U}_i(x)} \, dx = \frac{\kappa(\xi_j) e^{H(\xi_j, \xi_j) + \sum_{i \ne j} G(\xi_j, \xi_i)}}{\mu_j} \int_{- \frac{\delta_0}{2 \mu_j \varepsilon}}^{\frac{\delta_0}{2 \mu_j \varepsilon}} \frac{1 + O(\varepsilon^2) + O(\varepsilon |y|)}{1 + y^2} \, dy.
$$
Recalling the definition~\eqref{mu_j} of~$\mu_j$ and~\eqref{integralMexp}, we arrive at
$$
\varepsilon \int_{A_j} \kappa(x) \prod_{i = 1}^m e^{\mathscr{U}_i(x)} \, dx = 2 \int_{- \frac{\delta_0}{2 \mu_j \varepsilon}}^{\frac{\delta_0}{2 \mu_j \varepsilon}} \frac{1 + O(\varepsilon^2) + O(\varepsilon |y|)}{1 + y^2} \, dy = 2 \pi + O \big( {\varepsilon |{\log \varepsilon}|} \big).
$$
By this and~\eqref{J2techest}, we obtain
$$
J^{(2)} = 2 m \pi + O \big( {\varepsilon |{\log \varepsilon}|} \big).
$$
The conclusion of the lemma follows then by combining this last expansion with~\eqref{J1expansion},~\eqref{Jijexpansion}, and recalling once again the definition~\eqref{mu_j} of the~$\mu_j$'s.
\end{proof}

\begin{proof}[Proof of Theorem \ref{mainTheoremExistence}]
Let~$d \ge m \ge 1$ be two fixed integers and~$\mathcal{D} := \big\{ {\xi \in I^m : \xi_i = \xi_j \mbox{ for some } i \ne j} \big\}$. Consider the~$C^1$ function~$\Xi : I^m \setminus \mathcal{D} \to \R$ defined as in~\eqref{Phidef}. We extend it to a function~$\Xi : I^m \to \R \cup \{ - \infty \}$ by setting~$\Xi(\xi) := -\infty$ for all~$\xi \in \mathcal{D}$. Observe that~$I^m = \bigcup_{\beta \in \Sigma_{m, d}} S_\beta$, where
$$
\Sigma_{m, d} := \Big\{ {\beta = (\beta_1, \ldots, \beta_m) : \beta_j \in \left\{ 1, \ldots, d \right\} \mbox{ for all } j = 1, \ldots, m} \Big\} \quad \mbox{and} \quad S_\beta := \prod_{j = 1}^m I_{\beta_j}.
$$
Since~$d \ge m$, there exists at least one~$\hat{\beta} \in \Sigma_{m, d}$ such that~$\hat{\beta}_i \ne \hat{\beta}_j$ for every~$i, j \in\{ 1, \ldots, m\}$ such that~$i \ne j$. Clearly, the restriction of~$\Xi$ to~$S_{\hat{\beta}}$ is real-valued, as~$S_{\hat{\beta}} \cap \mathcal{D} = \varnothing$. In addition, setting, for~$\delta > 0$ small,
$$
Q_\delta := \Big\{ { \xi \in S_{\hat{\beta}} : \dist \! \big( {\xi, \partial S_{\hat{\beta}}} \big) > \delta} \Big\},
$$
it holds
$$
\lim_{\delta \searrow 0} \, \inf_{S_{\hat{\beta}} \setminus Q_\delta} \Xi = +\infty.
$$
From this and the continuity of~$\Xi$ in~$S_{\hat{\beta}}$, it follows that there exists~$\widehat{\xi} \in S_{\hat{\beta}}$ such that~$\Xi(\widehat{\xi}) = \min_{S_{\hat{\beta}}} \Xi < \min_{\partial Q_\delta} \Xi$, for every~$\delta > 0$ sufficiently small (see Figure \ref{figure}).

\begin{figure}[h]
\centering
\fbox{
\includegraphics[width=0.6\textwidth]{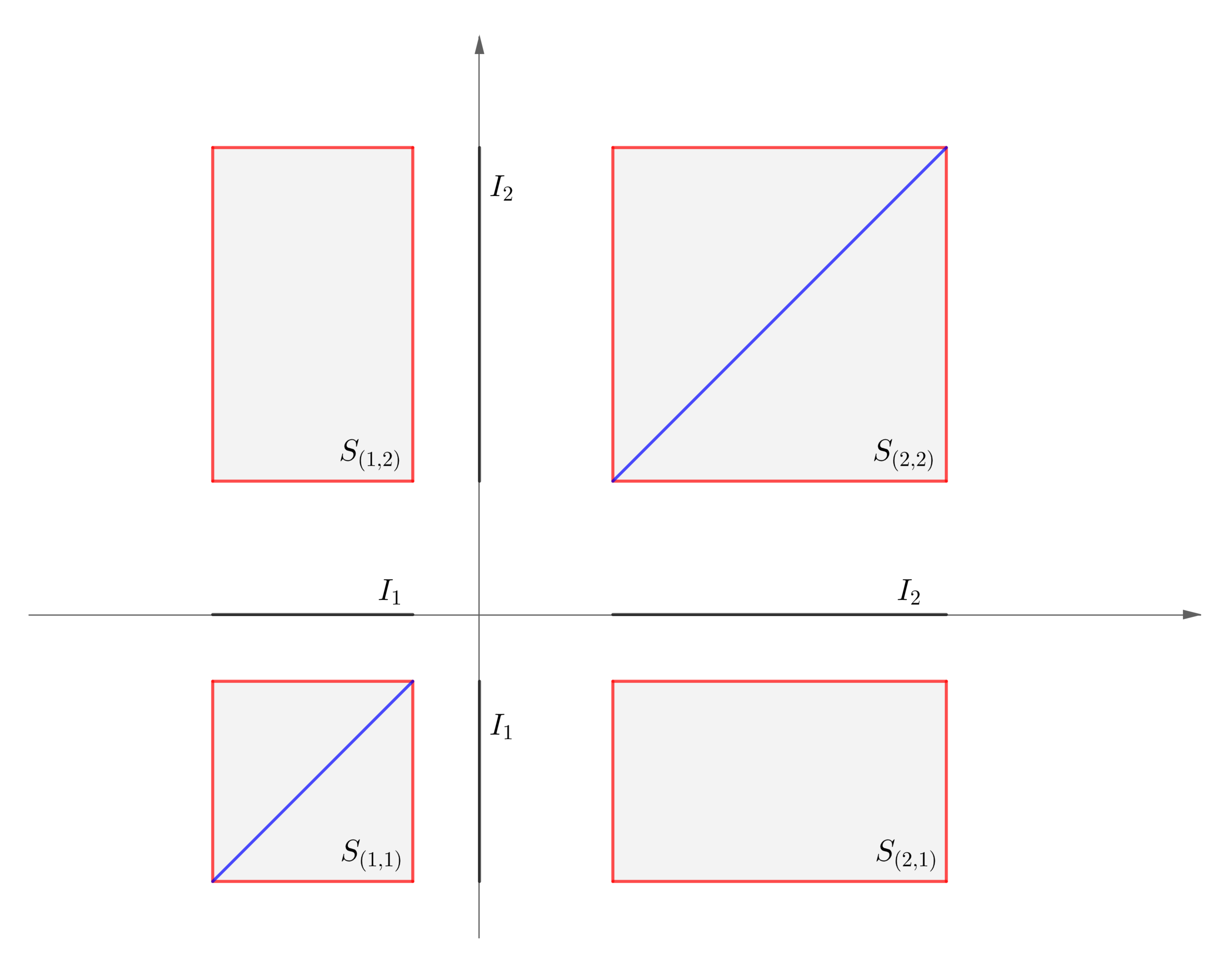}}
\caption{The geometry involved in the proof of Theorem \ref{mainTheoremExistence}, for~$d = m = 2$. The function~$\Xi$ takes value~$+\infty$ on the boundary of~$I^2$ (colored in red),~$-\infty$ on its diagonal (colored in blue), and is finite elsewhere. The sets~$S_{(1, 2)}$ and~$S_{(2, 1)}$ contain therefore local minima of~$\Xi$ (symmetric with respect to the diagonal).}
 \label{figure}
\end{figure}

Let now~$\delta_0 \in (0, 1/2)$ be small enough to have that~$\widehat{\xi} \in \mathcal{I}_{2 \delta_0}$ and $\sigma \in (0,1/2)$. Define~$\widetilde{F}_\varepsilon(\xi) := F_\varepsilon(\xi) + 2 \pi m (1 + \log \varepsilon)$ for all~$\xi \in \mathcal{I}_{\delta_0}$. Also, let~$\epsilon_0 \in (0,1)$ be as in Proposition~\ref{propnonlineartheory}. From Lemmas~\ref{FepsasJepspertlem} and~\ref{JepsasPhipertlem} it follows that
$$
\widetilde{F}_\varepsilon(\xi) = \pi \, \Xi(\xi) + \theta_\varepsilon(\xi), \quad \mbox{ for } \xi \in \mathcal{I}_{\delta_0},
$$
provided that~$\epsilon$ is sufficiently small. Here,~$\theta_\varepsilon : \mathcal{I}_{\delta_0} \to \R$ is a $C^1$ function satisfying~$|\theta_\varepsilon(\xi)| \le C \varepsilon |{\log \varepsilon}|$ for all~$\xi \in \mathcal{I}_{\delta_0}$. For every~$\epsilon$ small enough we then have that~$\widetilde{F}_\varepsilon(\widehat{\xi}) < \min_{\partial Q_{\delta_0}} \widetilde{F}_\varepsilon$. As a result, there exists~$\widehat{\xi}_\varepsilon \in Q_{\delta_0}$ such that~$\widetilde{F}_\varepsilon(\widehat{\xi}_\varepsilon) = \inf_{Q_{\delta_0}} \widetilde{F}_\varepsilon$. The point~$\widehat{\xi}_\varepsilon$ is a local minimum of~$F_\varepsilon$ too and thus, by Lemma~\ref{c=0iffstationarylem}, we conclude that~$c_j[\widehat{\xi}_\varepsilon] = 0$ for every~$j = 1, \ldots, m$. 

Finally, let~$\epsilon_\star \in (0,\epsilon_0)$ small enough to ensure that the previous argument holds for every~$\epsilon \in (0,\epsilon_\star)$. For~$\epsilon \in (0,\epsilon_\star)$, let~$\phi_{\epsilon}$ be the solution of~\eqref{nonlinprojprobforphi}--\eqref{phiortnonlin} given by Proposition~\ref{propnonlineartheory} applied with~$\xi = \widehat{\xi}_{\varepsilon} \in \mathcal{I}_{\delta_0}$. Since~$c_j[\widehat{\xi}_\varepsilon] = 0$ for every~$j = 1, \ldots, m$, the function~$\phi_{\epsilon}$ actually solves~\eqref{nonlinprobforphi}. Hence, denoting with~$\mathscr{U}[\widehat{\xi}_{\epsilon}]$ the function given in~\eqref{ansatz} with~$\xi = \widehat{\xi}_{\epsilon}$ and defining~$u_{\epsilon} := \mathscr{U}[\widehat{\xi}_{\epsilon}]+ \phi_{\epsilon} (\frac{\cdot}{\epsilon})$ for every~$\epsilon \in (0,\epsilon_\star)$, the result immediately follows.
\end{proof}

\section{The non-existence result} \label{The non-existence result}

\noindent This section is devoted to prove Proposition \ref{nonExistence}. The proof of this result relies on the generalized fractional Pohozaev identity recently established in \cite{D-F-W21}. More precisely, we make use of \cite[Theorem 1.1]{D-F-W21}. Let us state this result adapted to our framework. To that end, we consider the problem
\begin{equation} \label{problemPohozaev}
\left\{
\begin{aligned}
(-\Delta)^{\frac12} u & = f(u), \quad && \textup{ in } I, \\
u & = 0, && \textup{ in } \R \setminus I,
\end{aligned}
\right.
\end{equation}
with $f : \R \to \R$ locally Lipschitz, and define
$$
F(t):= \int_0^t f(s) ds.
$$
Also, given~$\Upsilon \in C^{\,0,1}(\R)$, we define the \textit{fractional deformation kernel associated to}~$\Upsilon$ as 
$$
K_{\Upsilon}(x,y):= \frac{1}{2\pi} \left( \Upsilon'(x) + \Upsilon'(y) - 2 \, \frac{(\Upsilon(x)-\Upsilon(y))(x-y)}{|x-y|^2} \right) \frac{1}{(x-y)^2}, \quad \textup{ for a.e.~} x,y \in \R,
$$
and denote by~$\mathcal{E}_{\Upsilon}$ the bilinear form associated to~$K_{\Upsilon}$, that is
$$
\mathcal{E}_{\Upsilon}(v,w):= \int_{\R} \int_{\R} (v(x) - v(y))(w(x)-w(y)) K_{\Upsilon}(x,y) dx dy, \quad \textup{for all } v,w \in H^{\frac12}(\R).
$$

\begin{theorem}{\rm{(\hspace{-0.003cm}\cite[Theorem 1.1]{D-F-W21}).}} \label{pohozaev}
Let $u \in H^{\frac12}(\R) \cap L^{\infty}(I)$ be a solution to \eqref{problemPohozaev}. Then, we have
\begin{equation} \label{pohozaevIdentity}
\frac{\pi}{4} \sum_{j = 1}^d \left( \Upsilon(b_j) \lim_{x \to b_j^{-}} \frac{u^2(x)}{b_j-x} - \Upsilon(a_j) \lim_{x \to a_j^+} \frac{u^2(x)}{x-a_j} \right)  = 2 \int_I F(u) \Upsilon' dx - \mathcal{E}_{\Upsilon}(u,u),
\end{equation}
for all~$\Upsilon \in C^{\,0,1}(\R)$.
\end{theorem}

\bigbreak
The second key ingredient to prove Proposition \ref{nonExistence} is the following quantitative version of the fractional Hopf's Lemma:

\begin{lemma}  \label{BrezisCabre}
Assume that~$u \in H^{\frac12}(\R) \cap C((0,1))$ satisfies~$(-\Delta)^{\frac12} u \geq 0$ in~$(0,1)$ with~$u \geq 0$ in~$\R$. Then, there exists a universal constant~$c_0 > 0$ such that
$$
u(x) \geq c_0 \, \|u\|_{L^1((0,1))} (\min\{x,1-x\})^{\frac12}, \quad \textup{ for all } x \in (0,1).
$$
\end{lemma}

\begin{proof}
If~$u = 0$ in~$(0, 1)$, there is nothing to prove. Hence, we assume~$\| u \|_{L^1((0, 1))} > 0$. Using the weak Harnack inequality (see, e.g.,~\cite[Theorem 2.2]{RO-S19}), we get the existence of a universal constant~$\overline{c} > 0$ for which
\begin{equation} \label{BCinterior}
u(x) \geq \overline{c} \int_{\R} \frac{u(x)}{(1+|x|)^2} \, dx \geq \overline{c} \int_0^1 \frac{u(x)}{(1+|x|)^2} \, dx \geq \frac{\overline{c}}{4} \, \|u\|_{L^1((0,1))}, \quad \textup{ for all } x \in \Big[ \frac38, \frac58 \Big].
\end{equation}
Next, we consider~$\varphi(x):= \varphi_2(2x-1)$ with~$\varphi_2 \in H^{\frac12}(\R) \cap C(\R)$ as in~\cite[Lemma 3.2]{RO-S14-2} and observe that
\begin{equation} \label{barrierROSLemma32}
\left\{
\begin{aligned}
(-\Delta)^{\frac12} \varphi & \leq 0, \quad && \textup{ in } (0,3/8)\cup(5/8,1),\\
\varphi & = 1, && \textup{ in } [3/8,5/8], \\
\varphi & \geq \widetilde{c}\, (\min\{x,1-x\})^{\frac12},\quad && \textup{ in } (0,1),\\
\varphi & = 0, && \textup{ in } \R \setminus (0,1),
\end{aligned}
\right.
\end{equation}
for some universal constant~$\widetilde{c} > 0$. Then, we define 
$$
\widetilde{u} := \frac{4 u}{\overline{c} \|u\|_{L^1((0,1))}},
$$ 
and, combining~\eqref{BCinterior} with~\eqref{barrierROSLemma32}, we easily see that
$$
\left\{
\begin{aligned}
(-\Delta)^{\frac12}(\widetilde{u} - \varphi) & \geq 0, \quad && \textup{ in } (0,3/8)\cup (5/8,1),\\
\widetilde{u} - \varphi & \geq 0, && \textup{ in } \R \setminus \big((0,3/8)\cup (5/8,1)\big).
\end{aligned}
\right.
$$
From the maximum principle and the properties of~$\varphi$ we infer that
\begin{equation} \label{BCboundary}
u(x) \geq \frac{\overline{c}}{4} \, \|u\|_{L^1((0,1))}\, \varphi(x) \geq  \frac{\widetilde{c}\,\overline{c}}{4} \, \|u\|_{L^1((0,1))} (\min\{x,1-x\})^{\frac12}, \quad \textup{ for all } x \in \Big(0,\frac38 \Big) \cup \Big(\frac58,1\Big).
\end{equation}
The result immediately follows from the combination of \eqref{BCinterior} and \eqref{BCboundary}.
\end{proof}

Having at hand these two ingredients we proceed to prove Proposition~\ref{nonExistence}. For~$b \ge 1$, let
$$
I = I_b := (-b-1,-b) \cup (b,b+1),
$$ 
A very useful preliminary observation is that, if $u$ is a solution to \eqref{main-problem} with $\kappa \equiv 1$ and
\begin{equation*}  
\lambda := \epsilon \int_{I_b} e^{u} dx,
\end{equation*}
then $v := \lambda^{-1} u$ is a solution to 
\begin{equation} \label{meanField}
\left\{
\begin{aligned}
(-\Delta)^{\frac12} v & = \frac{e^{\lambda v}}{\int_{I_b} e^{\lambda v} dz}, \quad && \textup{ in } I_b, \\
v & > 0, && \textup{ in } I_b, \\
v & = 0, && \textup{ in } \R \setminus I_b.
\end{aligned}
\right.
\end{equation}
Also, note that if~$v$ is a solution to~\eqref{meanField}, then~$\widehat{v}(x) = v(-x)$ is also a solution. Hence, without loss of generality, we may assume that, if $v$ is a solution to \eqref{meanField}, then
$$
\|v\|_{L^1((b,b+1))} \geq \|v\|_{L^1((-b-1,-b))}.
$$
Finally, observe that, if $v$ is a solution to \eqref{meanField}, then $w(x) = v(x+b)$ is a solution to
\begin{equation} \label{meanField2}
\left\{
\begin{aligned}
(-\Delta)^{\frac12} w & = \frac{e^{\lambda w}}{\int_{J_b} e^{\lambda w} dz}, \quad && \textup{ in } J_b, \\
w & > 0, && \textup{ in } J_b, \\
w & = 0, && \textup{ in } \R \setminus J_b.
\end{aligned}
\right.
\end{equation}
with 
\begin{equation} \label{Jb}
J_b:= (-2b-1,-2b) \cup (0,1).
\end{equation}
Moreover, taking into account the previous discussion, without loss of generality, we may assume that, if $w$ is a solution to \eqref{meanField2}, then
\begin{equation} \label{norm01}
\|w\|_{L^1((0,1))} \geq \|w\|_{L^1((-2b-1,-2b))}.
\end{equation}

\noindent The proof of Proposition~\ref{nonExistence} is then reduced to showing the validity of the following result.


\begin{proposition} \label{nonExistenceReformulated}
Let~$\delta_0 \in (0,1/10]$. There exist~$b_{\star} \geq 1$ and~$m_{\star} > 2$ for which the following holds true: for all~$m \geq m_{\star}$ and all~$\delta \in (0,1/10]$, there exists~$\epsilon_0 \in (0,1)$ such that, if~$\epsilon \in (0,\epsilon_0)$, then
\begin{equation} \label{meanField3}
\left\{
\begin{aligned}
(-\Delta)^{\frac12} w & = \frac{e^{\lambda w}}{\int_{J_{b_{\star}}} e^{\lambda w} dz}, \quad && \textup{ in } J_{b_{\star}}, \\
w & > 0, && \textup{ in } J_{b_{\star}}, \\
w & = 0, && \textup{ in } \R \setminus J_{b_{\star}},
\end{aligned}
\right.
\end{equation}
has no solution of the form $w = \lambda^{-1}(\mathscr{U} + \psi_{\epsilon})$ with:
\begin{itemize}
\item[$\circ$] $\displaystyle \lambda = \lambda(\epsilon) :=  \epsilon \int_{J_{b_{\star}}} e^{\mathscr{U}+\psi_{\epsilon}} dx,$
\item[$\circ$] $\dist(\xi_j, \R \setminus J_{b_{\star}}) \geq \delta_0$ and $|\xi_j-\xi_i| \geq \delta$, for all $i,j \in \{1,\ldots,m\}$ with $i\neq j$, \medbreak
\item[$\circ$] $\mu_j(\xi)$ as in \eqref{mu_j} for all $j \in \{1,\ldots,m\}$, \medbreak
\item[$\circ$] $\psi_{\epsilon} \in L^{\infty}(\R)$ satisfying $\psi_{\epsilon} = 0$ a.e.~in $\R \setminus J_{b_{\star}}$ and $\|\psi_{\epsilon}\|_{L^{\infty}(J_{b_{\star}})} \to 0$, as $\epsilon \to 0$.
\end{itemize}
\end{proposition}

\medbreak
Before going further, let us recall that the definition of~$\mathscr{U}$ is given in \eqref{ansatz}. To prove Proposition~\ref{nonExistenceReformulated}, we plan to apply Theorem~\ref{pohozaev} in the framework of~\eqref{meanField2} with~$\Upsilon(t) := \max\{t,0\}$. Note that in this case, the fractional deformation kernel associated to~$\Upsilon$ is given by
\begin{equation} \label{kernelReduced}
K(x,y):= K_{\Upsilon}(x,y) = \left\{
\begin{aligned}
& \,0, \quad && \mbox{ if } xy > 0, \\
& \frac{1}{2\pi} \left( 1 - \frac{2 \max\{x,y\}}{|x-y|} \right) \frac{1}{(x-y)^2}, \quad && \mbox{ if } xy < 0,
\end{aligned}
\right.
\end{equation}
and the Pohozaev identity~\eqref{pohozaevIdentity} is reduced to
\begin{equation}\label{pohozaevReduced}
\frac{\pi}{4} \lim_{x \to 1^{-}} \frac{w^2(x)}{1-x} = \frac{2}{\lambda} \int_0^1 \frac{e^{\lambda w}-1}{\int_{J_b} e^{\lambda w} dz} dx -\mathcal{E}(w),
\end{equation}
with
$$
\mathcal{E}(w) := \mathcal{E}_{\Upsilon}(w,w) = \int_{\R} \int_{\R} (w(x)-w(y))^2 K(x,y) dx dy.
$$
We aim to use~\eqref{pohozaevReduced} to deduce our non-existence result. As a first step, we analyze the term~$\mathcal{E}$. 

\begin{lemma}
Let~$b \geq 1$. For all~$v \in H^{\frac12}(\R)$ with~$v \equiv 0$ in~$\R \setminus J_b$ and~$v \geq 0$ in~$J_b$, it holds that
$$
-\frac{1}{\pi b^2} \|v\|_{L^1((0,1))} \|v\|_{L^{1}((-2b-1,-2b))} \leq \mathcal{E}(v) \leq \frac{1}{2\pi}[v]_{H^{\frac12}(\R)}.
$$ 
\end{lemma}

\begin{proof}
From the definition of~$K$ (cf.~\eqref{kernelReduced}), it is immediate to see that
$$
\mathcal{E}(v) \leq \frac{1}{2\pi} [v]_{H^{\frac12}(\R)}, \quad \textup{ for all } v \in H^{\frac12}(\R).
$$
Hence, we focus on proving the leftmost inequality. Let~$v \in H^{\frac12}(\R)$ be an arbitrary function satisfying~$v \equiv 0$ in~$\R \setminus J_b$ and~$v \geq 0$ in~$J_b$. Taking into account the definitions of~$J_b$ (cf.~\eqref{Jb}) and~$K$~(cf.~\eqref{kernelReduced}), doing suitable changes of variable, and using Fubini's Theorem, it follows that
\begin{align*}
\mathcal{E}(v) & = \int_{\R} \int_{\R} (v(x)-v(y))^2 K(x,y) \, dx dy \\
&  = \int_0^{\infty} \left( \int_{-\infty}^0 (v(x)-v(y))^2 K(x,y) \, dy \right) dx + \int_0^{\infty} \left(  \int_{-\infty}^0  (v(x)-v(y))^2 K(x,y) \, dx \right) dy \\
& = \int_0^{\infty} \left( \int_{-\infty}^0 \big(v(x)-v(y)\big)^2 \big(K(x,y) + K(y,x)\big) dy \right) dx \\
& = -\frac{1}{\pi} \int_0^{\infty} \left( \int_{-\infty}^0 \frac{(v(x)-v(y))^2}{(x-y)^2} \,\frac{x+y}{x-y}\, dy \right) dx = -\frac{1}{\pi} \int_0^{\infty} \left( \int_0^{\infty}  \frac{(v(x)-v(-s))^2}{(x+s)^2} \,\frac{x-s}{x+s} \, ds \right) dx\\
& = -\frac{1}{\pi} \Bigg\{ \int_0^{\infty} v(x)^2 \left( \int_0^{\infty} \frac{x-s}{(x+s)^3} \, ds \right) dx + \int_0^{\infty} v(-s)^2 \left( \int_0^{\infty} \frac{x-s}{(x+s)^3} \, dx \right)ds  \\
& \hspace{1.2cm} - 2 \int_0^{\infty} \left( \int_{0}^{\infty} v(x) v(-s)\,\frac{x-s}{(x+s)^3}\, \, ds \right) dx \Bigg\}.
\end{align*}  
Now, note that
$$
\int_0^{\infty} \frac{x-s}{(x+s)^3}\, ds = 0, \textup{ for all } x > 0, \quad \textup{ and } \quad \int_0^{\infty} \frac{x-s}{(x+s)^3}\, dx = 0, \textup{ for all } s > 0.
$$
Hence, taking into account the definition of $J_b$ and the properties of $v$, we conclude that
$$
\begin{aligned}
\mathcal{E}(v) & = - \frac{2}{\pi} \int_0^1 \left( \int_{2b}^{2b+1} v(x)v(-s) \frac{s-x}{(x+s)^3}\, ds \right) dx \geq -\frac{1}{\pi b^2} \|v\|_{L^1((0,1))} \|v\|_{L^1((-2b-1,-2b))},
\end{aligned}
$$
which is the desired inequality.
\end{proof}

Combining the previous lemma with the generalized Pohozaev identity~\eqref{pohozaevReduced}, we infer that, for any~$b \geq 1$, if~$w$ is a solution to~\eqref{meanField2}, then
\begin{align*} 
 \frac{\pi}{4} \lim_{x \to 1^{-}} \frac{w^2(x)}{1-x}  & \leq \frac{2}{\lambda} \int_0^1 \frac{e^{\lambda w}-1}{\int_{J_b} e^{\lambda w} dz} \, dx  + \frac{1}{\pi b^2} \|w\|_{L^1((0,1))}\|w\|_{L^1((-2b-1,-2b))} \\
 & < \frac{2}{\lambda} + \frac{1}{\pi b^2} \|w\|_{L^1((0,1))}\|w\|_{L^1((-2b-1,-2b))}
\end{align*}
Moreover, recall that, if~$w$ is a solution of~\eqref{meanField2}, we may assume without loss of generality that~\eqref{norm01} holds. Thus, combining the above chain of inequalities with Lemma~\ref{BrezisCabre} we actually get that, for any~$b \geq 1$,
$$
\frac{c_0^2\, \pi}{4} \|w\|^2_{L^1((0,1))} < \frac{2}{\lambda} + \frac{1}{\pi b^2} \|w\|_{L^1((0,1))}^2
$$
Recall that $c_0 > 0$ is a universal constant--in particular independent of~$b$. Then, by choosing
$$
b = b_{\star}:= \max\Big\{1, \frac{2\sqrt{2}}{c_0 \pi} \Big\},
$$
we get that
\begin{equation} \label{pohozaevDefinitive}
\frac{c_0^2\, \pi}{32} \|w\|^2_{L^1(J_{b_{\star}})} \leq \frac{c_0^2\, \pi}{8} \|w\|^2_{L^1((0,1))} < \frac{2}{\lambda}.
\end{equation}
We will use this inequality to prove Proposition~\ref{nonExistenceReformulated}. First, we need one more technical lemma.

\begin{lemma} \label{ansatzL1norm}
Let~$\delta_0 \in (0,1/10]$,~$m \geq 2$, and assume that~$\mu_j$ satisfies~\eqref{mu-bounded both sides}, for all~$j = 1,\ldots,m$. There exist two constants~$c_1 > 0$, independent of~$m$ and~$\delta_0$, and~$\epsilon_1 \in (0,1)$ such that, if~$\dist(\xi_j,\R\setminus J_{b_{\star}}) \geq \delta_0$, for all $j = 1,\ldots,m$, and~$\epsilon \in (0,\epsilon_1)$, then
$$
\int_{J_{b_{\star}}} \mathscr{U}dx \geq c_1\, m \log (1+\delta_0 ).
$$
\end{lemma}
 
\begin{proof}
First of all, combining~\eqref{probforuj}--\eqref{probforHj} with the fact that~$\kappa \equiv 1$, we see that, for all~$j \in \{1,\ldots, m\}$, 
$$
\left\{
\begin{aligned}
\Lfrac (u_j+H_j)&  = \epsilon e^{u_j}, \quad && \textup{ in } J_{b_{\star}}, \\
u_j + H_j & = 0, \quad && \textup{ in } \R \setminus J_{b_{\star}}.
\end{aligned}
\right.
$$
Thus, using Green's representation formula, we get that
$$
u_j(x)+H_j(x) =  \frac{\epsilon}{2\pi} \int_{J_{b_{\star}}} G(x,y)  e^{u_j(y)} \, dy, \quad \textup{ for all } x \in J_{b_{\star}} \textup{ and all } j \in \{1,\ldots,m\}.
$$
Having at hand this representation, we infer that
\begin{align*}
\|u_j + H_j\,\|_{L^1(J_{b_{\star}})} = \frac{\epsilon }{2\pi} \int_{J_{b_{\star}}} \int_{J_{b_{\star}}} G(x,y) e^{u_j(y)} \, dx  dy, \quad \textup{ for all } j \in \{1,\ldots,m\}.
\end{align*}
Next, let~$(a_j,b_j)$ be the connected component of~$J_{b_\star}$ for which~$\xi_j \in (a_j,b_j)$--i.e.,~$(a_j,b_j)$ is either $(-2b_{\star}-1,-2b_{\star})$ or~$(0,1)$. We set~$L_j:= (a_j + \frac{\delta_0}{2}, b_j-\frac{\delta_0}{2})$ and observe that~$L_j \times L_j \subset J_{b_{\star}} \times J_{b_{\star}}$. Thanks to~\cite[Corollary 1.2]{C-K-S08}--or to a direct comparison with the explicit Green function of the interval~$(a_j, b_j)$--, there exists a universal constant~$\widetilde{c} > 0$ such that 
$$
G(x,y) \geq \widetilde{c}\, \log \left( 1+ \frac{d(x)^{\frac{1}{2}} d(y)^{\frac{1}{2}}}{|x-y|} \right) \geq \widetilde{c}\, \log \left( 1 + \delta_0 \right), \quad \textup{ for all } (x,y) \in L_j \times L_j \textup{ with } x \neq y,
$$
where~$d(x) := \dist(x, \R \setminus J_{b_\star})$. From this and the non-negativity of~$G$, it follows that
\begin{align*}
\|u_j + H_j\,\|_{L^1(J_{b_{\star}})} & \geq \frac{\epsilon}{2\pi} \int_{L_j} \int_{L_j} G(x,y) e^{u_j(y)} \, dx  dy \geq \frac{\widetilde{c} \epsilon}{\pi}  (1-\delta_0) \log \left( 1 + \delta_0 \right) \int_{L_j} e^{u_j(y)} \, dy \\
& =  \frac{\widetilde{c}}{2 \pi} \log \left( 1 + \delta_0 \right) \int_{\frac{L_j-\xi_j}{\mu_j \epsilon}} \frac{2}{1+z^2}\, dz, \qquad \mbox{ for all } j \in \{1,\ldots,m\}.
\end{align*}
Note that the last identity follows from the change of variable $z = (y-\xi_j)(\mu_j \epsilon)^{-1}$. Finally, taking into account \eqref{mu-bounded both sides}, we get the existence of $\epsilon_1 \in (0,1)$ (depending only on $m$ and $\delta_0$) such that $(-1,1) \subseteq \frac{L_j-\xi_j}{\mu_j \epsilon}$ for all $\epsilon \in (0,\epsilon_1)$. Hence, for all $\epsilon \in (0,\epsilon_1)$ and all $j \in \{1,\ldots,m\}$, it follows that
\[
\|u_j + H_j\|_{L^1(J_{b_{\star}})} \geq  \frac{\widetilde{c}}{2 \pi} \log ( 1 + \delta_0 ) \int_{-1}^1 \frac{2}{1+z^2}\, dz = \frac{\widetilde{c}}{2} \log ( 1 +\delta_0), 
\] 
and the result immediately follows from the definition of $\mathscr{U}$ (cf. \eqref{ansatz}).
\end{proof} 

We can now deal with the

\begin{proof}[Proof of Proposition \ref{nonExistenceReformulated}]
Let~$\delta_0 \in (0,1/10]$. For~$c_0 > 0$ as in Lemma~\ref{BrezisCabre} and~$c_1 > 0$ as in Lemma~\ref{ansatzL1norm}, we define 
$$
m_{\star} := \left\lceil 3 \bigg( \frac{16}{c_0 c_1 \log(1+\delta_0)} \bigg)^{\! 2} \right\rceil = \min \bigg\{ n \in \N\, : \, n \geq 3 \bigg( \frac{16}{c_0 c_1 \log(1+\delta_0)} \bigg)^{\! 2} \, \bigg\}.
$$
Having at hand~$m_{\star}$, we fix arbitrary~$m \in \N$,~$m \geq m_{\star}$, and~$\delta \in (0,1/10]$. Corresponding to such~$m$ and~$\delta$, we choose~$\epsilon_0 \in (0,1)$ such that:
\begin{align}
& \circ\ \epsilon_0 \leq \epsilon_1(m), \textup{ with $\epsilon_1$ given by Lemma \ref{ansatzL1norm}}; \nonumber \\
& \circ\ \|\psi_{\epsilon}\|_{L^{\infty}(J_{b_{\star}})} \leq \frac{c_1\,m}{2} \log ( 1+\delta_0), \quad \textup{ for all } \epsilon \in (0,\epsilon_0); \label{epsilon2} \\
& \circ\ m\pi \leq \lambda(\epsilon) \leq 3m\pi, \quad \textup{ for all } \epsilon \in (0,\epsilon_0). \label{epsilon3}
\end{align}
Note that, since 
$$
\lim_{\epsilon \to 0} \|\psi_{\epsilon}\|_{L^{\infty}(J_{b_{\star}})} = 0
$$
and
\begin{equation} \label{convergence2mpi}
\lim_{\epsilon \to 0} \epsilon \int_{J_{b_{\star}}} e^{\mathscr{U} + \psi_{\epsilon}} dx = \lim_{\epsilon \to 0} \lambda(\epsilon) = 2m\pi,
\end{equation}
the existence of such an $\epsilon_0$ is guaranteed. Then, we assume by contradiction that, for some~$\epsilon \in (0,\epsilon_0)$, there exists a solution to~\eqref{meanField3} of the form~$w = \lambda^{-1} (\mathscr{U}+\psi_{\epsilon})$.  By Lemma~\ref{ansatzL1norm} and \eqref{epsilon2}, we have 
\begin{equation} \label{prop631}
\begin{aligned}
\int_{J_{b_{\star}}} (\mathscr{U} + \psi) \,dx & \geq \int_{J_{b_{\star}}} \mathscr{U} dx - \|\psi\|_{L^1({J_{b_{\star}}})}  \geq \frac{c_1}{2} \, m \log (1+\delta_0).
\end{aligned}
\end{equation}
Thus, combining \eqref{pohozaevDefinitive} with \eqref{epsilon3} and \eqref{prop631},  we obtain that
$$
m < 3 \bigg( \frac{16}{c_0 c_1 \log(1+\delta_0)} \bigg)^{\! 2}  \leq m_{\star}.
$$
The above inequality gives a contradiction with our choice $m$. The proposition is thus proved. 
\end{proof}

\begin{remark}
As one can see from the proof, we do not actually need the~$\mu_j$'s to be as prescribed by~\eqref{mu_j}, but only that they satisfy the two-sided bound~\eqref{mu-bounded both sides}. In such a case, the convergence~\eqref{convergence2mpi} would be replaced by a two-sided inequality of the form
$
\overline{c}m\pi \leq \lambda(\epsilon) \leq \overline{C}m\pi$ (for $\epsilon$ small enough) and we could then finish the proof arguing as before. 
\end{remark}

\appendix

\section{The one-dimensional half-Laplacian} \label{The one dimensional half-Laplacian}

\noindent In this appendix we recall some already known results concerning the fractional Laplacian adapted to our framework. Throughout the section~$J \subset \R$ is a general open subset of~$\R$ that can be written as the union of a finite number of bounded open intervals having positive distance from each other. More precisely, for $- \infty < a_1 < b_1 < \cdots < a_d < b_d < + \infty$ with $d \geq 1$, we consider
\[
J:= \bigcup_{j = 1}^{d} J_j, \quad \textup{ with } \quad J_j:= (a_j,b_j).
\]



\noindent In order to state the next result, we need also to introduce the following space of functions with moderate growth. Given a measurable set~$A \subset \R$, we define
$$
L^1_{1/2}(A) := \Big\{ v \in L^1(A) : \| v \|_{L^1_{1/2}(A)} < +\infty \Big\},
$$
where
$$
\| v \|_{L^1_{1/2}(A)} := \int_A \frac{|v(x)|}{1 + x^2} \, dx.
$$

\begin{proposition}
\label{boundary regularity half-laplacian}
Let~$h \in L^{\infty}(J)$ and~$g \in L^1_{1/2}(\R \setminus J) \cap C^\alpha(\overline{K} \setminus J)$ for some~$\alpha \in \left( \frac{1}{2}, 1 \right]$ and some bounded open interval~$K \subset \R$ such that~$\overline{J} \subset K$. Let~$u$ be the bounded solution of the Dirichlet problem
\begin{equation*}
\left\{
\begin{aligned}
\Lfrac u & = h, \quad && \textup{ in } J, \\
u & = g, && \textup{ in } \R \setminus J.
\end{aligned}
\right.
\end{equation*}
Then, $u \in C^{0,\frac{1}{2}}(\overline{K})$ and
$$
\|u\|_{C^{0,\frac{1}{2}}(\overline{K})} \le C \Big( \|h\|_{L^{\infty}(J)} + \|g\|_{C^{\alpha}(\overline{K} \setminus J)} + \| g \|_{L^1_{1/2}(\R \setminus J)} \Big),
$$
for some constant~$C > 0$ depending only on~$J$,~$K$, and~$\alpha$.
\end{proposition}

In the previous result we used the compact notation~$C^{\alpha}(\overline{K} \setminus J)$ to refer to the space~$C^{0,\alpha}(\overline{K} \setminus J)$ if~$\alpha \in (0,1)$ and to~$C^{1}(\overline{K} \setminus J)$ if~$\alpha = 1$.

\begin{proof}[Proof of Proposition~\ref{boundary regularity half-laplacian}]
First of all, up to a rescaling we may assume that~$\dist(J, \R \setminus K) \ge 3$. Next, without loss of generality, we can take~$g$ to be supported inside~$K$. Indeed, up to a translation we have that~$K = (-a, a)$ for some~$a > 3$. Consider a cutoff function~$\chi \in C^\infty(\R)$ satisfying~$\chi = 1$ in~$\left[ -a + 2, a - 2 \right]$,~$\chi = 0$ in~$\R \setminus \left( -a + 1, a - 1 \right)$, and~$|\chi'| \le 2$ in~$\R$. Then,~$\tilde{u} := u \chi$ solves
$$
\left\{
\begin{aligned}
	\Lfrac \tilde{u} & = \tilde{h}, \quad && \textup{ in } J, \\
	\tilde{u} & = \tilde{g}, && \textup{ in } \R \setminus J,
\end{aligned}
\right.
$$
with~$\tilde{g} := g \chi$ and
$$
\tilde{h}(x) := h(x) + \int_{\R \setminus \left( - a + 2, \, a - 2 \right)} \frac{\left( 1 - \chi(y) \right) g(y)}{(x - y)^2} \, dy, \quad \mbox{ for } x \in J.
$$
Clearly,~$\tilde{g} \in C^\alpha(\R \setminus J)$ with~$\supp(\tilde{g}) \subset \subset K$ and~$\| \tilde{g} \|_{C^\alpha(\R \setminus J)} \le C \| g \|_{C^\alpha(\overline{K} \setminus J)}$, whereas~$\tilde{h} \in L^\infty(J)$ with~$\| \tilde{h} \|_{L^\infty(J)} \le \| h \|_{L^\infty(J)} + C \| g \|_{L^1_{1/2}(\R\setminus J)}$, for some constant~$C > 0$ depending only on~$a$. Notice that here we took advantage of the fact that~$J \subset (-a + 3, a - 3)$.

From now on, we thus suppose that~$g \in C^\alpha(\R \setminus J)$ with~$g = 0$ in~$\R \setminus K$. Under this assumption, it is immediate to obtain an~$L^\infty$ bound for the solution~$u$. Indeed, it suffices to consider the barrier~$\overline{w}(x) := \| h \|_{L^\infty(J)} \sqrt{(a^2 - x^2)_+} + \| g \|_{L^\infty(\R \setminus J)}$, which satisfies
$$
\left\{
\begin{aligned}
	\Lfrac \overline{w} & \ge |h|, \quad && \textup{ in } J, \\
	\overline{w} & \ge |g|, && \textup{ in } \R \setminus J.
\end{aligned}
\right.
$$
The maximum principle then yields that~$\| u \|_{L^\infty(J)} \le \| \overline{w} \|_{L^\infty(J)} \le a \| h \|_{L^\infty(J)} + \| g \|_{L^\infty(\R)}$.

We now address the~$C^{\frac{1}{2}}(\overline{J})$ regularity of~$u$. First of all, notice that is suffices to consider the case of one interval--i.e., that~$d = 1$--as revealed by a cutoff procedure similar to the one carried out earlier. Up to a rescaling and a translation, we may also assume that~$J = (-1, 1)$. Finally, it does not harm the generality to suppose~$h$ to be a smooth function--the general case can be tackled via a simple approximation procedure. Under these assumptions, it is well-known--see, e.g.,~\cite{BGR61,L72,B16}--that the solution~$u$ can be represented via the explicit Poisson kernel and Green function for the~half-Laplacian in the interval~$(-1, 1)$. More explicitly, we have that~$u = u_1 + u_2$, where
$$
u_1(x) := \mathds{1}_{(-1, 1)}(x) \int_{\R \setminus (-1, 1)} \! \frac{P_1(x, z)}{2 \pi} \, g(z) \, dz + \mathds{1}_{\R \setminus (-1, 1)}(x) g(x), \quad \mbox{ with } P_1(x, z) := 2 \sqrt{\frac{1 - x^2}{z^2 - 1}} \frac{1}{|x - z|},
$$
and
$$
u_2(x) := \mathds{1}_{(-1, 1)}(x) \int_{-1}^1 \frac{G_1(x, z)}{2 \pi} \, h(z) \, dz, \quad \mbox{ with } G_1(x, z) := 2 \log \left( \frac{1 - x z + \sqrt{(1 - x^2)(1 - z^2)}}{|x - z|} \right),
$$
for all~$x \in \R$. We establish separately the~$C^{\frac{1}{2}}$ estimates for~$u_1$ and~$u_2$.

We begin by dealing with~$u_1$. First, we point out that
\begin{equation} \label{boundarydecayest}
|u_1(x) - g(1)| + |u_1(- x) - g(-1)| \le C \| g \|_{C^{\alpha}(\R \setminus J)} \sqrt{1 - x^2}, \quad \mbox{ for all } x \in [0, 1).
\end{equation}
In order to see this, recall that, by, e.g.,~\cite[Lemma~A.5]{B16},
\begin{equation} \label{Phasmass1}
\int_{\R \setminus (-1, 1)} P_1(x, z) \, dz = 2 \pi, \quad \mbox{ for all } x \in (-1, 1).
\end{equation}
Hence, given~$x \in [0, 1)$, we have
\begin{align*}
|u_1(x) - g(1)| & \le \int_{\R \setminus (-1, 1)} \frac{P_1(x, z)}{2 \pi} \left| g(z) - g(1) \right| \, dz \\
& \le \frac{\sqrt{1 - x^2}}{\pi} \left\{ [g]_{C^\alpha([1, 2])} \int_1^2 \frac{(z - 1)^{\alpha - \frac{1}{2}}}{(z - x)} \, dz + 2 \| g \|_{L^\infty(\R \setminus (-1, 2))} \int_{\R \setminus (-1, 2)} \frac{dz}{\sqrt{z - 1} |x - z|} \right\} \\
& \le C \| g \|_{C^{\alpha}(\R \setminus J)} \sqrt{1 - x^2} \left\{ \int_1^2 (z - 1)^{\alpha - \frac{3}{2}} \, dz + \int_{\R \setminus (-1, 2)} \frac{dz}{|z|^{\frac{3}{2}}} \right\} \le C \| g \|_{C^{\alpha}(\R \setminus J)} \sqrt{1 - x^2},
\end{align*}
thanks to the fact that~$\alpha > \frac{1}{2}$. Since an analogous computation can be carried out for~$|u_1(- x) - g(-1)|$, we conclude that~\eqref{boundarydecayest} holds true. From its definition, it is immediate to see that~$u_1$ is smooth inside~$(-1, 1)$. By differentiating under the integral sign, we find that its derivative satisfies
$$
u_1'(x) = \int_{\R \setminus (-1, 1)} \frac{\partial_x P_1(x, z)}{2 \pi} \, g(z) \, dz, \quad \mbox{ for all } x \in (-1, 1).
$$
Now, by differentiating identity~\eqref{Phasmass1}, we find that~$\partial_x P(x, \cdot)$ has vanishing integral over~$\R \setminus (-1, 1)$, for all~$x \in (-1, 1)$. Therefore, for all~$x \in [0, 1)$,
\begin{align*}
|u_1'(x)| & = \left| \int_{\R \setminus (-1, 1)} \! \frac{\partial_x P_1(x, z)}{2 \pi} \big( {g(z) - g(1)} \big) \, dz \right| \le \frac{1}{\pi \sqrt{1 - x^2}} \int_{\R \setminus (-1, 1)} \frac{\left| 1 - x z \right|}{\sqrt{z^2 - 1} (x - z)^2} \, |g(z) - g(1)| \, dz \\
& \le \frac{C \| g \|_{C^\alpha(\R \setminus J)}}{\sqrt{1 - x^2}} \left\{ \int_1^2 \frac{\left| 1 - x z \right| (z - 1)^{\alpha - \frac{1}{2}}}{(z - x)^2} \, dz + \int_{\R \setminus (-1, 2)} \frac{\left| 1 - x z \right|}{\sqrt{z^2 - 1} (z - x)^2} \, dz \right\} \\
& \le \frac{C \| g \|_{C^\alpha(\R \setminus J)}}{\sqrt{1 - x^2}} \left\{ \int_1^2 \frac{1 - x + x(z - 1)}{(z - x)^{\frac{5}{2} - \alpha}} \, dz + \int_{\R \setminus (-1, 2)} \frac{dz}{|z|^{\frac{3}{2}}} \right\} \le \frac{C \| g \|_{C^\alpha(\R \setminus J)}}{\sqrt{1 - x^2}}
\end{align*}
Since the same estimate also holds for~$x \in (-1, 0]$, we conclude that~$u_1$ belongs to~$C^{\frac{1}{2}}([-1, 1])$ and satisfies~$\| u_1 \|_{C^{\frac{1}{2}}([-1, 1])} \le C \| g \|_{C^\alpha(\R \setminus J)}$.

We now deal with~$u_2$. We claim that~$u_2 \in C^{\frac{1}{2}}(\R)$ with
\begin{equation} \label{u2half-cont}
|u_2(x) - u_2(y)| \le C \| h \|_{L^\infty(J)} \sqrt{|x - y|}, \quad \mbox{ for all } x, y \in \R.
\end{equation}
Since~$u_2 = 0$ outside of~$(-1, 1)$, it is enough to establish~\eqref{u2half-cont} when both~$x$ and~$y$ lie inside~$[-1, 1]$. It is fairly easy to see that~$|u_2(x)| \le C \| h \|_{L^\infty(J)} \sqrt{1 - x^2}$ for all~$x \in (-1, 1)$--e.g., by recalling that~$u_2$ satisfies~$(-\Delta)^{\frac{1}{2}} u_2 = h$ in~$(-1, 1)$ and comparing it with the supersolution~$\overline{w}_2(x) := \|h\|_{L^\infty(J)} \sqrt{(1 - x^2)_+}$. Thus, we can assume that~$x, y \in (-1 ,1)$. From the definition of~$u_2$ it is rather straightforward to derive interior H\"older estimates of any order~$\beta \in (0, 1)$. Hence, we can suppose that~$x, y \in (-1, 1) \setminus [-1 + \delta, 1 - \delta]$ for a fixed, arbitrary~$\delta \in \left( 0, \frac{1}{2} \right]$. Without loss of generality, we are then reduced to consider the case of~$x$ and~$y$ satisfying~$1 - \delta < x < y < 1$. We have
$$
|u_2(x) - u_2(y)| \le \frac{\| h \|_{L^\infty(J)}}{2 \pi} \int_{-1}^1 \big| {G_1(x, z) - G_1(y, z)} \big| \, dz \le \frac{\| h \|_{L^\infty(J)}}{\pi} \Big( {I_1(x, y) + I_2(x, y)} \Big),
$$
where
$$
I_1(x, y) := \int_{-1}^1 \left| \log \left( \frac{|y - z|}{|x - z|} \right) \right| dz
\quad \mbox{and} \quad
I_2(x, y) := \int_{-1}^1 \left| \log \left( \frac{1 - x z + \sqrt{(1 - x^2)(1 - z^2)}}{1 - y z + \sqrt{(1 - y^2)(1 - z^2)}}  \right)\right| dz.
$$
On the one hand, changing variables appropriately, we simply compute
\begin{align*}
I_1(x, y) & = \int_{-1}^x \log \left( \frac{y - z}{x - z} \right) dz + \int_x^{\frac{x + y}{2}} \log \left( \frac{y - z}{z - x} \right) dz + \int_{\frac{x + y}{2}}^y \log \left( \frac{z - x}{y - z} \right) dz + \int_y^1 \log \left( \frac{z - x}{z - y} \right) dz \\
& = (y - x) \left\{ \int_{1 + \frac{y - x}{x + 1}}^{+\infty} \frac{\log t}{(t - 1)^2} \, dt + 2 \int_1^{+\infty} \frac{\log t}{(t + 1)^2} \, dt + \int_{1 + \frac{y - x}{1 - y}}^{+\infty} \frac{\log t}{(t - 1)^2} \, dt \right\} \\
& \le C (y - x) \left| \log (y - x) \right|.
\end{align*}
On the other hand, the change of coordinates given by~$t = \frac{z}{\sqrt{1 - z^2}}$ yields that
$$
I_2(x, y) = \int_{-\infty}^{+\infty} \left| \log \left( 1 + (y - x) \frac{t + b(x, y)}{\sqrt{1 + t^2} - y t + \sqrt{1 - y^2}} \right) \right| \frac{dt}{(1 + t^2)^{\frac{3}{2}}},
$$
with~$b = b(x, y) := \frac{\sqrt{1 - \vphantom{y^2} x^2} - \sqrt{1 - y^2}}{y - x} = \frac{y + x}{\sqrt{1 - \vphantom{y^2}x^2} + \sqrt{1 - y^2}}$. Note that, by taking~$\delta$ sufficiently small, we have that~$b(x, y) \ge 2$ and~$(y - x) \frac{t + b(x, y)}{\sqrt{1 + t^2} - y t + \sqrt{1 - y^2}} \ge - \frac{1}{2}$ for all~$t \in \R$. Therefore, as~$|{\log (1 + w)}| \le (\log 4) |w|$ for all~$w \ge -\frac{1}{2}$, we get that
\begin{align*}
& \int_{-\infty}^{b(x, y)} \left| \log \left( 1 + (y - x) \frac{t + b(x, y)}{\sqrt{1 + t^2} - y t + \sqrt{1 - y^2}} \right) \right| \frac{dt}{(1 + t^2)^{\frac{3}{2}}} \\
& \hspace{20pt} \le C (y - x) \int_{-\infty}^{b} \frac{|t + b|}{\sqrt{1 + t^2} - y t + \sqrt{1 - y^2}} \frac{dt}{(1 + t^2)^{\frac{3}{2}}} \\
& \hspace{20pt} = C (y - x) \left\{ \int_{-\infty}^{- b} \frac{- b - t}{\sqrt{1 + t^2} - y t + \sqrt{1 - y^2}} \frac{dt}{(1 + t^2)^{\frac{3}{2}}} + \int_{-b}^{b} \frac{t + b}{\sqrt{1 + t^2} - y t + \sqrt{1 - y^2}} \frac{dt}{(1 + t^2)^{\frac{3}{2}}} \right\} \\
& \hspace{20pt} \le C (y - x) \left\{ \int_{-\infty}^{- b} \frac{dt}{|t|^3} + b \left[ \int_{-b(x, y)}^{\frac{1}{2}} \frac{dt}{(1 + t^2)^{\frac{3}{2}}} + \int_{\frac{1}{2}}^{b} \frac{dt}{t^3 \big( {\sqrt{1 + t^2} - t} \big)} \right] \right\} \\
& \hspace{20pt} \le C (y - x) \Big( {1 + b(x, y)} \Big) \le C \, \frac{y - x}{\sqrt{1 - x}} \le C \sqrt{y - x}.
\end{align*}
Furthermore, using that~$|{\log(1 + w)}| \le C_\beta w^\beta$ for all~$w \ge 0$ and~$\beta \in (0, 1]$, and for some~$C_\beta > 0$, we compute
\begin{align*}
& \int_{b(x, y)}^{+\infty} \left| \log \left( 1 + (y - x) \frac{t + b(x, y)}{\sqrt{1 + t^2} - y t + \sqrt{1 - y^2}} \right) \right| \frac{dt}{(1 + t^2)^{\frac{3}{2}}} \\
& \hspace{50pt} \le C (y - x)^\beta \int_{b(x, y)}^{+\infty} \left( \frac{t + b(x, y)}{\sqrt{1 + t^2} - y t + \sqrt{1 - y^2}} \right)^\beta \frac{dt}{(1 + t^2)^{\frac{3}{2}}} \\
& \hspace{50pt} \le C (y - x)^\beta \int_{b(x, y)}^{+\infty} \frac{dt}{t^{3 - \beta} \left( \sqrt{1 + t^2} - t \right)^\beta} \le C (y - x)^\beta \int_{2}^{+\infty} \frac{dt}{t^{3 - 2 \beta}} \le C (y - x)^\beta,
\end{align*}
provided we take~$\beta < 1$. By this and the previous estimate, we obtain that~$I_2(x, y) \le C \sqrt{y - x}$, which in turn leads to the desired bound~\eqref{u2half-cont} when combined with the one for~$I_1(x, y)$ which we previously established. The proof of the proposition is thus concluded.
\end{proof}

We point out that the regularity assumptions on the exterior datum~$g$ in Proposition~\ref{boundary regularity half-laplacian} are sharp, in the sense that if~$g$ is only of class~$C^{\frac{1}{2}}$ (i.e.,~$\alpha = 1/2$ in the previous statement), then the solution~$u$ is in general \emph{not}~$C^{\frac{1}{2}}$ up to the boundary, as logarithmic singularities may arise. This can be seen, for instance, by taking~$g \in C^\infty_c(\R)$ such that~$g(x) = \sqrt{x^2 - 1}$ for~$x \in [1, 2]$ and analyzing the~$\frac{1}{2}$-harmonic function in~$(- 1, 1)$ which coincides with~$g$ in~$\R \setminus (-1, 1)$, by representing it via the fractional Poisson kernel. See~\cite[Proposition~1.2]{A-RO20} for a similar observation.

Next, we recall the following interior regularity result for weak solutions, which can be easily deduced from~\cite[Theorem 1.1 and~Corollary 3.5]{Ro-Se-JDE-2016} and a cutoff argument similar to the one displayed at the beginning of the proof of Proposition~\ref{boundary regularity half-laplacian}.

\begin{proposition} \label{interior-regularity}
Let~$K \subset \R$ be a bounded open interval such that~$\overline{J} \subset K$,~$f$ be a measurable function on~$J$, and~$u \in L^1_{1/2}(\R \setminus J) \cap L^{\infty}(K)$ be a weak solution to
$$
(-\Delta)^{\frac12}u = f, \quad \textup{ in } J.
$$
Given any open interval~$I$ such that~$\overline{I} \subset J$, the following statements hold:
\begin{enumerate}[label=$(\roman*)$]
\item If~$f \in L^{\infty}(J)$, then~$u \in C^{0, 1 - \sigma}(I)$ for all~$\sigma \in (0,1)$ and there exists~$C_0 > 0$ for which
$$
\|u\|_{C^{0,1-\sigma}(I)} \leq C_0 \Big( { \|f\|_{L^{\infty}(J)} + \|u\|_{L^{\infty}(K)} + \| u \|_{L^1_{1/2}(\R \setminus J)} } \Big). 
$$
\item If~$f \in C^{\,l,\sigma}(J)$ with~$l \in \N \cup \{0\}$ and~$\sigma \in (0, 1)$, then~$u \in C^{\, l + 1, \sigma}(I)$ and there exists~$C_{l + 1} > 0$ for which
$$
\|u\|_{C^{\,l+1,\sigma}(I)} \leq C_{l + 1} \Big( { \|f\|_{C^{\,l,\sigma}(J)} + \|u\|_{L^{\infty}(K)} + \| u \|_{L^1_{1/2}(\R \setminus J)} } \Big).
$$
\end{enumerate}
The constants~$C_l$ depend only on~$\sigma$,~$I$,~$J$,~$K$, and~$l$.
\end{proposition}


We then have the following strong maximum principle for equations which possess a positive supersolution. Notice that the nonlocality of~$(-\Delta)^{\frac{1}{2}}$ allows the result to be true even in disconnected sets--such as~$J$ with~$d \ge 2$. 

\begin{lemma} \label{maxprincequivlemma}
Let~$c \in L^{\infty}(J)$ and~$u \in C(\overline{J})$ be a weak solution to
$$
\left\{
\begin{aligned}
\Lfrac u - cu & \geq 0, \quad && \textup{ in } J,\\
u & \geq 0, && \textup{ in } \R \setminus J.
\end{aligned}
\right.
$$
Suppose that there exists~$z \in C(\overline{J})$ (weakly) satisfying
$$
\left\{
\begin{aligned}
\Lfrac z - cz & \geq 0, \quad && \textup{ in } J,\\
z & > 0, \quad && \textup{ in } \overline{J},\\
z & \geq 0, && \textup{ in } \R \setminus J.
\end{aligned}
\right.
$$
Then, either~$u > 0$ in $J$ or~$u = 0$ in~$\R$.
\end{lemma}

\begin{proof}
We first prove that~$u \geq 0$ in~$\R$. Assume by contradiction that~$u^- \not \equiv 0$ and define~$v = u + \mu z$ with~$\mu$ the smallest value for which~$v  \geq 0$ in $J$,~i.e.,~$\mu := \sup_J \frac{u^-}{z} = \max_J \frac{u^-}{z} > 0$--the supremum is achieved by the continuity of~$u$ and~$z$. Since~$\mu > 0$ and~$u \ge 0$,~$z > 0$ on~$\partial J$, it follows that~$v > 0$ on~$\partial J$. Thus, there exists~$x_0 \in J$ at which~$v(x_0) = 0$. Since~$(-\Delta)^{\frac{1}{2}} v - c v \ge 0$ in~$J$, by the strong maximum principle we get that~$v \equiv 0$ in~$J$--the strong maximum principle for non-negative weak supersolutions of nonlocal equations with bounded zeroth order terms follows from the weak Harnack inequality, which can in turn be established via the methods of, e.g.,~\cite{K09,DC-K-P14}; see also~\cite[Remark~8]{P18}. Since~$v \in C(\overline{J})$, this is in contradiction with the fact that~$v > 0$ on~$\partial J$. Hence, we deduce that~$u \geq 0$ in~$\R$. A further application of the strong maximum principle yields that~$u$ is either positive in~$J$ or vanishes identically.
\end{proof}

Let~$G$ and~$H$ be respectively the Green function for the half-Laplacian in~$J$ and its regular part. The following lemma collects a couple of notable properties of these functions.

\begin{lemma} \label{regularity-Green}
For every point~$z \in J$, we have that~$H(\cdot,z) \in C^\infty(J) \cap C^{\,0,\frac12}(\R)$ and~$G(\cdot,z) \in C^\infty(J \setminus \{z\}) \cap C^{\,0,\frac12}(\R \setminus \{z\})$. Moreover, for any~$z_0 \in \partial J$,
$$
\lim_{J \ni z \rightarrow z_0} H(z, z) = -\infty.
$$
\end{lemma}
\begin{proof}
To establish the regularity of~$H$ and~$G$, we argue as in, e.g.,~\cite[Lemma 3.1.1]{A15}. The regular part~$H$ of the Green function solves the Dirichlet problem
$$
\left\{
\begin{aligned}
(-\Delta)^{\frac{1}{2}} H(\cdot, z) & = 0, \quad && \mbox{ in } J, \\
H(\cdot, z) & = - \Gamma( \cdot - z), \quad && \mbox{ in } \R \setminus J,
\end{aligned}
\right.
$$
where, as usual,~$\Gamma(y) = - 2 \log |y|$. The~$C^{\frac{1}{2}}$ boundary regularity of~$H(\cdot, z)$ is then a consequence of Proposition~\ref{boundary regularity half-laplacian}, while its smoothness inside~$J$ follows from Proposition~\ref{interior-regularity}. The regularity properties of~$G$ simply follow from the fact that~$G(x, z) = H(x, z) + \Gamma(x - z)$ for all~$x \in \R$.

To establish the (negative) blow up of~$H(z, z)$ when~$z$ approaches the boundary of~$J$, we follow the elegant argument of~\cite[Lemma~7.6]{D-LR-S17}. Up to a translation and a dilation, we may assume that~$z_0 = 1 \in \partial J$ and that~$(-1, 1) \subset \R \setminus \overline{J}$. Denote with~$H_1$ the Green function for the half-Laplacian in~$(-1, 1)$ and call~$\widetilde{G}_1$ its Kelvin transform corresponding to~$(-\Delta)^{\frac{1}{2}}$, that is~$\widetilde{G}_1(x, z) := G_1 \! \left( \frac{1}{x}, \frac{1}{z} \right)$ for every~$x \in \R \setminus \{ 0 \}$ and~$z \in \R \setminus [-1, 1]$. It is not hard to see that~$\widetilde{G}_1$ is the Green function in~$\R \setminus [-1, 1]$, i.e., that it solves
$$
\left\{
\begin{aligned}
(-\Delta)^{\frac{1}{2}} \widetilde{G}_1(\cdot, z) & = 2 \pi \delta_z, \quad && \mbox{ in } \R \setminus [-1, 1], \\
\widetilde{G}_1(\cdot, z) & = 0, \quad && \mbox{ in } [-1, 1],
\end{aligned}
\right.
$$
in the distributional sense, for every~$z \in \R \setminus [-1, 1]$. Set now~$\widetilde{H}_1(x, z) := \widetilde{G}_1(x, z) - \Gamma(x - z)$. As~$\widetilde{G}_1$ is non-negative, it immediately follows that~$\widetilde{H}_1$ satisfies
$$
\left\{
\begin{aligned}
(-\Delta)^{\frac{1}{2}} \widetilde{H}_1(\cdot, z) & = 0, \quad && \mbox{ in } J, \\
\widetilde{H}_1(\cdot, z) & \ge - \Gamma(\cdot - z), \quad && \mbox{ in } \R \setminus J.
\end{aligned}
\right.
$$
Let now~$z \in J \subset \R \setminus [-1, 1]$. Since~$H(\cdot , z)$ is also~$\frac{1}{2}$-harmonic in~$J$ and coincides with~$- \Gamma( \cdot - z)$ outside~$J$, from the maximum principle we conclude that~$H(\cdot, z) \le \widetilde{H}_1(\cdot, z)$ in~$\R$. In particular, taking advantage of the explicit expression for~$G_1$, we infer that
\begin{align*}
H(x, z) \le \widetilde{H}_1(x, z) = G_1 \! \left( \frac{1}{x}, \frac{1}{z} \right) - \Gamma(x - z) = 2 \log \left( (x z - 1) \, \sgn(x z) + \sqrt{(x^2 - 1)(z^2 - 1)} \right)
\end{align*}
for all~$x, z \in J$. Taking~$x = z$, this becomes~$H(z, z) \le 2 \log \big( {2(z^2 - 1)} \big)$, which blows up to~$-\infty$ when~$z \rightarrow z_0 = 1$. This concludes the proof.
\end{proof}

\section{Non-degeneracy of entire solutions for the $\frac{1}{2}\,$--$\,$Liouville equation} \label{App Nondeg}

\noindent In this appendix we provide a direct proof of Proposition~\ref{prop-nondegeneracy-L}. First observe that the transformation~$\varphi(x) = \phi(\eta + \mu x)$ reduces equation~\eqref{linearized-problem} to 
$$
(-\Delta)^{\frac12}\varphi - \frac{2}{1+x^2} \, \varphi = 0, \quad \textup{ in } \R.
$$
Hence, without loss of generality, we may assume $\mu = 1$ and $\eta = 0$. Then, for $J(x):= \frac{2}{1+x^2}$, we define
$$
L_1(\phi):= (-\Delta)^{\frac12} \phi-J(x) \phi, \quad  \varphi_0(x) := Z_{0,1,0}(x) = \frac{x^2-1}{x^2+1} \quad \textup{and} \quad \varphi_1(x):= Z_{1,1,0}(x) = \frac{2x}{1+x^2},
$$
and the proof of Proposition \ref{prop-nondegeneracy-L} is reduced to show the validity of the following:
\begin{proposition} \label{prop-nondegeneracy}
If $\phi \in L^{\infty}(\R)$ is a solution to $L_1(\phi) = 0$ in $\R$, then $\phi = a \varphi_0 + b \varphi_1$ with $a,b, \in \R$. 
\end{proposition}
Let us begin with some preliminary lemmas. As a first step we prove that all the bounded solutions to $L_1(\phi)= 0$ in $\R$ are smooth and belong to $\dot{H}^{\frac12}(\R) := \left\{ v \in L^1_\loc(\R) : [v]_{H^{\frac{1}{2}}(\R)} < +\infty \right\}$, where
$$
[v]_{H^{\frac{1}{2}}(\R)} := \left( \int_{\R} \int_{\R} \frac{(v(x)-v(y))^2}{(x-y)^2} \, dx dy \right)^{\! \frac{1}{2}}.
$$
More precisely, we have the following:

\begin{lemma} \label{energyidlem}
If $\phi \in L^{\infty}(\R)$ is a solution to $L_1(\phi) = 0$ in $\R$,   then $\phi \in C^\infty(\R) \cap \dot{H}^{\frac{1}{2}}(\R)$ and
\begin{equation} \label{energyid}
[\phi]_{H^{\frac{1}{2}}(\R)}^2 = 2 \pi \int_\R J(x) \phi^2(x) dx.
\end{equation}
In addition, let~$\phi^*(x) := \phi(1/x)$ for~$x \in \R \setminus \{ 0 \}$. Then,~$\phi^*$ can be extended to a function in~$C^\infty(\R)$ which satisfies~$L_1 (\phi^*) = 0$ in~$\R$ and~$[\phi^*]_{H^{\frac{1}{2}}(\R)} = [\phi]_{H^{\frac{1}{2}}(\R)}$.
\end{lemma}
\begin{proof}
Let $\phi \in L^{\infty}(\R)$ be a solution to $L_1(\phi) = 0$ in $\R$, Proposition \ref{interior-regularity} and a simple bootstrap argument yield that~$\phi$ is smooth. In particular, the equation holds in the pointwise sense.

We now check that~$\phi \in \dot{H}^{\frac{1}{2}}(\R)$ and that identity~\eqref{energyid} holds true. To this aim, for any fixed~$R \ge 2$, let~$\eta_R \in C^\infty_c(\R)$ be such that~$\eta_R = 1$ in~$(- R, R)$ and
\begin{equation} \label{etaRsmall}
[\eta_R]_{H^{\frac{1}{2}}(\R)}^2 \le \frac{1}{R}.
\end{equation}
Note that the existence of such a function is warranted by standard capacity arguments--namely, Proposition~\ref{zerocapprop} stated at the end of this appendix and the density of~$C^\infty_c(\R)$ in~$H^{\frac{1}{2}}(\R)$. Without loss of generality, we may also assume that~$0 \le \eta_R \le 1$ in $\R$. We multiply~$L_1 (\phi)$ against~$\eta_R^2 \phi$ and integrate in~$\R$. After a symmetrization, we obtain
\begin{equation} \label{identity-weak-formulation-etaR_phi}
\begin{aligned}
\int_\R J(x) \eta_R^2(x) \phi^2(x) \, dx & = \int_\R \eta_R^2(x) \phi(x) \Lfrac \phi(x) \, dx \\
& = \frac{1}{2 \pi} \int_\R \int_\R \frac{\big( \phi(x) - \phi(y) \big) \big( \eta_R^2(x) \phi(x) - \eta_R^2(y) \phi(y) \big)}{(x - y)^2} \, dx dy.
\end{aligned}
\end{equation}
On one hand, since
$$
\big( \phi(x) - \phi(y) \big) \big( \eta_R^2(x) \phi(x) - \eta_R^2(y) \phi(y) \big) = \big( \eta_R(x) \phi(x) - \eta_R(y) \phi(y) \big)^2 - \phi(x) \phi(y) \big( \eta_R(x) - \eta_R(y) \big)^2,
$$
and $\eta_{R} \equiv 1$ in $(-R,R)$, it follows from \eqref{identity-weak-formulation-etaR_phi} that
\begin{align*}
2 \pi \int_\R J(x) \eta_R^2(x) \phi^2(x) \, dx & = \int_\R \int_\R \frac{\big( \eta_R(x) \phi(x) - \eta_R(y) \phi(y) \big)^2}{(x-y)^2} \, dx dy \\
& \quad - \int_\R \int_\R \frac{\phi(x) \phi(y) \big( \eta_R(x) - \eta_R(y) \big)^2}{(x - y)^2} \, dx dy \\
& \geq \int_{-R}^{R} \int_{-R}^{R} \frac{(\phi(x)-\phi(y))^2}{(x-y)^2} \, dx dy - \|\phi\|_{L^{\infty}(\R)}^2 [\eta_R]_{H^{\frac12}(\R)}^2.
\end{align*}
Letting~$R \to \infty$, applying Fatou's Lemma and using~\eqref{etaRsmall}, we easily deduce that
\begin{equation} \label{energy-one-side}
2 \pi \int_\R J(x) \phi^2(x) \, dx \geq [\phi]_{H^{\frac12}(\R)}^2,
\end{equation}
and so, that $\phi \in \dot{H}^{\frac12}(\R)$. On the other hand, having at hand that $\phi \in  \dot{H}^{\frac12}(\R)$, we use again \eqref{identity-weak-formulation-etaR_phi} and we obtain 
$$
2 \pi \int_\R J(x) \eta_R^2(x) \phi^2(x) \, dx \leq \int_\R \eta_R^2(x) \phi(x) \Lfrac \phi(x) \, dx  \leq \int_\R \phi(x) \Lfrac \phi(x) \, dx  = [\phi]_{H^{\frac12}(\R)}^2.
$$
Applying again Fatou's Lemma with~$R \to \infty$ and taking into account~\eqref{energy-one-side}, we obtain~\eqref{energyid}.

Finally, we consider the function~$\phi^*(x) := \phi(1/x)$ for~$x \ne 0$. Clearly,~$\phi^* \in L^\infty(\R) \cap C^\infty(\R \setminus \{0\})$. Moreover, changing variables one sees that~$[\phi^*]_{H^{\frac{1}{2}}(\R)} = [\phi]_{H^{\frac{1}{2}}(\R)}$ and that~$L_1 (\phi^*) = 0$ in~$\R \setminus \{ 0\}$. Let now~$\psi \in C^\infty_c(\R)$. For~$\varepsilon > 0$, we consider a function~$\eta_\varepsilon \in C^\infty_c(\R)$ such that~$0 \le \eta_\varepsilon \le 1$ in~$\R$,~$\eta_\varepsilon = 1$ in a neighborhood of~$0$, and~$\| \eta_\varepsilon \|_{H^{\frac{1}{2}}(\R)} \le \varepsilon$. Again, this function exists thanks to Proposition~\ref{zerocapprop}. We have
\begin{align*}
\langle \phi^*, \psi \rangle_{\dot{H}^{\frac{1}{2}}(\R)} & = \langle \phi^*, (1 - \eta_\varepsilon) \psi \rangle_{\dot{H}^{\frac{1}{2}}(\R)} + \langle \phi^*, \eta_\varepsilon \psi \rangle_{\dot{H}^{\frac{1}{2}}(\R)} \\
& = 2 \pi \langle J \phi^*, (1 - \eta_\varepsilon) \psi \rangle_{L^2(\R)}  + \langle \phi^*, \eta_\varepsilon \psi \rangle_{\dot{H}^{\frac{1}{2}}(\R)} \\
& = 2 \pi \langle J \phi^*, \psi \rangle_{L^2(\R)} - 2 \pi \langle J \phi^*, \eta_\varepsilon \psi \rangle_{L^2(\R)}  + \langle \phi^*, \eta_\varepsilon \psi \rangle_{\dot{H}^{\frac{1}{2}}(\R)},
\end{align*}
where the second identity follows since~$(1 - \eta_\varepsilon) \psi$ is supported away from the origin. Note that, by H\"older's inequality,
$$
\left| \langle J \phi^*, \eta_\varepsilon \psi \rangle_{L^2(\R)} \right| \le \| \phi^* \|_{L^\infty(\R)} \| \psi \|_{L^\infty(\R)} \| J \|_{L^2(\R)} \| \eta_\varepsilon  \|_{L^2(\R)} \le C \varepsilon
$$
and
\begin{align*}
\left| \langle \phi^*, \eta_\varepsilon \psi \rangle_{\dot{H}^{\frac{1}{2}}(\R)} \right| &  = \left| \int_\R \int_\R  \frac{\big(\phi^*(x)-\phi^*(y)\big)\big( \eta_{\varepsilon}(x) \psi(x) - \eta_{\varepsilon}(y)\psi(x) + \eta_{\varepsilon}(y) \psi(x) - \eta_{\varepsilon}(y) \psi(y) \big) }{(x-y)^2}\, dxdy \right| \\
 & \le \| \psi \|_{L^\infty(\R)} [\phi^*]_{H^{\frac{1}{2}}(\R)} [\eta_\varepsilon]_{H^{\frac{1}{2}}(\R)} + \int_\R |\eta_\varepsilon(x)| \left( \int_\R \frac{|\phi^*(x) - \phi^*(y)| |\psi(x) - \psi(y)|}{(x - y)^2} \, dy \right) dx \\
& \le C \varepsilon + \| \eta_\varepsilon \|_{L^2(\R)} [\phi^*]_{H^{\frac{1}{2}}(\R)} \sup_{x \in \R} \left( \int_\R \frac{|\psi(x) - \psi(y)|^2}{(x - y)^2} \, dy \right)^{\! \frac{1}{2}} \le C \varepsilon.
\end{align*}
for some costant~$C > 0$ independent of~$\varepsilon$. As~$\varepsilon$ can be taken arbitrarily small, we conclude that~$L_1 (\phi^*) = 0$ in~$\R$ in the weak sense. But, as~$\phi^* \in L^\infty(\R)$, it follows that~$\phi^*$ is smooth in the whole of~$\R$ and that it satisfies the equation in the pointwise sense.
\end{proof}

\begin{remark}
Having at hand Lemma~\ref{energyidlem}, Proposition~\ref{prop-nondegeneracy} can be immediately deduced from~\cite[Theorem 1.4]{S19}. However, for the benefit of the reader, we provide a detailed and (almost) self-contained proof.
\end{remark} 
 
\medbreak
Let us consider the stereographic projection~$\pi: \S^1 \setminus \{ (0, 1) \} \to \R$ defined by~$\pi(z) = z_1 / (1 - z_2)$ and the lifting~$\widetilde{\phi}$ of~$\phi: \R \to \R$ to~$\S^1$ given by~$\widetilde{\phi} := \phi \circ \pi$. As the next result shows, when~$\phi$ is a bounded solution to~$L_1 (\phi) = 0$ in $\R$, its lifting~$\widetilde{\phi}$ satisfies an equation in $\S^1$ driven by the integro-differential operator
\begin{equation} \label{fract-Lapl-sphere}
(-\Delta_{\S^1})^{\frac{1}{2}} \widetilde{\phi}(z) := \frac{1}{\pi} \, \PV \int_{\S^1} \frac{\widetilde{\phi}(z) - \widetilde{\phi}(w)}{|z - w|^2} \, d\Haus^1(w), \quad \mbox{for } z \in \S^1.
\end{equation}

\begin{lemma} \label{lemma-equation-sphere}
If $\phi \in L^{\infty}(\R)$ is a solution of~$L_1 (\phi) = 0$ in~$\R$,  then $\widetilde{\phi}$ can be extended to a~$C^\infty(\S^1)$ function satisfying
\begin{equation} \label{eqfortildephi}
(-\Delta_{\S^1})^{\frac{1}{2}} \, \widetilde{\phi} - \widetilde{\phi} = 0 \quad \mbox{in} \quad \S^1.
\end{equation}
\end{lemma}
\begin{proof}
Since~$\phi \in C^\infty(\R)$, we immediately deduce from its definition that~$\widetilde{\phi}$ is smooth away from~$\{(0, 1)\}$. Moreover, letting~$\phi^*$ be defined as in Lemma~\ref{energyidlem}, it is easy to see that~$\widetilde{\phi}(z) = \phi^*(z_1 / (1 + z_2))$ for every~$z \in \S^1 \setminus \{ (0, - 1), (0, 1) \}$. As~$\phi^* \in C^\infty(\R)$, we conclude that~$\widetilde{\phi}$ can be extended to a smooth function in the whole~$\S^1$.

We now address the validity of~\eqref{eqfortildephi}. Using the parametrization~$\big[ {-\frac{3 \pi}{2}}, \frac{\pi}{2} \big) \ni \theta \mapsto e^{i \theta} \in \S^1$ for~$\S^1$, we compute
$$
(-\Delta_{\S^1})^{\frac{1}{2}} \widetilde{\phi}(e^{i \theta}) := \frac{1}{\pi} \, \PV \int_{-\frac{3\pi}{2}}^{\frac{\pi}{2}} \frac{\widetilde{\phi}(e^{i \theta}) - \widetilde{\phi}(e^{i \eta})}{|e^{i \theta} - e^{i \eta}|^2} \, d\eta = \frac{1}{\pi} \, \PV \int_{-\frac{3\pi}{2}}^{\frac{\pi}{2}} \frac{\phi \big( \frac{\cos \theta}{1 - \sin \theta} \big) - \phi \big( \frac{\cos \eta}{1 - \sin \eta} \big)}{2(1 - \cos (\theta - \eta))} \, d\eta,
$$
for every~$\theta \in \big( {-\frac{3 \pi}{2}}, \frac{\pi}{2} \big)$. Now, setting~$x = \frac{\cos \theta}{1 - \sin \theta}$ and~$y = \frac{\cos \eta}{1 - \sin \eta}$ and taking advantage of the fact that
$$
\frac{d\eta}{2(1 - \cos(\theta - \eta))} = \frac{1}{J(x)} \frac{dy}{|x - y|^2},
$$
we obtain
$$
(-\Delta_{\S^1})^{\frac{1}{2}} \widetilde{\phi}(e^{i \theta}) = \frac{1}{J(x)} \frac{1}{\pi} \, \PV \int_0^{2 \pi} \frac{\phi(x) - \phi(y)}{|x - y|^2} \, dy = \frac{1}{J(x)} \Lfrac \phi(x).
$$
From this, it follows that~$(-\Delta_{\S^1})^{\frac{1}{2}} \widetilde{\phi} - \widetilde{\phi} = 0$ in~$\S^1 \setminus \{ (0, 1)\}$. To verify that the equation holds at $\{(0,1)\}$ as well, one can argue in a similar fashion, but writing~$\widetilde{\phi}$ in terms of~$\phi^*$ and using that~$\phi^*$ also satisfies~$L_1 (\phi^*) = 0$ in~$\R$, by Lemma~\ref{energyidlem}.
\end{proof}

Next, given $\widetilde{\phi} \in L^1(\S^1)$, let us define 
$$
\widehat{\phi}(n) = \frac{1}{2\pi} \int_0^{2\pi} \widetilde{\phi}(e^{it}) e^{-int} dt, \quad \mbox{for } n \in \Z,
$$
and point out that, for $\widetilde{\phi}$ smooth, the integro-differential operator given in \eqref{fract-Lapl-sphere} can be equivalently defined as
\begin{equation} \label{half-laplac-series}
(-\Delta_{\S^1})^{\frac12} \widetilde{\phi}(e^{i\theta}) = \sum_{n \in \Z} |n| \widehat{\phi}(n) e^{in\theta}, \quad \textup{ for } \theta \in [0,2\pi).
\end{equation}
See for instance to \cite[Appendix A]{DL-M-R15} for a proof of this equivalence. Having at hand \eqref{half-laplac-series}, in the next lemma we characterize the form of the solutions to \eqref{eqfortildephi}.

\begin{lemma} \label{lemma-eigenvalues-sphere}
If $\widetilde{\phi} \in C^{\infty}(\S^1)$ is a solution to \eqref{eqfortildephi}, then $\widetilde{\phi}(e^{i\theta}) = a \cos(\theta) + b \sin(\theta)$ for some $a, b \in \R$. 
\end{lemma}

\begin{proof}
Let $\widetilde{\phi} \in C^{\infty}(\S^1)$. We Fourier expand $\widetilde{\phi}$ and get
$$
\widetilde{\phi}(e^{i\theta}) = \sum_{n \in \Z} \widehat{\phi}(n) e^{in\theta}, \quad \mbox{for all } \theta \in [0,2\pi).
$$
Thus, taking into account the definition given in \eqref{half-laplac-series}, if $\widetilde{\phi}$ is a solution to \eqref{eqfortildephi}, it follows that
$$
\sum_{n \in \Z} (|n|-1) \widehat{\phi}(n) e^{in\theta} = 0, \quad \mbox{for all } \theta \in [0,2\pi).
$$
We immediately deduce that $\widehat{\phi}(n) = 0$ for all $n \in \Z \setminus\{-1,1\}$ and so, the result follows.
\end{proof}

\begin{proof}[Proof of Proposition \ref{prop-nondegeneracy}]
First , let us consider the lifting of $\varphi_0$ and $\varphi_1$ to $\S^1$, i.e. let us define $\widetilde{\varphi}_i$ as $\widetilde{\varphi}_i = \varphi_i \circ \pi$ for $i = 0,1$, where we recall $\pi: \S^1 \setminus \{ (0, 1) \} \to \R$ is the stereographic projection. One can directly check that
$$
\widetilde{\varphi}_0(e^{i\theta}) = \sin(\theta) \quad \textup{ and } \quad  \widetilde{\varphi}_1(e^{i\theta}) = \cos(\theta), \quad \mbox{for all } \theta \in \Big(-\frac{3\pi}{2}, \frac{\pi}{2} \Big).
$$
Now, let~$\phi \in L^{\infty}(\R)$ be a solution to $L_1(\phi) = 0$ in $\R$. By Lemma~\ref{lemma-equation-sphere}, we know that its lifting~$\widetilde{\phi}$ can be extended to a~$C^{\infty}(\S^1)$ function (still denoted by~$\widetilde{\phi}$) satisfying~\eqref{eqfortildephi}. Then, by Lemma~\ref{lemma-eigenvalues-sphere}, we have  
$$
\widetilde{\phi}(e^{i\theta}) = a \widetilde{\varphi}_0(e^{i\theta}) + b \widetilde{\varphi}_1(e^{i\theta}), \quad \mbox{for all } \theta \in \Big[-\frac{3\pi}{2}, \frac{\pi}{2} \Big),
$$
for some $a,b \in \R$. Pulling back this identity, we conclude that 
$\phi = a \varphi_0 + b \varphi_1$ in $\R$ with $a,b \in \R$.
\end{proof}

We conclude this appendix with a result about fractional capacities, mentioned in the proof of Lemma~\ref{energyidlem}. Let~$E$ be a subset of~$\R$. We define the following two notions of~$\frac{1}{2}$-capacity of~$E$:
\begin{align*}
\CAP_{\frac{1}{2}}(E) & := \inf \left\{ [v]^2_{H^{\frac{1}{2}}(\R)} : v \in H^{\frac{1}{2}}(\R) \mbox{ and } v = 1 \mbox{ a.e.~in~an open neighborhood of~} E  \right\}, \\
\overline{\CAP}_{\frac{1}{2}}(E) & := \inf \left\{ \| v \|^2_{H^{\frac{1}{2}}(\R)} : v \in H^{\frac{1}{2}}(\R) \mbox{ and } v = 1 \mbox{ a.e.~in~an open neighborhood of~} E  \right\}.
\end{align*}
Recall that~$H^{\frac{1}{2}}(\R) := L^2(\R) \cap \dot{H}^{\frac{1}{2}}(\R)$.

\begin{proposition} \label{zerocapprop}
It holds~$\CAP_{\frac{1}{2}} \! \big( (x_0 - r,\, x_0 + r) \big) = 0$ for every~$x_0 \in \R$ and~$r > 0$. Furthermore,~$\CAP_{\frac{1}{2}} \! \big( \{ x_0 \} \big) = \overline{\CAP}_{\frac{1}{2}} \! \big( \{ x_0 \} \big) = 0$ for every~$x_0 \in \R$.
\end{proposition}
\begin{proof}
Up to a translation, we may take~$x_0 = 0$. Moreover, by scaling we have that~$\CAP_{\frac{1}{2}} \! \big( (- r, r) \big) = \CAP_{\frac{1}{2}} \! \big( (-1/2, 1/2) \big)$ for every~$r > 0$. Therefore, we may also restrict to the case~$r = 1/2$.

For $R \ge 3$ we consider the following truncated rescalings of the fundamental solution of~$(-\Delta)^{\frac{1}{2}}$:
$$
v_R(x) := \begin{dcases}
1 & \quad \mbox{if } x \in [-1, 1], \\
\frac{\log R - \log |x|}{\log R} & \quad \mbox{if } x \in [-R, R] \setminus [-1, 1], \\
0 & \quad \mbox{if } x \in \R \setminus [-R, R].
\end{dcases}
$$
Then, we compute
\begin{align*}
[v_R]^2_{H^{\frac{1}{2}}(\R)} & = 2 \int_{-1}^1 \int_{\R \setminus [-R, R]} \frac{dx dy}{(x - y)^2} + \frac{2}{\log^2 \! R} \int_{-1}^1 \left( \int_{[-R, R] \setminus [-1, 1]} \frac{\log^2 |y|}{(x - y)^2} \, dy \right) dx \\
& \quad + \frac{2}{\log^2 \!R} \int_{\R \setminus [-R, R]} \left( \int_{[-R, R] \setminus [-1, 1]} \frac{\big( \log R - \log |y| \big)^2}{(x - y)^2} \, dy \right) dx \\
& \quad + \frac{1}{\log^2 \! R} \int_{[-R, R] \setminus [-1, 1]} \int_{[-R, R] \setminus [-1, 1]} \frac{\big| \log |x| - \log |y| \, \big|^2}{(x - y)^2} \, dx dy.
\end{align*}
First, it is easy to see that
$$
\int_{-1}^1 \int_{\R \setminus [-R, R]} \frac{dx dy}{(x - y)^2} \le 4 \int_0^1 \left( \int_{R}^{+\infty} \frac{dy}{(y - x)^2} \right) dx = 4 \log \left( 1 + \frac{1}{R - 1} \right) \le \frac{C}{R},
$$
for some constant~$C > 0$ independent of~$R$. Secondly,
$$
\int_{-1}^1 \left( \int_{[-R, R] \setminus [-1, 1]} \frac{\log^2 |y|}{(x - y)^2} \, dy \right) dx \le 4\int_{1}^R \left(\int_0^1 \frac{dx}{(y - x)^2} \right) \log^2 y \, dy  = 4 \int_{1}^R \frac{\log^2 y}{(y - 1) y} \, dy \le C.
$$
Next, changing variables appropriately, we estimate
\begin{align*}
\int_{\R \setminus [-R, R]} \left( \int_{[-R, R] \setminus [-1, 1]} \frac{\big( \log R - \log |y| \big)^2}{(x - y)^2} \, dy \right) dx & \le 4 \int_{1}^R \left( \int_{R}^{+\infty} \frac{dx}{(x - y)^2} \right) \big( \log R - \log y \big)^2 \, dy \\
& = 4 \int_{1/R}^1 \frac{\log^2 z}{1 - z} \, dz \le C.
\end{align*}
Finally, we have
\begin{align*}
\int_{[-R, R] \setminus [-1, 1]} \int_{[-R, R] \setminus [-1, 1]} \frac{\big| \log |x| - \log |y| \, \big|^2}{(x - y)^2} \, dx dy & \le 8 \int_1^R \left( \int_{1}^x \frac{\big( \log x - \log y \big)^2}{(x - y)^2} \, dy \right) dx \\
& = 8 \int_1^R \left( \int_{1/x}^1 \frac{\log^2 z}{(1 - z)^2} \, dz \right) \frac{dx}{x} \le C \log R.
\end{align*}
Combining the last four inequalities, we see that~$[v_R]^2_{H^{\frac{1}{2}}(\R)} \le \frac{C}{\log R}$. As~$R$ can be taken arbitrarily large, we conclude that~$\CAP_{\frac{1}{2}} \! \big( (-1/2, 1/2) \big) = 0$.
 
On the other hand, the statement concerning the~$\frac{1}{2}$--capacity of points also immediately follows using $v_R$. Indeed, observe that
\begin{align*}
\overline{\CAP}_{\frac{1}{2}} \big( \{ 0 \} \big) & \le \| v_R(R \, \cdot \,) \|^2_{L^2(\R)} + [ v_R(R \, \cdot \,) ]_{H^{\frac{1}{2}}(\R)}^2 = \frac{1}{R} \, \| v_R \|^2_{L^2(\R)} + [ v_R ]_{H^{\frac{1}{2}}(\R)}^2 \\
& \le \frac{2}{R} \left\{ 1 + \frac{1}{\log^2 R} \int_1^R \log^2 \! \left( \frac{R}{x} \right) dx \right\} + \frac{C}{\log R} \le \frac{2}{\log^2 \! R} \int_{1/R}^1 \log^2 z \, dz + \frac{C}{\log R} \le \frac{C}{\log R},
\end{align*}
for every large~$R$.
\end{proof}

\section{Dependence on the $\xi_j$'s} \label{App dependence on the xijs}

\noindent This appendix is devoted to the proofs of Propositions~\ref{prop1derivative} and~\ref{prop2derivative}.

\begin{proof}[Proof of Proposition \ref{prop1derivative}]
The proof of the continuity of~$\phi$ with respect to~$\eta_l$ is similar to that of its differentiability and simpler. We therefore omit it and address directly the differentiability. To this aim, we start by doing a formal analysis. Let~$\eta$ be an interior point of~$\epsilon^{-1} \mathcal{I}_{\delta_0}$ and write~$g = g_\eta \in L^{\infty}_{\star,\sigma}(I_\epsilon)$. Let~$\phi =  L_\eta^{-1}(g)$ be the solution of~\eqref{mainprojlinearprob}-\eqref{phiorthtoZ1j}--whose existence and uniqueness is guaranteed by Proposition~\ref{mainlinearprop}. For~$l \in \{1,\ldots,m\}$, \emph{formally}, $\varphi := \partial_{\eta_l} \phi$ should satisfy
$$
\left\{
\begin{aligned}
L_\eta \varphi & = \phi \partial_{\eta_l} W + \sum_{j = 1}^m c_j \partial_{\eta_l} (\chi_j Z_{1j}) + \partial_{\eta_l} g + \sum_{j=1}^m d_j \chi_j Z_{1j}, \quad && \textup{ in } I_{\epsilon}, \\
\varphi & = 0, && \textup{ in } \R \setminus I_{\epsilon},
\end{aligned}
\right.
$$
where, still \emph{formally}, we have set~$d_j:= \partial_{\eta_l} c_j$ for all~$j = 1,\ldots, m$. We also stress that the orthogonality conditions~\eqref{phiorthtoZ1j} become
$$
\int_{\R} \varphi \chi_j Z_{1j} = -  \int_{\R} \phi \partial_{\eta_l}(\chi_j Z_{1j}), \quad \textup{ for all } j = 1, \ldots, m.
$$ 
Then, we define
\[ \widetilde{\varphi} := \varphi + \sum_{k=1}^m b_k \tilde{z}_{1k},
\]
with~$\tilde{z}_{1k}$ as in the proof of Proposition~\ref{mainlinearprop} and
\begin{equation} \label{bkderivative}
b_k :=  \displaystyle \dfrac{ \displaystyle \int_{\R} \phi \partial_{\eta_l}(\chi_k Z_{1k})}{ \displaystyle \int_{\R} \chi_k Z_{1k}^2}, \quad \textup{ for all } k = 1,\ldots,m. 
\end{equation}
Arguing as in the proof of Proposition~\ref{mainlinearprop}, one can easily see that
\begin{equation} \label{orthogonalityDerivative}
\int_{\R} \widetilde{\varphi} \chi_j Z_{1j} = 0, \quad \textup{ for all } j = 1, \ldots, m.
\end{equation}
Hence, we get that, \emph{formally},~$\widetilde{\varphi}$ is a bounded solution to 
$$
\left\{
\begin{aligned}
L_\eta \widetilde{\varphi} & = f + \sum_{j=1}^m d_j \chi_j Z_{1j}, \quad && \textup{ in } I_{\epsilon}, \\
\widetilde{\varphi} & = 0, && \textup{ in } \R \setminus I_{\epsilon},
\end{aligned}
\right.
$$
satisfying the orthogonality conditions~\eqref{orthogonalityDerivative}, where
\begin{equation}
\begin{aligned} \label{fDerivative}
f := \partial_{\eta_l} g + \sum_{i = 1}^3 f_i, \qquad \mbox{ with } \,\, f_1 := \phi \partial_{\eta_l} W, \,\,\, f_2 := \sum_{j = 1}^m c_j \partial_{\eta_l} (\chi_j Z_{1j}), \,\, \textup{ and } \,\, f_3 := \sum_{j = 1}^m b_j L_\eta \tilde{z}_{1j}\,.
\end{aligned}
\end{equation}
Proposition~\ref{mainlinearprop} would then imply that
\begin{equation} \label{estimateVarphitilde}
\|\widetilde{\varphi}\|_{L^{\infty}(\R)} \leq C |{\log \varepsilon}| \|f\|_{\star,\sigma} \le C |{\log \varepsilon}| \left( \big\| {\partial_{\eta_l} g} \big\|_{\star,\sigma} + \sum_{i = 1}^3 \| f_i \|_{\star,\sigma} \right).
\end{equation}
We estimate separately the~$\|\cdot\|_{\star,\sigma}$--norm of each~$f_i$. First of all, note that
\begin{equation} \label{355}
\partial_{\eta_l} (\chi_j Z_{1j}) = -4\epsilon\, \frac{(y-\eta_j) \chi(y-\eta_j) \mu_j \partial_{\xi_l} \mu_j}{(\mu_j^2 + (y-\eta_j)^2)^2}, \quad \textup{ for } j \neq l,
\end{equation}
and
\begin{equation} \label{356}
\partial_{\eta_l} (\chi_l Z_{1l}) = -2\,\frac{\chi(y-\eta_l) + (y-\eta_l)\chi'(y-\eta_l)}{\mu_l^2+(y-\eta_l)^2} + \frac{4 (y - \eta_l)^2}{(\mu_l^2 + (y-\eta_l)^2)^2} -4\epsilon\, \frac{(y-\eta_l) \chi(y-\eta_l) \mu_l \partial_{\xi_l} \mu_l}{(\mu_l^2 + (y-\eta_l)^2)^2}.
\end{equation}
Thus, taking into account that~$|\partial_{\xi_j} \mu_k| \le C$ for all $j,k \in \{1,\ldots,m\}$ (cf.~the proof of Lemma~\ref{derofVwrtetalem}) as well as the estimates~\eqref{mainlinearaprioriest} and~\eqref{Mjj}, we easily see that
$$
|b_j| \leq C \epsilon |{\log \epsilon}| \, \|g\|_{\star,\sigma}, \ \textup{ if } j \neq l, \quad \textup{ and }  \quad
|b_l| \leq C |{\log \epsilon}| \, \|g\|_{\star,\sigma}.
$$
Also, from Lemma~\ref{ztilde1lem} and~\eqref{Wasperturb}--\eqref{thetabounds} we deduce that
$$
|L_\eta \tilde{z}_{1j}(y)| \leq C \epsilon^{1-\sigma} \left( \epsilon + \frac{1}{(1+|y-\eta_j|)^{1+\sigma}} \right), \quad \textup{ for all } y \in  I_{\epsilon} \mbox{ and } j = 1,\ldots,m.
$$
Hence, it follows that
\[
\|f_3\|_{\star,\sigma} \leq C \epsilon^{1-\sigma} |{\log \epsilon}| \, \|g\|_{\star,\sigma}.
\]
Next, combining~\eqref{mainlinearaprioriest} and~\eqref{ckestimates}, we immediately see that
$$
|c_j| \leq \|g\|_{\star,\sigma}, \quad \textup{ for all } j = 1,\ldots,m.
$$
On the other hand, taking into account \eqref{355}-\eqref{356}, we easily check that, for all $y \in I_{\epsilon}$, 
$$
|\partial_{\eta_l}(\chi_j(y) Z_{1j}(y))| \leq \frac{C \epsilon}{(1+|y-\eta_j|)^{1+\sigma}}, \, \textup{ if } j \neq l,
$$
and
$$
|\partial_{\eta_l}(\chi_l(y) Z_{1l}(y))| \leq \frac{C}{(1+|y-\eta_l|)^{1+\sigma}}.
$$
Thus, we get that
$$
\|f_2\|_{\star,\sigma} \leq C \|g\|_{\star,\sigma}.
$$
Finally, taking into account~\eqref{Wasperturb}--\eqref{thetabounds} and Lemma~\ref{derofVwrtetalem}, we see that, for all~$y \in I_\epsilon$,
\begin{align*}
|\partial_{\eta_l} W(y)| = \bigg| \bigg(\sum_{j = 1}^m \frac{2 \mu_j}{\mu_j^2 + (y - \eta_j)^2} + \theta(y) \bigg) (Z_{1j}+R_j) \bigg| \leq C \bigg( \epsilon + \sum_{j=1}^m \frac{1}{(1+|y-\eta_j|)^{1+\sigma}} \bigg).
\end{align*}
Thus, using one more time~\eqref{mainlinearaprioriest}, we infer that
$$
\|f_1\|_{\star,\sigma} \leq C |{\log \epsilon}| \, \|g\|_{\star,\sigma}.
$$
Recalling~\eqref{estimateVarphitilde}, we conclude that
$$
\|\widetilde{\varphi}\|_{L^{\infty}(\R)} \leq C |{\log \epsilon}|^2 \left( \|g\|_{\star,\sigma} + \frac{1}{|{\log\epsilon|}} \, \big\| {\partial_{\eta_l} g} \big\|_{\star,\sigma} \right),
$$
and thus that
$$
\big\| {\partial_{\eta_{l}} \phi} \big\|_{L^{\infty}(\R)} \leq C |{\log \epsilon}|^2 \left( \|g\|_{\star,\sigma} + \frac{1}{|{\log\epsilon}|} \, \big\| {\partial_{\eta_l} g} \big\|_{\star,\sigma} \right).
$$

\medbreak The above computations are only \emph{formal}, since we do not know a priori that~$\phi$ is differentiable with respect to~$\eta_l$. To make them \emph{rigorous}, we define~$\varphi:= L_{\eta}^{-1}(f) - \sum_{k=1}^m b_k \tilde{z}_{1k}$, with~$f$ as in~\eqref{fDerivative} and the~$b_k$'s as in~\eqref{bkderivative}. We will prove that~$\phi$ is differentiable with respect to~$\eta_l$ and that~$\varphi = \partial_{\eta_l} \phi$. To this aim, we follow the argument presented at the end of~\cite[Section~4]{D-dP-W14}. Denote with~$e_l$ the~$l$-th element of the canonical basis of~$\R^m$. For~$t \in (- 1, 1) \setminus \{ 0 \}$ we define~$\eta_{l}^t = \eta + t e_l$, which belongs to~$\epsilon^{-1} \mathcal{I}_{\delta_0}$ provided~$t$ is sufficiently small. For an arbitrary function~$h = h(\eta)$, we set
$$
 \quad D_l^t h(\eta) := \frac{h(\eta_l^t) - h(\eta)}{t}.
$$
Using this notation, we define~$\varphi^t:= D_l^t \phi$ and $d_j^{\,t}:= D_l^t c_j$ for all $j = 1,\ldots,m$. Moreover, we choose
\[
b_k^t := \dfrac{\displaystyle \int_\R \phi(\eta^t_l) D_l^t (\chi_k Z_{1 k})}{\displaystyle \int_{\R} \chi_k Z_{1k}^2} = - \dfrac{ \displaystyle \int_{\R} \varphi^t \chi_k Z_{1k}}{ \displaystyle \int_{\R} \chi_k Z_{1k}^2}, \quad \textup{ for all } k = 1,\ldots,m. 
\]
and set~$\widetilde{\varphi}^{\,t} := \varphi^t + \sum_{k=1}^m b_{k}^t \tilde{z}_{1k}$. Arguing as in the proof of Proposition~\ref{mainlinearprop}, it is immediate to check that
$$
\int_{\R} \widetilde{\varphi}^{\,t} \chi_j Z_{1j} = 0, \quad \textup{ for all } j = 1,\ldots, m. 
$$
Moreover, we have that
$$
\left\{
\begin{aligned}
L_\eta \widetilde{\varphi}^{\,t} & = f^{t} + \sum_{j=1}^m d_j^{\,t} \chi_j Z_{1j}, \quad && \textup{ in } I_{\epsilon}, \\
\widetilde{\varphi} & = 0, \quad && \textup{ in } \R \setminus I_{\epsilon},
\end{aligned}
\right.
$$
with 
$$
f^t:= D_l^t g + \phi(\eta_l^t) D_l^t W + \sum_{j=1}^m c_j(\eta_l^t) D_l^t(\chi_j Z_{1j}) + \sum_{j=1}^m b_j^t L_\eta \tilde{z}_{1j}.
$$
From here, using the linearity of the problem, the continuity of~$\phi$ with respect to~$\eta_l$, and estimate \eqref{mainlinearaprioriest} of Proposition~\ref{mainlinearprop}, one checks that
$$
\lim_{t \to 0} \|\varphi^t  - \varphi\|_{L^{\infty}(\R)} = 0.
$$
This means that~$\phi$ is differentiable with respect to~$\eta_l$ with~$\varphi = \partial_{\eta_l} \phi$ a.e.~in~$\R$. The computation made in the first part of the proof is thus now fully justified and yields the bound~\eqref{boundDerivative} for~$\partial_{\eta_l} (L_\eta^{-1}(g))$. Moreover, its continuous dependence in~$\eta$ follows from the continuous dependence of the data involved in the definition of~$\varphi$. The proof is finished.
\end{proof}

\begin{proof}[Proof of Proposition \ref{prop2derivative}]
Write~$\phi = \Phi(\eta)$. Arguing as in the proof of~\cite[Proposition~5.1]{D-dP-W14}, it is not difficult to see that~$\eta \mapsto \Phi(\eta)$ is a~$C^1$ map. Since~$\phi = L_\eta^{-1} \left( - \mathscr{E} + \mathcal{N}(\phi) \right)$ and~$\eta \mapsto - \mathscr{E} + \mathcal{N}(\phi)$ is of class~$C^1$, 
Proposition~\ref{prop1derivative} gives that
\begin{equation} \label{mainDerivativeNonlinear}
\|\partial_{\eta_l} \phi\|_{L^{\infty}(\R)} \leq C |{\log \epsilon}| \Big\{ {|{\log \epsilon}| \big( {\|\mathscr{E}\|_{\star,\sigma} + \| \mathcal{N}(\phi) \|_{\star,\sigma}} \big)  + \|\partial_{\eta_l} \mathscr{E} \|_{\star,\sigma} + \|\partial_{\eta_l} \mathcal{N}(\phi)\|_{\star,\sigma}} \Big\}.
\end{equation}
We have already seen (cf. Lemma~\ref{error-measure} and~\eqref{nstarestimate}) that, for all~$\phi \in L^{\infty}(\R)$ with~$\|\phi\|_{L^{\infty}(\R)} \leq 1$, 
$$
\|\mathscr{E}\|_{\star,\sigma} \leq C \epsilon^{1-\sigma} \quad \textup{ and } 
\quad
\|\mathcal{N}(\phi)\|_{\star,\sigma} \leq C \|\phi\|_{L^{\infty}(\R)}^2.
$$
Moreover, having at hand the computations made in the proof of Lemma~\ref{derofVwrtetalem} and taking advantage of the fact that~$\|\partial_{\eta_l} W\|_{\star,\sigma} + \|W\|_{\star,\sigma} \leq C$ (cf. proof of Proposition \ref{prop1derivative}), it is not difficult to see that 
$$
\|\partial_{\eta_l} \mathscr{E}\|_{\star,\sigma} \leq C \epsilon^{1-\sigma}.
$$
Also, arguing as in the proof of Proposition~\ref{propnonlineartheory}, we get that
\begin{align*}
\|\partial_{\eta_l} \mathcal{N}(\phi)\|_{\star,\sigma} & \le C \Big(  \|\partial_{\eta_l} W\|_{\star,\sigma} \|\phi\|_{L^{\infty}(\R)}^2 + \|W\|_{\star,\sigma} \|\phi\|_{L^{\infty}(\R)} \|\partial_{\eta_l}\phi\|_{L^{\infty}(\R)} \Big) \\
& \le C \|\phi\|_{L^{\infty}(\R)} \Big( \|\phi\|_{L^{\infty}(\R)} + \|\partial_{\eta_l}\phi\|_{L^{\infty}(\R)} \Big).
\end{align*}
Substituting these estimates in~\eqref{mainDerivativeNonlinear} and using that~$\phi$ satisfies~\eqref{estfornonlinearphi}, we finally obtain that
$$
\|\partial_{\eta_l} \phi\|_{L^{\infty}(\R)} \leq C \epsilon^{1-\sigma} |{\log \epsilon}|^2,
$$
and the proof is complete. 
\end{proof}

\bibliographystyle{plain}
\bibliography{Bibliography}

\end{document}